\documentclass[11pt,reqno]{amsart}

\usepackage[margin=1.2in]{geometry}
\usepackage{amsmath,amssymb,amsthm,amscd,mathrsfs,mathtools}
\usepackage{bbm}
\usepackage{esint}
\usepackage{cases}
\usepackage{cite}
\usepackage{enumerate}
\usepackage{graphicx}
\usepackage[dvipsnames]{xcolor}
\usepackage[implicit=true]{hyperref}
\usepackage[nameinlink,noabbrev]{cleveref}
\usepackage{microtype}

\numberwithin{equation}{section}

\newtheorem{theorem}{Theorem}[section]
\newtheorem{lemma}[theorem]{Lemma}
\newtheorem{remark}[theorem]{Remark}
\newtheorem{definition}[theorem]{Definition}
\newtheorem{proposition}[theorem]{Proposition}
\newtheorem{example}[theorem]{Example}
\newtheorem{corollary}[theorem]{Corollary}

\def\e{{\mathrm{e}}}
\def\eps{\varepsilon}

\def\p{\partial}

\def\<{{\langle}}
\def\>{{\rangle}}

\newcommand{\bx}{\begin{example}}\newcommand{\ex}{\end{example}}

\def\dif{{\mathord{{\rm d}}}}

\def\min{{\mathord{{\rm min}}}}

\def\no{\nonumber}
\def\={&\!\!=\!\!&}
\def\bt{\begin{theorem}}
\def\et{\end{theorem}}
\def\bl{\begin{lemma}}
\def\el{\end{lemma}}
\def\br{\begin{remark}}
\def\er{\end{remark}}
\def\bd{\begin{definition}}
\def\ed{\end{definition}}
\def\bp{\begin{proposition}}
\def\ep{\end{proposition}}
\def\bc{\begin{corollary}}
\def\ec{\end{corollary}}
\def\cA{{\mathcal A}}

\def\cD{{\mathcal D}}
\def\cE{{\mathcal E}}
\def\cF{{\mathcal F}}
\def\cG{{\mathcal G}}
\def\cH{{\mathcal H}}
\def\cI{{\mathcal I}}

\def\cL{{\mathcal L}}

\def\cO{{\mathcal O}}

\def\cR{{\mathcal R}}
\def\cS{{\mathcal S}}
\def\cT{{\mathcal T}}

\def\mB{{\mathbb B}}

\def\mE{{\mathbb E}}

\def\mN{{\mathbb N}}

\def\mP{{\mathbb P}}

\def\mR{{\mathbb R}}

\def\mT{{\mathbb T}}

\def\mZ{{\mathbb Z}}

\def\bC{{\mathbf C}}

\def\sI{{\mathscr I}}
\def\sJ{{\mathscr J}}

\def\sL{{\mathscr L}}

\def\sN{{\mathscr N}}

\def\sW{{\mathscr W}}

\def\ba{\begin{align}}
\def\ea{\end{align}}

\def\geq{\geqslant}
\def\leq{\leqslant}

\def\loc{\mathord{{\rm loc}}}

\allowdisplaybreaks

\begin{document}
	
	\title[]{
Uniform-in-time diffusion approximations for multiscale stochastic systems}
	
	\date{}
	
	\author{Longjie Xie and Xicheng Zhang}

\address{Longjie Xie:
	School of Mathematics and Statistics, Jiangsu Normal University,
	Xuzhou, Jiangsu 221000, China\\
	Email: longjiexie@jsnu.edu.cn
}

\address{Xicheng Zhang: School of Mathematics and Statistics, Beijing Institute of Technology, Beijing 100081, China;
Faculty of Computational Mathematics and Cybernetics, Shenzhen MSU-BIT University, 518172 Shenzhen, China.\\ Email: xczhang.math@bit.edu.cn
}

\thanks{
    This work is supported by National Key R\&D program of China (No. 2023YFA1010103) and NSFC (No. 12471140,12595282).
}

	\begin{abstract}
This paper establishes a quantitative, uniform-in-time diffusion approximation for the joint law of a broad class of fully coupled multiscale stochastic systems. We derive a precise characterization of the limiting joint distribution as a specific skew-product of the conditional equilibrium of the fast process and the homogenized law of the slow component, thereby providing a rigorous uniform-in-time formulation of the adiabatic elimination principle. The convergence rate  explicitly separates the initial relaxation of the fast dynamics from the long-time homogenized evolution and depends only on the regularity of the coefficients in the slow variable. As a consequence, we obtain the first quantitative  identification  of the limiting stationary distribution of the original multiscale system and prove the commutativity of the limits $\eps\to0$ and $t\to\infty$ for a large class of observables. Our framework accommodates unbounded and irregular coefficients, degenerate structures, and weakly mixing dynamics. We illustrate its scope with three applications: {\it (i)} a uniform-in-time averaging principle for fast-slow systems; {\it (ii)} a uniform Smoluchowski--Kramers approximation for degenerate Langevin systems, yielding convergence of the joint position-scaled velocity law  and   global-in-time asymptotics of key thermodynamic functionals (e.g., total energy, entropy production, free energy); and {\it (iii)} the first uniform-in-time periodic homogenization result for SDEs with distributional drifts.

		\bigskip
		
		
\end{abstract}

	\maketitle
	
\tableofcontents

\section{Introduction}

\subsection{Background and motivation}

Multiscale stochastic differential equations (SDEs) provide a probabilistic counterpart to homogenization theory for partial differential equations (see e.g. \cite{HP,Par,PS}), and are essential for modeling  real-world dynamical systems involving separated time scales, with a wide range of applications spanning  planetary motion, molecular dynamics, climate science, neuroscience and  financial mathematics. We refer the
readers to the books \cite{FW,PS2} for a  comprehensive overview.  In this paper, we consider the following general, fully coupled multiscale system on $\mathbb{R}^d \times \mathbb{R}^\vartheta$:
\begin{equation}\label{sde0}
\left\{
\begin{aligned}
\dif X^\eps_t
&=\eps^{-2}b(X^{\eps}_t, Y^{\eps}_t)\dif t+\eps^{-1}c(X^{\eps}_t, Y^{\eps}_t)\dif t
+\sqrt{2}\,\eps^{-1}\sigma(X^{\eps}_t, Y^{\eps}_t)\dif W_t,\quad X_0^\eps=x,\\
\dif Y^{\eps}_t
&=F(X^{\eps}_t, Y^{\eps}_t)\dif t+\eps^{-1} H(X^{\eps}_t, Y^{\eps}_t)\dif t
+\sqrt{2}\,G(X^{\eps}_t, Y^{\eps}_t)\dif W_t,\qquad\quad  Y_0^\eps=y,
\end{aligned}
\right.
\end{equation}
where $W$ is an $m$-dimensional Brownian motion on $(\Omega,\mathcal{F},\mathbb{P})$,  the coefficient
\[
(b,c,\sigma): \mathbb{R}^d\times\mathbb{R}^\vartheta \to (\mathbb{R}^d, \mathbb{R}^d, \mathbb{R}^d\otimes\mathbb{R}^m),\qquad
(F,H,G): \mathbb{R}^d\times\mathbb{R}^\vartheta \to (\mathbb{R}^\vartheta, \mathbb{R}^\vartheta, \mathbb{R}^\vartheta\otimes\mathbb{R}^m)
\]
are measurable functions, and $\eps\in(0,1)$ quantifies the separation of time scales: the fast component $X^\eps$ evolves on the $\eps^{-2}$ scale, while the  variable $Y^\eps$ is driven simultaneously by its own intrinsic dynamics and the rapid fluctuations of $X^\eps$. This formulation is rather versatile. When $c \equiv H \equiv 0$, it reduces to the classical fast-slow setting.
More importantly, it encompasses the physically fundamental Langevin stochastic system, obtained   by taking $F\equiv 0$, $G\equiv 0$, $H(x,y)=x$, and
\[
b(x,y) = -A(y)x,\qquad c(x,y) = -\nabla_y U(y),\qquad \sigma(x,y) = \sigma(y),
\]
where $A: \mathbb{R}^d \to \mathbb{R}^d\otimes\mathbb{R}^d$ is a matrix-valued friction coefficient and $U: \mathbb{R}^d \to [0,\infty)$ is a potential. This leads to
\begin{equation}\label{lav}
\left\{
\begin{aligned}
\dif X^{\eps}_t
&= -\eps^{-2}A(Y_t^\eps)X^{\eps}_t\dif t-\eps^{-1}\nabla_y U(Y^{\eps}_t)\dif t
+\sqrt{2}\,\eps^{-1}\sigma(Y_t^\eps)\dif W_t,\\
\dif Y^{\eps}_t
&= \eps^{-1}X_t^\eps\dif t,
\end{aligned}
\right.
\end{equation}
which is  equivalent to the second-order stochastic system
\begin{equation}\label{ex11}
\eps^2\,\ddot Y_t^\eps = -\nabla_y U(Y_t^\eps) - A(Y_t^\eps)\dot Y_t^\eps + \sqrt{2}\,\sigma(Y_t^\eps)\dot W_t.
\end{equation}
The limiting regime $\eps \to 0$ in \eqref{ex11} is known as the celebrated Smoluchowski--Kramers (S-K) approximation; see, e.g., \cite{B18, BHVW17,BW18, CF, CG,CGX, HMVW, HVW, XY}.

\smallskip

The classical  diffusion approximation (or second-order averaging) provides an  effective dynamics for the slow
variable $Y_t^\eps$ in (\ref{sde0}) as $\eps\to0$ through an auxiliary Poisson equation involving the frozen fast dynamics. More precisely,  for any fixed
$y\in\mR^\vartheta$, let $X^y$ solve the frozen equation
\begin{equation}\label{sde2}
\dif X_t^y=b(X_t^y,y)\dif t+\sqrt{2}\,\sigma(X_t^y,y)\dif W_t,\qquad X_0^y=x\in\mR^d,
\end{equation}
and denote by $\mu_y$ its invariant measure.  Under suitable assumptions, there exists a unique (centered) solution
$\Phi(\cdot,y)$ to the Poisson/cell equation
\begin{equation}\label{pde10}
\sL_1\Phi(x,y):=(\sigma\sigma^*)(x,y):\nabla_x^2\Phi(x,y)+b(x,y)\cdot\nabla_x\Phi(x,y)=-H(x,y),
\end{equation}
where $:$ denotes matrix contraction. Define the homogenized drift and diffusion coefficients by
\begin{align}\label{bcf}
\begin{split}
\cF(y)&:=\mu_y\big((F+\Gamma_1)(\cdot,y)\big),\\
\cG(y)&:=\mu_y\big(\big(GG^*+(\Gamma_2+\Gamma_2^*)/2\big)(\cdot,y)\big),
\end{split}
\end{align}
with
\begin{align}\label{bcf0}
\begin{split}
\Gamma_1
&:=(c\cdot\nabla_x+H\cdot\nabla_y)\Phi+2\,\sigma G^*:\nabla_x\nabla_y\Phi,\\
\Gamma_2
&:=H\otimes\Phi+2\,(G\sigma^*)\cdot\nabla_x\Phi.
\end{split}
\end{align}
Here and throughout the paper, we employ the shorthand notation $\mu(f) := \int_{\mR^d} f(x)\mu(\dif x)$ for the expectation of a function $f$ with respect to a probability measure $\mu$ on $\mR^d$. Then, as $\eps\to0$, $Y_t^\eps$ converges weakly to $\bar Y_t$  which solves the homogenized SDE
\begin{equation}\label{sde000}
\dif \bar Y_t=\cF(\bar Y_t)\dif t+\sqrt{2}\,\bar\Sigma(\bar Y_t)\dif \bar W_t,
\qquad \bar\Sigma(y)\bar\Sigma(y)^*=\cG(y),
\end{equation}
where $\bar W$ is a $\vartheta$-dimensional Brownian motion. History  of the literature and related topics can be found in  \cite{BK,KY,PSV,Par,Pa-Ve1,Pa-Ve2,Ro-Xi1}. In the special case $c\equiv H\equiv0$, this reduces to
the classical averaging principle for fast-slow systems; see, e.g., \cite{Ch-L,DGHS,GR,HL,Ki,KY2,LS,FW93} and the references therein.

\smallskip

Despite the substantial literature, the available theory  on  diffusion approximation for fully coupled multiscale systems still exhibits several  notable limitations:

\begin{enumerate}[(i)]
\item {\bf Finite-time horizon:} Most convergence results are typically established only over finite time intervals $[0,T]$, leaving the crucial long-time behavior--particularly the stationary regimes--unclear. Even in the simpler averaging regime (i.e., when $c\equiv H\equiv0$ in  (\ref{sde0})), uniform-in-time averaging principles have only been studied relatively recently under strong dissipativity and regularity assumptions on the coefficients \cite{BDOZ,Cr-Do-Go-Ot-So}. Obviously, the uniform-in-time  diffusion approximation for the general system (\ref{sde0}) is  substantially more delicate, as it requires controlling the accumulation of  corrector terms and thus presents additional difficulties.

   \item {\bf Lack of joint law analysis:} Even within the averaging framework, prior work typically focuses solely on convergence of the marginal law of the slow variable $Y_t^\eps$. The limiting joint distribution of the slow and fast components, which is essential for understanding the full statistical  physics of the coupled multiscale system--including key thermodynamic functionals and the structure of invariant measures (information that cannot be recovered from the marginal law of the slow variable alone)--has remained largely uncharacterized.  In the specialized setting of  S-K approximation, recent works \cite{BW18,CG,CGX} have studied the convergence of stationary measures, but still only for the  marginal distribution of the slow component, leaving the full joint law unresolved.

\item {\bf Restrictive coefficient assumptions:} Many existing proofs rely on strong boundedness, uniform ellipticity, or high regularity  assumptions on the coefficients, thereby excluding important physical models  such as degenerate Langevin dynamics or systems with growing coefficients and, most strikingly, highly irregular or even distributional drifts.
\end{enumerate}

The first two points are particularly critical for understanding the stationary behavior of multiscale systems. As emphasized by Freidlin \cite{Frei}:``while the classical result of diffusion approximation provides a good approximation to the
slow process $Y_t^\eps$ of the multi-scale system, but  how is the stationary distribution of the original system  (\ref{sde0})  related to that of the homogenized equation (\ref{sde000})? Statement of diffusion approximation says nothing on this relation." To the best of our knowledge, this question has remained open in full generality.

\subsection{Main results and contributions}

In this paper, we develop a {\it quantitative and uniform-in-time diffusion approximation for the joint law} of $(X_t^\eps,Y_t^\eps)$ in the multiscale system \eqref{sde0} under a remarkably flexible and general framework that accommodates unbounded and irregular coefficients, degenerate structures, and  weakly mixing dynamics. Our results represent a significant advance beyond classical finite-horizon theory and provide a unified route from transient homogenization to stationary asymptotics. Our core contributions are as follows:

\begin{itemize}
\item {\bf Quantitative identification of the joint-law limit.} We prove that the joint law of $(X_t^\eps,Y_t^\eps)$ converges uniformly in time  to
   a specific skew-product $\mu_{y}(\dif x)\cL_{\bar Y_t}(\dif y)$ (rather than an independent product), see {\bf Theorem \ref{main}}. This limit rigorously captures the physical principle of adiabatic elimination: the fast component instantaneously equilibrates to its conditional distribution $\mu_{y}$ given the current state of the slow variable, which itself evolves according to the homogenized dynamics $\bar Y_t$.
  The explicit convergence rate comprises an initial-layer term describing the rapid relaxation of the fast motion and a main error term, which  depends only on the regularity of coefficients in the slow variable. As far as we know, this constitutes  the first uniform-in-time result for the joint law, even within the classical averaging regime. Moreover, this precise characterization enables the rigorous  asymptotic analysis of many important physical observables  that intrinsically depend on  the full system state--such as total energy, entropy production rate, and free energy--globally in time (see Subsection 1.3.2).

\item {\bf Quantitative stationary-measure convergence.} As a direct consequence, we obtain a quantitative answer to the question mentioned by Freidlin \cite{Frei}: the invariant measure $\nu_\eps$ of the original multiscale system converges weakly to the skew-product $\mu_y(\dif x)\bar\mu(\dif y)$, where $\bar\mu$ is the invariant measure of the homogenized equation, see {\bf Theorem \ref{main1}}. This provides a precise asymptotic description of the stationary regime for the multiscale system  and, more importantly, establishes the commutativity of the limits $\eps \to 0$ and $t \to \infty$ (see \eqref{com} below) for a broad class of observables. Consequently, it validates the use of the homogenized equation for long-time prediction and offers a firm foundation for multiscale numerical methods   (e.g., the Heterogeneous Multiscale Method \cite{ELV}) in the stationary regimes.

\item {\bf Weak regularity, mild dissipativity and degenerate structures.} Our assumptions are rather mild:  coefficients may be unbounded and only weighted H\"older continuous, and crucially, no non-degeneracy of the full multiscale system is imposed, making the theory applicable to genuinely degenerate models arising in physical applications.  Furthermore, our ergodicity hypotheses permit polynomial mixing rates. This is achieved through a novel set of hypotheses {\bf (H)} that are both mathematically flexible and physically verifiable:
\begin{itemize}
  \item {\bf (H$_0$)} basically ensures well-posedness and uniform moment bounds.
  \item {\bf (H$_1$)} guarantees ergodicity of the frozen fast dynamics with a weak mixing rate.
  \item {\bf (H$_2$)} assumes the homogenized system (\ref{sde000}) is non-degenerate and ergodic.
\end{itemize}
A key feature of {\bf (H$_2$)} is that the non-degeneracy and dissipativity of the homogenized limit  can be either inherited directly from the original slow dynamics
or transferred entirely from the corresponding properties of the fast motion via the corrector terms. This flexibility is crucial for applications where the slow dynamics alone are degenerate, see Section 6 for the degenerate Langevin system.
\end{itemize}

\smallskip
To state our main results  and for notational  convenience, we set $Z^\eps_t:=(X^\eps_t,Y^\eps_t)$ and define, for $z=(x,y)\in\mR^d\times\mR^\vartheta$,
\[
B_\eps(z):=
\begin{pmatrix}
\eps^{-2}b(x,y)\\
F(x,y)
\end{pmatrix},
\quad
D(z):=
\begin{pmatrix}
c(x,y)\\
H(x,y)
\end{pmatrix},
\quad
\Theta_\eps(z):=
\begin{pmatrix}
\eps^{-1}\sigma(x,y)\\
G(x,y)
\end{pmatrix}.
\]
Then the system \eqref{sde0} can be written compactly as
\begin{align}\label{SDE1}
\dif Z^\eps_t=\big[B_\eps+\eps^{-1}D\big](Z^\eps_t)\dif t+\sqrt{2}\,\Theta_\eps(Z^\eps_t)\dif W_t,\quad Z^\eps_0=z.
\end{align}

\paragraph{\textit{Notation: Polynomial weights and  weighted H\"older spaces}}
To handle unbounded coefficients and capture moment growth, we introduce a class of polynomial weight functions. For $N\in\mN$, define
\begin{align}\label{sW}
\sW_N
:=
\left\{\rho:\mR^N\to(0,\infty):\ \exists\, r\in\mR,\,
0<\inf_{x\in\mR^N}\frac{\rho(x)}{1+|x|^r}\leq \sup_{x\in\mR^N}\frac{\rho(x)}{1+|x|^r}<\infty\right\}.
\end{align}
In words, $\sW_N$ consists of functions that exhibit at most polynomial  growth or decay at infinity. Clearly, $\sW_N$ is closed under algebraic operations of addition, multiplication, and division: for any $\rho_1,\rho_2\in\sW_N$,
\begin{align}\label{Da49}
\rho_1+\rho_2,\quad \rho_1\cdot \rho_2,\quad \rho_1/\rho_2\in\sW_N.
\end{align}
Furthermore, each $\rho\in\sW_N$ satisfies a local translation estimate: there exists a constant $C=C(\rho)\geq1$ such that for all $x\in\mR^N$ and $y\in\mR^N$ with $|y|\leq1$,
\begin{equation}\label{weight-shift}
C^{-1}\rho(x)\leq \rho(x+y)\leq C\rho(x).
\end{equation}
Throughout the paper, weights in $\sW_{d+\vartheta}$ are denoted by $\varrho(z)$. Given $\rho\in\sW_d$ or $\omega\in\sW_\vartheta$,
we may view it as an element of $\sW_{d+\vartheta}$ via $\varrho(x,y)=\rho(x)$ or $\varrho(x,y)=\omega(y)$. Conversely, for any
$\varrho\in\sW_{d+\vartheta}$ there exist $(\rho,\omega)\in\sW_d\times\sW_\vartheta$ such that
\begin{equation*}
\varrho(x,y)\leq \rho(x)\,\omega(y),\quad (x,y)\in\mR^d\times\mR^\vartheta.
\end{equation*}
With polynomial weights at hand, we now introduce the corresponding anisotropic H\"older spaces.
Given $\alpha,\beta\geq0$ and weights $(\rho,\omega)\in\sW_d\times\sW_\vartheta$, we denote
\[
\bC^\alpha_\rho\bC^\beta_\omega:=\bC^\alpha_\rho(\mR^d;\bC^\beta_\omega(\mR^\vartheta)),
\quad
\bC^\alpha_{\rm p}\bC^\beta_{\rm p}:=\bigcup_{\rho\in\sW_d,\ \omega\in\sW_\vartheta}\bC^\alpha_\rho\bC^\beta_\omega.
\]
A precise definition and a discussion of the relevant properties of these spaces are provided in Section 2 below.

\medskip
\paragraph{\textit{Hypotheses.}}
Our analysis is built upon  three hypotheses concerning the multiscale system: the first ensures basic well-posedness and stability of the original system; the second guarantees the ergodicity of the frozen fast dynamics; and the third posits that the anticipated homogenized limit is itself a well-behaved, ergodic process. Each is formulated to be as weak as possible while still enabling a quantitative, uniform-in-time analysis.

We begin with the fundamental assumptions regarding the well-posedness and moment bounds for the original multiscale system \eqref{SDE1}:

\vspace{1mm}
\begin{enumerate}[{\bf (H$_0$)}]
\item
For each $z\in\mR^{d+\vartheta}$ and $\eps\in(0,1)$, there exists a unique weak solution $(Z^\eps_t)_{t\geq0}$ to \eqref{SDE1}
starting from $z$. Moreover, for every $\varrho_0\in\sW_{d+\vartheta}$ there exists $\varrho_1\in\sW_{d+\vartheta}$ such that
\begin{align}\label{Mom1}
\sup_{t\geq0}\sup_{\eps\in(0,1)}\mE\,\varrho_0(Z^\eps_t(z))\leq \varrho_1(z)<\infty,\quad z\in\mR^d\times\mR^\vartheta.
\end{align}
\end{enumerate}

\noindent
\textit{Discussion of {\bf (H$_0$)}}: The  uniformity in time moment estimates (\ref{Mom1}) is a natural and necessary  prerequisite for the uniform-in-time analysis. Such bounds can be verified for a wide range of models through Lyapunov function techniques or other model-specific arguments.

\vspace{2mm}
Next, we impose assumptions on the frozen fast dynamics (\ref{sde2}).

\vspace{1mm}
\begin{enumerate}[{\bf (H$_1$)}]
\item
\begin{enumerate}[(a)]
\item There exist $\lambda,\Lambda\in\sW_{d+\vartheta}$ such that for all $(x,y)\in\mR^d\times\mR^\vartheta$,
\begin{align*}
\lambda(x,y)|\xi|^2 \leq |\sigma(x,y)^*\xi|^2 \leq \Lambda(x,y)|\xi|^2,\quad \forall\,\xi\in\mR^d.
\end{align*}

\item For each fixed $y\in\mR^\vartheta$, there exists  a unique weak solution $X_t^y(x)$ to \eqref{sde2}. Furthermore, $X_t^y$ admits a unique invariant probability measure $\mu_y$ satisfying: there exists a decreasing function $\ell_0$ on $\mR_+$ with
$\int^\infty_0t\ell_0(t)\dif t<\infty$,
such that
for any $\rho_0\in\sW_d$, one can find a weight $\varrho_1\in\sW_{d+\vartheta}$ such that for all $(x,y)\in\mR^d\times\mR^\vartheta$ and $\varphi\in\bC^0_{\rho_0}(\mR^d)$,
\begin{align}\label{xy}
\big|\mE[\varphi(X^y_t(x))]-\mu_y(\varphi)\big|
\leq \ell_0(t)\,\varrho_1(x,y)\,\|\varphi\|_{\bC^0_{\rho_0}},
\quad \forall\,t>0.
\end{align}

\item The drift $H$ in (\ref{sde0}) satisfies the  centering condition
\begin{align*}
\mu_y\big(H(\cdot,y)\big)=\int_{\mR^d}H(x,y)\,\mu_y(\dif x)\equiv 0,\quad \forall y\in\mR^\vartheta.
\end{align*}
\end{enumerate}
\end{enumerate}

\noindent
\textit{Discussion of {\bf (H$_1$)}}:
Condition (a) is primarily used to handle unbounded and irregular coefficients with respect to the fast variable. Note that condition (b)  allows for rather general  mixing rates  (including polynomial decay of $\ell_0$),  which is essential for applications where only weak ergodicity is available. Finally, condition (c) is  classical  and  necessary to ensure that the $O(\eps^{-1})$ term $H$ contributes to the effective drift and diffusion at the macroscopic scale. It serves as the probabilistic analogue of the solvability condition for the cell problem in homogenization theory.

 \vspace{2mm}
 Finally, we state the hypotheses on the homogenized  dynamics \eqref{sde000}, whose coefficients $\cF$ and $\cG$ are defined in \eqref{bcf}.

\vspace{1mm}
\begin{enumerate}[{\bf (H$_2$)}]
\item
\begin{enumerate}[(a)]
\item For each $y\in\mR^\vartheta$, the homogenized SDE \eqref{sde000} admits a unique weak solution $\bar Y_t(y)$.
\item There exist $\lambda,\Lambda\in\sW_\vartheta$ such that for all $y\in\mR^\vartheta$ and $\xi\in\mR^\vartheta$,
\[
\lambda(y)|\xi|^2\leq \langle \cG(y)\xi,\xi\rangle\leq \Lambda(y)|\xi|^2.
\]
\item There exists a unique invariant probability measure $\bar\mu$ on $\mR^\vartheta$ with finite moments of all orders such that:
for every $\omega_0\in\sW_\vartheta$, there exist $\omega_1\in\sW_\vartheta$ and a decreasing function $\bar\ell\in L^1([0,\infty))$
satisfying, for all $\psi\in\bC^0_{\omega_0}(\mR^\vartheta)$ and all $t>0$,
\begin{align}\label{Exp1}
\big|\mE[\psi(\bar Y_t(y))]-\bar\mu(\psi)\big|
\leq \bar\ell(t)\,\omega_1(y)\,\|\psi\|_{\bC^0_{\omega_0}},\quad y\in\mR^\vartheta.
\end{align}
\end{enumerate}
\end{enumerate}

\noindent
\textit{Discussion of {\bf (H$_2$)}}: Lemma \ref{Le42} below shows that $\cG$ is always positive semidefinite, condition (b) strengthens this to strict positivity. This is a key structural assumption needed to handle unbounded and irregular coefficients in the slow variable. Condition (c) posits that the limiting  dynamics is ergodic with a rather weak mixing rate, which is also natural and essential for the uniformity in time of our main result, as it allows us to control the long-time behavior of $\bar Y_t$. Such  mixing properties are a standard consequence of the ellipticity and appropriate dissipativity structure  of the effective coefficients, see \cite[Examples 1.6 and 1.7]{XXZ26} for  precise and broad sufficient conditions. It is particularly interesting to note that both the positivity and the dissipativity   of the homogenized limit can be either inherited directly from  the original slow dynamics (see Section 5) or transferred entirely  from the mixing of the fast process via the corrector $\Phi$ (see Section 6 for the degenerate Langevin system).

\vspace{2mm}
For brevity, we denote the combined set of hypotheses by
\[
{\bf (H)}:={\bf (H_0)}+{\bf (H_1)}+{\bf (H_2)}.
\]
The main result of this paper is the following uniform-in-time diffusion approximation for the joint law of $(X_t^\eps,Y_t^\eps)$ in system (\ref{SDE1}).

\bt\label{main}
Assume that ${\bf (H)}$ hold. In addition, suppose that for some $\alpha,\beta\in(0,1]$,
\begin{align}\label{Red1}
b,\sigma,H\in\bC^\alpha_{\rm p}\bC^{1+\beta}_{\rm p},
\quad
c,F,G\in\bC^\alpha_{\rm p}\bC^{\beta}_{\rm p}.
\end{align}
Then for any $\gamma\in(\beta,2)$ and $(\rho_0,\omega_0)\in\sW_d\times\sW_\vartheta$, there exist
a weight $\varrho\in\sW_{d+\vartheta}$ such that for all $\varphi\in\bC^\alpha_{\rho_0}\bC^\gamma_{\omega_0}$,
 $\eps\in(0,1)$,  $z=(x,y)\in\mR^d\times\mR^\vartheta$,and every $t> 0$,
\begin{align}\label{XXX}
\big|\mE[\varphi(Z^\eps_t(z))]-\mE[\bar\varphi(\bar Y_t(y))]\big|
\leq \varrho(z)\Big(
\eps^\beta\|\varphi\|_{\bC^\alpha_{\rho_0}\bC^\gamma_{\omega_0}}
+\ell_0(t/\eps^2)\,\|\varphi-\bar\varphi\|_{\bC^0_{\rho_0}\bC^0_{\omega_0}}
\Big),
\end{align}
where $\bar\varphi(y):=\mu_y(\varphi(\cdot,y))$, with $\mu_y$  the invariant measure of the frozen system \eqref{sde2}, $\ell_0(t)$ is given in \eqref{xy}, and $\bar Y_t(y)$ solves the homogenized equation (\ref{sde000}).
\et

\br
\emph{(i) (Convergence of the slow marginal).} If $\varphi(x,y)=\psi(y)$ depends only on the slow variable, then $\bar\varphi=\psi$ and the estimate \eqref{XXX}  simplifies to
\[
\sup_{t\geq0}\big|\mE[\psi(Y_t^\eps(z))]-\mE[\psi(\bar Y_t(y))]\big|
\leq \eps^\beta\,\varrho(z)\,\|\psi\|_{\bC^\gamma_{\omega_0}}.
\]
This provides a uniform-in-time diffusion approximation for the slow component, recovering and extending classical finite-horizon results. Notably, the coefficients can be unbounded and the convergence rate is independent of the regularity of the coefficients in the fast variable.

\emph{(ii) (Convergence of the fast marginal).} If $\varphi(x,y)=\psi(x)$ depends only on the fast variable, then the result yields, for every $t>0$,
\begin{align}\label{XX}
\big|\mE[\psi(X_t^\eps(z))]-\mE[\mu_{\bar Y_t(y)}(\psi)]\big|
\leq\varrho(z)\big(\eps^\beta+\ell_0(t/\eps^2)\big)\,\|\psi\|_{\bC^\alpha_{\rho_0}}.
\end{align}
This implies that the law of $X_t^\eps(z)$ converges weakly to the mixture
$$
\mE[\mu_{\bar Y_t(y)}(\dif x)]=\int_{\mR^\vartheta}\mu_{\tilde y}(\dif x)\cL_{\bar Y_t(y)}(\dif \tilde y)
$$
as $\eps\to 0$.   The  presence of the term $\ell_0(t/\eps^2)$ in (\ref{XX}) is natural and physically meaningful: it captures the rapid relaxation of the fast variable to its local equilibrium $\mu_y$ on the  $O(\eps^{-2})$ time scale, conditioned on the current state of the slow variable, which itself evolves according to the homogenized dynamics. Thus, the limit as well as the convergence rate (\ref{XX}) precisely describe  this conditional equilibration at each macroscopic time $t$.

A standard approximation argument confirms that \eqref{XX} indeed implies weak convergence for any bounded continuous function. Indeed, let $\psi$ be a bounded continuous function on $\mR^d$ and $\psi_n$ be the mollification of $\psi$. Then
\begin{align*}
\big|\mE[\psi(X_t^\eps(z))]-\mE[\mu_{\bar Y_t(y)}(\psi)]\big|
&\leq\big|\mE[\psi_n(X_t^\eps(z))]-\mE[\mu_{\bar Y_t(y)}(\psi_n)]\big|\\
&+\big|\mE[\psi_n(X_t^\eps(z))]-\mE[\psi(X_t^\eps(z))]\big|\\
&+\big|\mE[\mu_{\bar Y_t(y)}(\psi_n)]-\mE[\mu_{\bar Y_t(y)}(\psi)]\big|\\
&=:I^{(1)}_{n,\eps}+I^{(2)}_{n,\eps}+I^{(3)}_{n}.
\end{align*}
For $I^{(1)}_{n,\eps}$, by what we have proved, for fixed $n$,
$$
\lim_{\eps\to0}I^{(1)}_{n,\eps}=0.
$$
For $I^{(2)}_{n,\eps}$, since $\psi$ is continuous, we have for any $R>0$,
$$
\lim_{n\to\infty}\sup_{|x|\leq R}|\psi_n(x)-\psi(x)|=0.
$$
Thus, by Chebyschev's inequality,
\begin{align*}
I^{(2)}_{n,\eps}
&\leq \sup_{|x|\leq R}|\psi_n(x)-\psi(x)|+\|\psi\|_\infty(\mP(|X^\eps_t(z)|>R)+\mP(|X^\eps_t(z)|>R))\\
&\leq \sup_{|x|\leq R}|\psi_n(x)-\psi(x)|+\|\psi\|_\infty(2\,\mE|X^\eps_t(z)|^2/R^2),
\end{align*}
which implies by \eqref{Mom1} that
$$
\lim_{n\to\infty}\sup_{\eps\in(0,1)}I^{(2)}_{n,\eps}=0.
$$
For $I^{(3)}_n$, by the dominated convergence theorem, we clearly have
$$
\lim_{n\to\infty}I^{(3)}_{n}=0.
$$
Combining the above limits, we obtain
$$
\lim_{\eps\to 0}\big|\mE[\psi(X_t^\eps(z))]-\mE[\mu_{\bar Y_t(y)}(\psi)]\big|=0.
$$
\er

\br
If one drops the ergodicity assumption  {\rm (c)} in {\bf (H$_2$)} and replaces the uniform moment bound \eqref{Mom1} by a finite-time moment bound,   the proof below  can be adapted to yield  a
finite-horizon convergence result.  Specifically,  for any fixed $T>0$, there exists a positive constant $C_T$ such that for all $z=(x,y)\in\mR^d\times\mR^\vartheta$ and $t\in[0,T]$,
\[
\big|\mE[\varphi(Z^\eps_t(z))]-\mE[\bar\varphi(\bar Y_t(y))]\big|
\leq C_T\cdot\varrho(z)\Big(
\eps^\beta\|\varphi\|_{\bC^\alpha_{\rho_0}\bC^\gamma_{\omega_0}}
+\ell_0(t/\eps^2)\,\|\varphi-\bar\varphi\|_{\bC^0_{\rho_0}\bC^0_{\omega_0}}
\Big).
\]
\er

\medskip

As a direct and significant consequence of Theorem \ref{main}, we obtain a quantitative resolution to Freidlin's question regarding the asymptotic behavior of the stationary measures of multiscale systems. Namely, suppose that for each $\eps>0$ the original system \eqref{SDE1} admits an invariant probability measure $\nu_\eps(\dif z)=\nu_\eps(\dif x,\dif y)$ such that for every $\varphi\in\bC^\alpha_{\rho_0}\bC^\gamma_{\omega_0}$ and $z\in\mR^d\times\mR^\vartheta$,
\begin{align}\label{inva}
\lim_{t\to\infty}\big|\mE[\varphi(Z_t^{\eps}(z))]-\nu_\eps(\varphi)\big|=0.
\end{align}
We show that $\nu_\eps$ converges weakly to the skew-product $\mu_y(\dif x)\bar\mu(\dif y)$ as $\eps\to0$, where $\bar\mu$ is the invariant measure of the homogenized system \eqref{sde000}.

\bt\label{main1}
Under the assumptions of Theorem~\ref{main} and \eqref{inva},
for each $\gamma\in(\beta,2)$ and $(\rho_0,\omega_0)\in\sW_d\times\sW_\vartheta$,
there exists a $C>0$ such that for every
$\varphi\in\bC^\alpha_{\rho_0}\bC^\gamma_{\omega_0}$ and $\eps\in(0,1)$,
\begin{align}\label{es1}
\bigg|\int_{\mR^{d+\vartheta}}\varphi(x,y)\,\nu_\eps(\dif x,\dif y)
-\int_{\mR^\vartheta}\!\int_{\mR^d}\varphi(x,y)\,\mu_y(\dif x)\,\bar\mu(\dif y)\bigg|
\leq C\eps^{\beta}\|\varphi\|_{\bC^\alpha_{\rho_0}\bC^\gamma_{\omega_0}}.
\end{align}
\et

\br
 (i) The skew-product structure of the limiting stationary measure is a precise realization of the   adiabatic elimination principle in the stationary regime: the fast degrees of freedom have been completely  integrated out, leaving a homogenized dynamics for the slow variables alone, but with statistical information about the fast scale encoded in the conditional measure $\mu_y$.

 (ii) From a theoretical perspective, this result demonstrates the commutativity of the limits $\eps\to0$ and $t\to\infty$ for a large class of observables:
\begin{align}\label{com}
\lim_{\eps\to0}\lim_{t\to\infty}\mE[\varphi(Z^\eps_t(z))] =\lim_{t\to\infty}\lim_{\eps\to0}\mE[\varphi(Z^\eps_t(z))].
 \end{align}
 This property is far from obvious given the complex  nature of the multiscale coupling. It thus validates the use of reduced models for long-time prediction, and provides a firm foundation for multiscale numerical methods, such as the Heterogeneous Multiscale Method (see e.g. \cite{ELV}), when applied to systems in their stationary regime.
\er

\begin{proof}
Fix $\varphi\in\bC^\alpha_{\rho_0}\bC^\gamma_{\omega_0}$ and let
$\bar\varphi(y):=\mu_y(\varphi(\cdot,y))$.
For any $t>0$ and $z=(x,y)\in\mR^d\times\mR^\vartheta$, we write
\begin{align*}
\big|\nu_\eps(\varphi)-\bar\mu(\bar\varphi)\big|
&\leq
\big|\nu_\eps(\varphi)-\mE[\varphi(Z_t^\eps(z))]\big|
+\big|\mE\big[\bar\varphi(\bar Y_t(y))\big]-\bar\mu(\bar\varphi)\big|\\
&\quad +\big|\mE[\varphi(Z_t^\eps(z))]-\mE\big[\bar\varphi(\bar Y_t(y))\big]\big|.
\end{align*}
By \eqref{Exp1} and the uniform  estimate \eqref{XXX}, the second and third terms are bounded by
\[
\bar\ell(t)\omega_1(y)\|\bar\varphi\|_{\bC^0_{\omega_0}}+
\varrho(z)\Big(
\eps^\beta\|\varphi\|_{\bC^\alpha_{\rho_0}\bC^\gamma_{\omega_0}}
+\ell_0(t/\eps^2)\|\varphi-\bar\varphi\|_{\bC^0_{\rho_0}\bC^0_{\omega_0}}
\Big).
\]
Taking the limit $t\to\infty$ in the above estimate and using \eqref{inva}, together with the facts that $\bar\ell(t)\to0$ and $\ell_0(t/\eps^2)\to0$, we obtain
\[
\big|\nu_\eps(\varphi)-\bar\mu(\bar\varphi)\big|\leq \varrho(z)\,\eps^{\beta}\|\varphi\|_{\bC^\alpha_{\rho_0}\bC^\gamma_{\omega_0}}.
\]
This proves \eqref{es1}. Indeed, it suffices to fix once and for all a reference point $z_*=(0,0)$ at the beginning of the argument and then take $C:=\varrho(z_*)$.
\end{proof}

\subsection{Applications}

To further demonstrate the power and scope of our framework, we present several highly non-trivial   applications, deriving explicit limiting dynamics and precise error bounds.

\subsubsection{Uniform-in-time averaging principle  (Section 5)}
In the classical averaging regime, where $c\equiv H\equiv0$ in \eqref{sde0}, our results still provide a significant generalization. In this setting, the homogenization correction naturally vanishes, and the effective coefficients reduce to
\begin{align*}
\cG(y)=\int_{\mR^{d}} GG^*(x,y)\,\mu_y(\dif x),
\quad
\cF(y)=\int_{\mR^{d}} F(x,y)\,\mu_y(\dif x).
\end{align*}
Thus, the ellipticity of $\cG$ and the dissipativity of $\cF$ required in hypothesis {\bf (H$_2$)} follow directly from the corresponding properties of $G$ and $F$ in the original system. We immediately obtain a uniform convergence  for the joint law of the system under general dissipativity and regularity assumptions, thereby superseding standard finite-horizon results that concern only the   convergence of the slow marginal.

\subsubsection{Uniform Smoluchowski--Kramers approximation with state-dependent friction (Section 6)}
A critical application of our joint-law analysis is the derivation of a uniform-in-time S-K limit for the  degenerate Langevin system (\ref{ex11}) with a state-dependent, possibly non-symmetric friction matrix $A(y)$. We overcome the strict degeneracy of the slow equation to prove that the joint position-scaled velocity system converges uniformly to a homogenized manifold with explicit, closed-form expressions for the effective macroscopic dynamics. Critically, the precise description of the joint law limit and the uniformity in time  enable us, for the first time, to rigorously justify the asymptotic limits of key thermodynamic functionals globally in time.  Specifically, we
\begin{itemize}
\item provide an explicit formula for the limiting effective energy, demonstrating a generalized equipartition theorem in which the macroscopic kinetic energy is regulated by the friction matrix through the Lyapunov equation;
\item obtain an explicit limit for the scaled entropy production rate, revealing a universal macroscopic fluctuation-dissipation relation that no longer depends on the non-symmetric  friction and is determined solely by the noise intensity;
\item give a formal asymptotic description of the effective free energy, where the correction is expressed by the Kullback--Leibler divergence between the local equilibrium and a reference Gaussian distribution. This clarifies how information carried by the fast degrees of freedom is renormalized at the macroscopic level.
\end{itemize}

Moreover, the above analysis extends naturally to a broader class of stochastic Hamiltonian systems with multiple time scales, where the fast dynamics retains a gradient-type or symplectic structure; see, e.g., \cite{B18,BW18}.

\subsubsection{Periodic homogenization with distributional drifts  (Section 7)}
Finally, we apply our methodology to a genuinely highly singular regime: periodic homogenization where the fast drift is a distribution. Specifically, we consider the operator
$$
\sL_\eps f(x) = \Delta f(x) + [\eps^{-1}b + c](x/\eps)\cdot\nabla f(x),
$$
where $b \in \bC^\alpha(\mT^d)$ possesses strictly negative H\"older regularity $\alpha \in (-1, -1/2)$ and is divergence-free ($\nabla \cdot b = 0$).  In this deeply singular regime, the product $b\cdot\nabla$ is classically ill-defined via standard paraproducts.
By integrating a periodic Zvonkin-type transformation directly into our multiscale framework, we algebraically eliminate the singularity. We prove that the effective macroscopic drift and diffusion, computed   for the regularized transformed system,  accurately govern the deeply singular original process, yielding an explicit uniform-in-time convergence rate of order $\mathcal{O}(\eps)$. To our knowledge, this is the first uniform-in-time homogenization result for stochastic systems driven by distributional drifts.

\subsection{Overview of the proof and organization of the paper}

The proof strategy is based on a novel decomposition of the joint-law error into three contributions: an effective slow-motion error, a conditional-equilibrium error for the fast component, and an initial-layer boundary term. Uniform-in-time control for the convergence of the slow process and for the deviation of the fast component from its local equilibrium relies on new regularity estimates for parameter-dependent Poisson equations and infinite-horizon Kolmogorov equations in weighted H\"older spaces, together with long-time decay estimates for derivatives of the associated semigroups (Theorems 3.2 and 3.4), which are of independent interest. These ingredients yield uniform fluctuation bounds for  the original multiscale system \eqref{sde0} (Lemmas 4.2 and 4.3).
Moreover, in stark contrast to works that rely exclusively on exponential derivative decay to derive uniform-in-time convergence results in settings such as Euler schemes and averaging limits \cite{BDOZ,Cr-Do-Go-Ot-So,Cr-Do-Ot,MST,PP}, we show that algebraic (polynomial) decay already suffices for a uniform-in-time diffusion approximation.

\medskip
The rest of the paper is organized as follows. In Section 2, we introduce weighted H\"older spaces. Section 3 is devoted to parameter-dependent Poisson and Kolmogorov equations with unbounded coefficients. Section 4 contains the core proof of the main theorems. Sections 5, 6, and 7 present the three applications detailed above.

\medskip

Throughout the paper, $C$ denotes a generic constant whose value may change from line to line. We write $A\lesssim_C B$ to mean $A\leq CB$ for some inessential constant $C>0$.

\section{Preliminaries: Banach-valued weighted H\"older spaces}

This section introduces the functional  framework that supplies our analysis. To handle coefficients with polynomial growth and to quantify regularity in the fast and slow variables separately, we introduce the weighted anisotropic H\"older spaces. We define these spaces, establish their basic properties, and provide the necessary mollification estimates, which will be used for handling low-regularity coefficients in later sections.

Let $\cO\subset\mR^d$ be an open set and $\mB$ a Banach space. For $k\in\mN_0$, we denote by $\bC^{k}(\cO;\mB)$ the space of $k$-times Fr\'echet differentiable $\mB$-valued functions on $\cO$ equipped with the norm
\[
\|f\|_{\bC^{k}(\cO;\mB)}:=\|f\|_{L^\infty(\cO;\mB)}+\cdots+\|\nabla^k f\|_{L^\infty(\cO;\mB)}<\infty.
\]
For $\alpha\in(0,1)$,  the H\"older space $\bC^{k+\alpha}(\cO;\mB)$ is defined by
\[
\|f\|_{\bC^{k+\alpha}(\cO;\mB)}
:=\|f\|_{\bC^{k}(\cO;\mB)}+[\nabla^k f]_{\bC^{\alpha}(\cO;\mB)}<\infty,
\]
where
\[
[\nabla^k f]_{\bC^{\alpha}(\cO;\mB)}
:=\sup_{x,y\in\cO, x\neq y}
\frac{\|\nabla^k f(x)-\nabla^k f(y)\|_\mB}{|x-y|^\alpha}.
\]
We also use $\bC^{k+\alpha}_{\loc}(\cO;\mB)$ to denote the local H\"older space:
\[
\bC^{k+\alpha}_{\loc}(\cO;\mB)
:=\Big\{f:\ f\phi\in \bC^{k+\alpha}(\cO;\mB),\ \forall \phi\in\bC^\infty_c(\cO;\mR)\Big\}.
\]

Recall the class $\sW_d$ in \eqref{sW}. For $\rho\in\sW_d$ and $\alpha\geq 0$, we define the weighted H\"older space
\[
\bC_\rho^\alpha(\mR^d;\mB)
:=\left\{ f\in \bC^\alpha_{\loc}(\mR^d;\mB):\
\|f\|_{\bC^\alpha_\rho(\mR^d;\mB)}
:=\sup_{x\in\mR^d}\frac{\|f\|_{\bC^\alpha(B_1(x);\mB)}}{\rho(x)}<\infty\right\},
\]
where $B_1(x)$ is the Euclidean ball in $\mR^d$ centered at $x$ with radius $1$.
When there is no confusion, we write
\[
\bC_\rho^\alpha(\mB):=\bC_\rho^\alpha(\mR^d;\mB),
\quad
\bC_\rho^\alpha:=\bC_\rho^\alpha(\mR^d;\mR).
\]

A key feature of this localization-based definition is its compatibility with difference operators. For $h\in\mR^d$, define the translation operator $\tau_h f(x):=f(x+h)$ and the first-order difference
\[
\cD_h f(x):=\tau_h f(x)-f(x).
\]
Then, for suitable $f,g$, we have the product rule
\begin{align}\label{Da2}
\cD_h(fg)=\cD_h f\cdot \tau_h g+f\cdot \cD_h g.
\end{align}
The following lemma provides a convenient characterization of the weighted H\"older norm via finite differences and establishes a product estimate.
\bl
Let $\alpha\in(0,1)$ and $k\in\mN_0$. Then
\begin{align}\label{Ho1}
\|f\|_{\bC_\rho^{k+\alpha}(\mB)}
\asymp
\|f\|_{\bC_\rho^{k}(\mB)}
+\sup_{|h|\leq 1}\frac{\|\cD_h\nabla^k f\|_{\bC_\rho^0(\mB)}}{|h|^\alpha}.
\end{align}
Moreover, for any $\rho_1,\rho_2\in\sW_d$, there exists a constant $C>0$ such that for all
$f\in\bC^{k+\alpha}_{\rho_1}(\mB)$ and $g\in\bC^{k+\alpha}_{\rho_2}(\mB)$,
\begin{align}\label{Ho11}
\|fg\|_{\bC^{k+\alpha}_{\rho_1\rho_2}(\mB)}
\leq C\,
\|f\|_{\bC^{k+\alpha}_{\rho_1}(\mB)}\,
\|g\|_{\bC^{k+\alpha}_{\rho_2}(\mB)}.
\end{align}
\el

\begin{proof}
The equivalence \eqref{Ho1}  follows from the standard characterization of H\"older norms via difference quotients, applied to $\nabla^k f$ and combined with the weight normalization over unit balls. The product estimate \eqref{Ho11} is a consequence of the Leibniz rule for $\nabla^k(fg)$, the rule \eqref{Da2} for the difference operator, and the elementary weight bound $\rho_i(x+h)\lesssim \rho_i(x)$ for $|h|\leq 1$ (see \eqref{weight-shift}), where $i=1,2$.
\end{proof}

Due to the low regularity assumptions on the coefficients, we will need some smoothing arguments. Let $\phi\in C^\infty_c(\mR^d)$ be a symmetric  probability density, and set
 \[
\phi_n(x):=n^d\phi(nx),\quad n\in\mN.
\]
For a function $f\in\bC^\alpha_\rho(\mB)$, we define its mollification by
\[
f_n(x):=f*\phi_n(x):=\int_{\mR^d} f(x-y)\phi_n(y)\,\dif y.
\]
 The following lemma quantifies the approximation and smoothing properties of this operation in the weighted setting.

\bl
For $0\leq\alpha\leq\gamma\leq 2$,
there is a constant $C=C(d,\alpha,\gamma,\phi)>0$ such that for all $f\in\bC^\gamma_\rho(\mB)$ and $n\in\mN$,
\begin{align}\label{App1}
\|f_n-f\|_{\bC^\alpha_\rho(\mB)}\leq C n^{\alpha-\gamma}\|f\|_{\bC^\gamma_\rho(\mB)},
\end{align}
and  for all $f\in\bC^\alpha_\rho(\mB)$ and $n\in\mN$,
\begin{align}\label{App2}
\|f_n\|_{\bC^\gamma_\rho(\mB)}\leq C n^{\gamma-\alpha}\|f\|_{\bC^\alpha_\rho(\mB)}.
\end{align}
\el

\begin{proof}
We sketch the main ideas; the presence of the weight $\rho$ is handled by using $\rho(x-y)\lesssim \rho(x)$ for $|y|\leq 1$ and the compact support of $\phi$.

\smallskip
\noindent{\it Step 1: estimate in $\bC^0_\rho(\mB)$.}
By symmetry and the fact that $\int \phi_n=1$, we have
\[
f_n(x)-f(x)=\int_{\mR^d}\big(f(x-y)-f(x)\big)\phi_n(y)\,\dif y.
\]
If $f\in\bC^\gamma_\rho(\mB)$ with $\gamma\in(0,1]$, then
$\|f(x-y)-f(x)\|_\mB\lesssim |y|^\gamma \rho(x)\|f\|_{\bC^\gamma_\rho(\mB)}$, hence
\[
\frac{\|f_n(x)-f(x)\|_\mB}{\rho(x)}
\lesssim \|f\|_{\bC^\gamma_\rho(\mB)}\int |y|^\gamma \phi_n(y)\,\dif y
\lesssim n^{-\gamma}\|f\|_{\bC^\gamma_\rho(\mB)}.
\]
For $\gamma>1$, one expands to order $\lfloor\gamma\rfloor$ and uses the vanishing moments of the symmetric kernel up to order $1$ (or, more generally, Taylor's theorem plus the cancellation coming from $\int y\,\phi_n(y)\,\dif y=0$) to obtain the same scaling $n^{-\gamma}$ in $\bC^0_\rho$.

\smallskip
\noindent{\it Step 2: estimate H\"older seminorms.}
Using the commutation $\cD_h f_n=(\cD_h f)*\phi_n$ and the characterization \eqref{Ho1} (with $k=\lfloor \alpha\rfloor$ when $\alpha\notin\mN$), one reduces \eqref{App1} to the $\bC^0_\rho$-bound applied to $\nabla^k \cD_h f$:
\[
\|\cD_h\nabla^k (f_n-f)\|_{\bC^0_\rho(\mB)}
\lesssim n^{-(\gamma-k)}\|\cD_h\nabla^k f\|_{\bC^{\gamma-k}_\rho(\mB)}
\lesssim n^{-(\gamma-k)}|h|^{\alpha-k}\|f\|_{\bC^\gamma_\rho(\mB)}.
\]
Dividing by $|h|^{\alpha-k}$ and taking $\sup_{|h|\leq 1}$ yields
(\ref{App1}).

\smallskip
\noindent{\it Step 3: smoothing estimate.}
For \eqref{App2}, note that for any integer $m,k\geq 0$,
\[
\nabla^{m+k} f_n =\nabla^k f*\nabla^m\phi_n,
\quad
\|\nabla^m \phi_n\|_{L^1}\lesssim n^{m}.
\]
Thus $\|\nabla^{m+k} f_n\|_{\bC^0_\rho(\mB)}\lesssim n^m\|\nabla^kf\|_{\bC^0_\rho(\mB)}$. For fractional $\gamma$, apply \eqref{Ho1} to $\nabla^{\lfloor\gamma\rfloor} f_n$ and use
\[
\cD_h \nabla^{\lfloor\gamma\rfloor} f_n
= f * \cD_h \nabla^{\lfloor\gamma\rfloor}\phi_n,
\quad
\|\cD_h \nabla^{\lfloor\gamma\rfloor}\phi_n\|_{L^1}
\lesssim (n|h|)^{\gamma-\lfloor\gamma\rfloor}\,n^{\lfloor\gamma\rfloor}.
\]
Combining these bounds gives  (\ref{App2}).
\end{proof}

Our analysis involves functions that depend on both the fast variable $x\in\mR^d$ and the slow variable $\theta\in\mR^\vartheta$, often with different regularity exponents in each direction. To capture this anisotropy, let $\omega\in\sW_\vartheta$ and $\beta\geq 0$. Taking $\mB=\bC^\beta_{\omega}(\mR^\vartheta)$, we consider functions $f(x,\theta)$ belonging to mixed H\"older spaces
\[
\bC^\alpha_{\rho}\big(\mR^d;\bC^\beta_{\omega}(\mR^\vartheta)\big).
\]
By swapping the order of variables and unfolding the definitions (all norms are taken on unit balls), we obtain an equivalent norm:
\begin{align}\label{Equi2}
\|f\|_{\bC^\alpha_{\rho}\big(\mR^d;\bC^\beta_{\omega}(\mR^\vartheta)\big)}
\asymp
\|\tilde f\|_{\bC^\beta_{\omega}\big(\mR^\vartheta;\bC^\alpha_{\rho}(\mR^d)\big)},
\quad \text{where}\quad \tilde f(\theta,x):=f(x,\theta).
\end{align}
For brevity, we write
\[
\bC^{\alpha}_{\rho}\bC^{\beta}_{\omega}
:=\bC^{\alpha}_{\rho}\big(\mR^d; \bC^{\beta}_{\omega}(\mR^\vartheta)\big)
=\bC^{\beta}_{\omega}\big(\mR^\vartheta; \bC^{\alpha}_{\rho}(\mR^d)\big)
=\bC^{\beta}_{\omega}\bC^{\alpha}_{\rho}.
\]
In particular, $\bC^{\alpha}_{\rho}\bC^{\beta}_{\omega}$ can be viewed as an anisotropic H\"older space on $\mR^d\times\mR^\vartheta$.

The following interpolation result is often useful. It shows that a function which is H\"older in each variable separately enjoys a mixed H\"older regularity of any intermediate order.

\bl
Let $\alpha,\beta\in[0,1]$ and $(\rho,\omega)\in\sW_d\times\sW_\vartheta$. If
$f\in \bC^{\alpha}_{\rho}\bC^0_{\omega}\cap\bC^0_{\rho}\bC^{\beta}_{\omega}$, then for any $\theta\in[0,1]$,
\[
f\in \bC^{\theta\alpha}_{\rho}\bC^{(1-\theta)\beta}_{\omega}.
\]
Moreover, the mixed second-order increment satisfies that for all $z\in\mR^d$ with $|z|\leq 1$ and $h\in\mR^\vartheta$ with $|h|\leq 1$,
$$
\|\Delta_{z,h}f\|_{\bC^0_{\rho\omega}} \lesssim |z|^{\theta\alpha}|h|^{(1-\theta)\beta},
$$
where
$$
\Delta_{z,h}f(x,y):=f(x+z,y+h)-f(x+z,y)-f(x,y+h)+f(x,y).
$$
\el

\begin{proof}
Fix $(x,y)\in\mR^d\times\mR^\vartheta$, $z\in\mR^d$, $h\in\mR^\vartheta$ with $|z|\leq 1$ and $|h|\leq 1$. Using the $y$-H\"older regularity (uniformly in $x$) and \eqref{weight-shift}, we obtain
\begin{align*}
|\Delta_{z,h}f(x,y)|
&\leq |f(x+z,y+h)-f(x+z,y)| + |f(x,y+h)-f(x,y)|\\
&\lesssim |h|^\beta \rho(x+z)\omega(y) + |h|^\beta \rho(x)\omega(y)\lesssim |h|^\beta \rho(x)\omega(y).
\end{align*}
Similarly, using the $x$-H\"older regularity (uniformly in $y$), we have
\begin{align*}
|\Delta_{z,h}f(x,y)|
&\leq |f(x+z,y+h)-f(x,y+h)| + |f(x+z,y)-f(x,y)|\\
&\lesssim |z|^\alpha \rho(x)\omega(y+h) + |z|^\alpha \rho(x)\omega(y)\lesssim |z|^\alpha \rho(x)\omega(y),
\end{align*}
Combining these two estimates yields
$$
|\Delta_{z,h}f(x,y)| \lesssim (|z|^\alpha \wedge |h|^\beta)\,\rho(x)\omega(y).
$$
Finally, for any $a,b\geq 0$ and $\theta\in[0,1]$, we have the elementary inequality$\min(a,b)\leq a^\theta b^{1-\theta}$. Applying this  with $a=|z|^\alpha$ and $b=|h|^\beta$ gives
$$
|\Delta_{z,h}f(x,y)| \lesssim |z|^{\theta\alpha}|h|^{(1-\theta)\beta}\rho(x)\omega(y).
$$
This yields the claimed mixed bound.
\end{proof}

\section{Parameter-dependent Poisson and Kolmogorov equations}

This section develops the core analytic tools for our analysis: regularity estimates for Poisson and infinite-horizon Kolmogorov equations whose coefficients depend on a parameter.  We establish two main results: (i) regularity of the solution to the Poisson equation with respect to the parameter (see Theorem \ref{p1}), and (ii) regularity and long-time decay estimates for the parameter-dependent infinite-horizon Kolmogorov equation (see Theorem \ref{cp}). These estimates are of independent interest and may find applications in other problems involving parameter-dependent stochastic dynamics, such as sensitivity analysis, numerical schemes or other limit theorems for multiscale SDEs.

\subsection{Parameterized SDE and standing assumptions}
To develop a unified framework that captures the dependence on the  parameter of the frozen system, we consider a family of SDEs indexed by  a parameter $\theta\in\mR^\vartheta$:
\begin{align}\label{SDE-P}
\dif X^\theta_t=b_\theta(X^\theta_t)\dif t+\sqrt{2}\,\sigma_\theta(X^\theta_t)\dif W_t,\quad X^\theta_0=x\in\mR^d,
\end{align}
where $W_t$ is an $m$-dimensional standard Brownian motion and
\[
(b,\sigma):\mR^\vartheta\times\mR^d\to\big(\mR^d,\mR^d\otimes\mR^m\big)
\]
are measurable functions of $(\theta,x)$.
Throughout this section, we impose the following set of assumptions.

\begin{enumerate}[{\bf (A$^\alpha_\beta$)}]
\item
For some $\alpha\in(0,1)$ and $\beta\in(0,2)$, we have
$
b,\sigma\in \bC^{\beta}_{\rm p}\bC^{\alpha}_{\rm p}=\bC^{\beta}_{\rm p}(\mR^\vartheta;\bC^{\alpha}_{\rm p}(\mR^d)).
$
Moreover,  the following conditions hold:
\begin{enumerate}[(a)]
\item For  any $T>0$ and $\rho_0\in\sW_d$, there exist $(\rho_1,\omega_1)\in\sW_d\times\sW_\vartheta$ such that for all $t\in[0,T]$ and $(x,\theta)\in\mR^d\times\mR^\vartheta$,
\begin{align}\label{Mom9}
\mE\rho_0(X^\theta_t(x))\leq \omega_1(\theta)\rho_1(x).
\end{align}
\item There exist $\lambda,\Lambda\in\sW_{\vartheta+d}$ such that for all $(\theta,x)\in\mR^\vartheta\times\mR^d$,
\begin{equation*}
\lambda_\theta(x)|\xi|^{2} \leq |\sigma^*_\theta(x)\xi|^2 \leq \Lambda_\theta(x) |\xi|^{2},
\quad \forall \,\xi \in \mR^d.
\end{equation*}
\item For each $\theta\in\mR^\vartheta$, the process $X^\theta_t(x)$  admits a unique invariant probability
measure $\mu_\theta$. Furthermore, for any $\rho_0\in \sW_d$, there exist
$(\omega,\rho)\in \sW_\vartheta\times \sW_d$ and a decreasing function $\ell_0$
with $\int^\infty_0 t\ell_0(t)\dif t<\infty$
such that for any $\varphi\in \bC^0_{\rho_0}$,
\begin{align}\label{Erg11}
\|\cT^\theta_t\varphi-\mu_\theta(\varphi)\|_{\bC^0_\rho}
\leq \ell_0(t)\,\omega(\theta)\|\varphi\|_{\bC^0_{\rho_0}},\ \ \forall (t,\theta)\in (1,\infty)\times\mR^\vartheta,
\end{align}
where $
\cT^\theta_t\varphi(x):=\mE[\varphi(X^\theta_t(x))].
$
\end{enumerate}
\end{enumerate}

\vspace{2mm}
 From \eqref{Mom9} and \eqref{Erg11}, we immediately obtain a  bound on the invariant measure: for any $\rho_0\in\sW_d$ there exists $\omega\in\sW_\vartheta$ such that for any $\varphi\in\bC^0_{\rho_0}$,
\begin{align}\label{Mom2}
|\mu_\theta(\varphi)|\leq \omega(\theta)\|\varphi\|_{\bC^0_{\rho_0}}.
\end{align}
The generator of \eqref{SDE-P} is given by
\begin{align*}
\sL_\theta f
=(\sigma_\theta\sigma_\theta^*):\nabla^2_x f+b_\theta\cdot\nabla_x f.
\end{align*}
Since $\mu_\theta$ is invariant for $\cT^\theta$, we have $\sL_\theta^*\mu_\theta=0$ and hence
\begin{equation}\label{InvGen}
\mu_\theta(\sL_\theta\varphi)=0
\quad \text{for all }\varphi\in\bC^2_\rho(\mR^d)\text{ and }\rho\in\sW_d.
\end{equation}

\subsection{Poisson equation on the whole space with a parameter}
In this subsection, we study the following Poisson equation on $\mR^d$:
\begin{align}\label{pde4}
\sL_\theta u_\theta(x)=f_\theta(x)-\mu_\theta(f_\theta)=:\widetilde f_\theta(x),
\end{align}
where $\theta$ is a parameter, and the source term $f$ belongs to an anisotropic H\"older space:
for some $\gamma\geq0$ and $(\rho_0,\omega_0)\in\sW_d\times\sW_\vartheta$,
\begin{align}\label{fff}
f\in \bC_{\omega_0}^{\gamma}\bC_{\rho_0}^{\alpha},
\quad\text{i.e.}\quad
f_\theta(\cdot)\in \bC^\alpha_{\rho_0}\ \text{and}\ \theta\mapsto f_\theta \ \text{is }\gamma\text{-H\"older in }\bC^\alpha_{\rho_0}.
\end{align}
Under assumption  {\bf (A$^\alpha_\beta$)} and by \cite[Theorem 5.1]{XXZ26}, equation \eqref{pde4} admits a unique solution which is given by
\begin{align}\label{pou}
u_\theta(x):=-\int_0^\infty \cT^\theta_t \widetilde f_\theta(x)\,\dif t.
\end{align}
We first recall the following Schauder estimate from \cite[Theorem 5.2]{XXZ26}.

\bl\label{lem:XXZ-u-est}
Assume {\bf (A$^\alpha_\beta$)} and (\ref{fff}) hold. Then there exist weights $(\rho,\omega)\in\sW_d\times\sW_\vartheta$ such that for all $\theta\in\mR^\vartheta$,
\begin{align}\label{u3}
\|u_\theta\|_{\bC^{2+\alpha}_\rho}\leq \omega(\theta)\|f_\theta\|_{\bC^\alpha_{\rho_0}}.
\end{align}
\el

We now prove the main result of this subsection: the regularity of  $\theta\mapsto \mu_\theta(f_\theta)$ and $\theta\mapsto u_\theta$.

\bt\label{p1}
Assume {\bf (A$^\alpha_\beta$)} and (\ref{fff}) hold. For any $(\rho_0,\omega_0)\in\sW_d\times\sW_\vartheta$ and $\gamma\in[0,\beta]$,
there exist weights $\bar\omega\in\sW_\vartheta$ and $(\rho,\omega)\in\sW_d\times\sW_\vartheta$ such that
\begin{align}\label{u12}
\|\mu_\cdot(f_\cdot)\|_{\bC^{\gamma}_{\bar\omega}}\leq
\|f\|_{\bC_{\omega_0}^{\gamma}\bC^{\alpha}_{\rho_0}},
\end{align}
and
\begin{align}\label{u23}
\|u\|_{\bC^{\gamma}_{\omega}\bC^{2+\alpha}_{\rho}}
\leq\|f\|_{\bC_{\omega_0}^{\gamma}\bC^{\alpha}_{\rho_0}}.
\end{align}
\et

\begin{proof}
By interpolation, it suffices to consider $\gamma=\beta$. We split the proof into two cases.

\medskip
\noindent{\it (i) Case 1: $\beta\in(0,1]$.}
Let $\cD^\theta_h$ denote the difference operator in the $\theta$-variable:
$$
\cD^\theta_h g(\theta):=\tau_h^\theta g(\theta)-g(\theta),\quad\tau_h^\theta g(\theta):=g(\theta+h).
$$
Applying $\cD^\theta_h$ to \eqref{pde4} and using the product rule \eqref{Da2} (in the $\theta$-variable) yields
\begin{align}\label{Es23}
\cD^\theta_h\widetilde f_\theta
=\cD^\theta_h(\sL_\theta u_\theta)
=\sL_\theta(\cD^\theta_h u_\theta)+(\cD^\theta_h\sL_\theta)\,\tau^\theta_h u_\theta.
\end{align}
Integrating against $\mu_\theta$ and using \eqref{InvGen} gives
\[
\mu_\theta(\cD^\theta_h\widetilde f_\theta)
=\mu_\theta\big((\cD^\theta_h\sL_\theta)\,\tau^\theta_h u_\theta\big).
\]
Since $\widetilde f_\theta=f_\theta-\mu_\theta(f_\theta)$, we obtain the key identity
\begin{align}\label{Es21}
\cD^\theta_h(\mu_\theta(f_\theta))
=\mu_\theta(\cD^\theta_h f_\theta)-\mu_\theta\big((\cD^\theta_h\sL_\theta)\,\tau^\theta_h u_\theta\big).
\end{align}
Moreover, note that
\begin{align}\label{Mom3}
(\cD^\theta_h\sL_\theta)\,\tau^\theta_h u_\theta
=
\cD^\theta_h(\sigma_\theta\sigma_\theta^*):\nabla^2_x u_{\theta+h}
+(\cD^\theta_h b_\theta)\cdot\nabla_x u_{\theta+h}.
\end{align}
Using $b,\sigma\in\bC^\beta_{\rm p}\bC^\alpha_{\rm p}$, the difference characterization \eqref{Ho1},
and the Schauder estimate \eqref{u3}, we obtain weights $(\rho_1,\omega_1)\in\sW_d\times\sW_\vartheta$ such that for all $|h|\leq 1$ and $\theta\in\mR^\vartheta$,
\[
\|(\cD^\theta_h\sL_\theta)\,\tau^\theta_h u_\theta\|_{\bC^0_{\rho_1}}
\leq  |h|^\beta\,\omega_1(\theta)\,\|f_{\theta+h}\|_{\bC^\alpha_{\rho_0}}.
\]
Inserting this bound into \eqref{Es21} and using \eqref{Mom2}  yields, for some $\bar\omega\in\sW_\vartheta$,
\[
|\cD^\theta_h(\mu_\theta(f_\theta))|
\leq  |h|^\beta\,\bar\omega(\theta)\,\|f\|_{\bC^\beta_{\omega_0}\bC^\alpha_{\rho_0}},
\]
which is exactly \eqref{u12} for $\gamma=\beta\in(0,1]$.

Next, from \eqref{Es23} we rewrite
\begin{align}\label{Es25}
\sL_\theta(\cD^\theta_h u_\theta)=\cD^\theta_h\widetilde f_\theta-(\cD^\theta_h\sL_\theta)\,\tau^\theta_h u_\theta=:g^h_\theta.
\end{align}
By construction and \eqref{InvGen}, $\mu_\theta(g^h_\theta)=0$.
Using again \eqref{Mom3}, \eqref{Ho1}, \eqref{u3}, \eqref{u12} (already proved) and
$b,\sigma\in\bC^{\beta}_{\rm p}\bC^\alpha_{\rm p}$,
we obtain weights $(\rho_g,\omega_g)\in\sW_d\times\sW_\vartheta$ such that
 for all $|h|\leq 1$ and $\theta\in\mR^\vartheta$,
\[
\|g^h_\theta\|_{\bC^\alpha_{\rho_g}}
\leq |h|^\beta\,\omega_g(\theta)\,\|f\|_{\bC^\beta_{\omega_0}\bC^\alpha_{\rho_0}}.
\]
Applying \eqref{u3} to the Poisson equation $\sL_\theta v=g^h_\theta$ (with $\mu_\theta(g^h_\theta)=0$) gives, for some $(\rho,\omega)\in\sW_d\times\sW_\vartheta$,
\[
\|\cD^\theta_h u_\theta\|_{\bC^{2+\alpha}_{\rho}}
\leq |h|^\beta\,\omega(\theta)\,\|f\|_{\bC^\beta_{\omega_0}\bC^\alpha_{\rho_0}},
\]
which yields \eqref{u23} for $\gamma=\beta\in(0,1]$ by the definition of $\bC^\beta_\omega$ in the $\theta$-variable.

\medskip
\noindent{\it (ii) Case 2: $\beta\in(1,2)$.}
Fix $i\in\{1,\dots,\vartheta\}$ and let $e_i$ be the $i$-th coordinate vector. Taking $h=\eps e_i$ in \eqref{Es21}, dividing by $\eps$,
and letting $\eps\to0$ (the limit is justified by the continuity of \(\theta\mapsto u_\theta\) established in Case 1 and dominated convergence), we obtain
\[
\p_{\theta_i}\big[\mu_\theta(f_\theta)\big]
=\mu_\theta(\p_{\theta_i}f_\theta)-\mu_\theta\big((\p_{\theta_i}\sL_\theta)u_\theta\big)
=:\mu_\theta(g_\theta),
\]
where
\[
g_\theta:=\p_{\theta_i}f_\theta-(\p_{\theta_i}\sL_\theta)u_\theta.
\]
Using $b,\sigma\in\bC^\beta_{\rm p}\bC^\alpha_{\rm p}$ and \eqref{u3}, we infer that
$g\in \bC^{\beta-1}_{\omega_g}\bC^\alpha_{\rho_g}$ for some weights $(\rho_g,\omega_g)$, with
\[
\|g\|_{\bC^{\beta-1}_{\omega_g}\bC^\alpha_{\rho_g}}\leq \|f\|_{\bC^\beta_{\omega_0}\bC^\alpha_{\rho_0}}.
\]
Applying Case~1 to the map $\theta\mapsto \mu_\theta(g_\theta)$ (now with exponent $\beta-1\in(0,1]$) yields the
$\bC^{\beta-1}$-regularity of $\p_{\theta_i}[\mu_\theta(f_\theta)]$, hence $\mu_\cdot(f_\cdot)\in \bC^\beta$ and \eqref{u12} follows.

For \eqref{u23}, as above, differentiate \eqref{Es25} in $\theta_i$ to get
\[
\sL_\theta(\p_{\theta_i}u_\theta)
=\p_{\theta_i}\widetilde f_\theta-(\p_{\theta_i}\sL_\theta)u_\theta=:g_\theta,
\quad \mu_\theta(g_\theta)=0,
\]
and then repeat the difference-quotient argument of Case~1 with $\p_{\theta_i}u_\theta$ in place of $u_\theta$.
This gives $\p_{\theta_i}u\in \bC^{\beta-1}_\omega\bC^{2+\alpha}_\rho$, and therefore $u\in\bC^\beta_\omega\bC^{2+\alpha}_\rho$, establishing \eqref{u23}.
\end{proof}

\begin{remark}
Theorem \ref{p1} provides a systematic regularity transfer mechanism: from the known regularity of the Poisson solution $u_\theta$ in the  $x$-variable (as provided by Lemma \ref{lem:XXZ-u-est}), we deduce its regularity with respect to the parameter $\theta$, as well as the regularity of the averaged quantity $\mu_\theta(f_\theta)$. Beyond the present multiscale context, such parameter-dependent estimates are valuable in sensitivity analysis for SDEs (e.g., computing derivatives of invariant measures with respect to parameters) and in the numerical analysis of multiscale methods, where the corrector must be approximated as a function of the slow variable.
\end{remark}

\subsection{Infinite-horizon Kolmogorov equation with a parameter}
We now turn to study the regularity of $\cT^\theta_t\varphi_\theta(x)$ with respect to parameter $\theta$.
Note that the function $u_\theta(t,x):=\cT^\theta_t\varphi_\theta(x)$ solves the following Kolmogorov equation on $[0,\infty)\times\mR^d$:
\begin{equation}\label{eqV}
\p_t u_\theta(t,x)=\sL_\theta u_\theta(t,x),\quad u_\theta(0,x)=\varphi_\theta(x).
\end{equation}
We first recall the following fundamental estimate from \cite[Theorems 1.5 and 4.11]{XXZ26}, which quantifies the decay of the semigroup in Schauder estimate.

\bt
Under {\bf (A$^\alpha_\beta$)}, for any
$\gamma\in[0,\alpha]$, $\delta\in(\gamma,1]$ and $\rho_0\in\sW_d$, there exist
$(\rho_1,\omega_1)\in\sW_d\times\sW_\vartheta$  and $\ell$ with $\int_0^\infty(1\vee t)\ell(t)\dif t<\infty$
such that for all
$\varphi\in\bC^\delta_{\rho_0}(\mR^d)$,
\begin{align}\label{Erg121}
\|\cT^\theta_t\varphi-\mu(\varphi)\|_{\bC^{2+\gamma}_{\rho_1}}
\leq \ell(t)\,\omega_1(\theta)\,\|\varphi\|_{\bC^\delta_{\rho_0}},
\quad t>0,\ \theta\in\mR^\vartheta.
\end{align}
\et

We now prove the $\theta$-regularity for $\cT^\theta_t\varphi_\theta$ when $\beta\in(0,1]$
and $\varphi\in\bC^\beta_{\omega_0}\bC^\alpha_{\rho_0}$, which will be crucial in Section 4 for proving the convergence of the joint law.

\bt\label{cp}
Assume {\bf (A$^\alpha_\beta$)} holds with $\beta\in(0,1]$.
For any $\gamma\in[0,\alpha]$, $\delta\in(\gamma,1]$ and $(\rho_0,\omega_0)\in\sW_d\times\sW_\vartheta$, there
exist weights $(\rho,\omega)\in\sW_d\times\sW_\vartheta$ such that for all $\varphi\in\bC^\beta_{\omega_0}\bC^\delta_{\rho_0}$ and all $t>0$,
\begin{align}\label{cp0}
\|\cT^\cdot_t\widetilde\varphi_\cdot\|_{\bC^\beta_\omega\bC^0_\rho}
&\leq \|\varphi\|_{\bC^\beta_{\omega_0}\bC^\delta_{\rho_0}},
\end{align}
where $\widetilde\varphi_\theta(x):=\varphi_\theta(x)-\mu_\theta(\varphi_\theta)$, and
\begin{align} \label{cp1}
\|\cT^\cdot_t\widetilde\varphi_\cdot\|_{\bC^\beta_\omega\bC^{2+\gamma}_\rho}
&\leq\Big(\ell(t/2)+\int_t^\infty \ell(s)\,\dif s\Big)\,
\|\varphi\|_{\bC^\beta_{\omega_0}\bC^\delta_{\rho_0}},
\end{align}
where $\ell$ is the same as in (\ref{Erg121}).
\et

\begin{proof}
In the following, for simplicity of notation,
we always use $(\rho,\omega)\in\sW_d\times\sW_\vartheta$ to denote generic polynomial weights,
which may change from line to line.

For $\varphi\in\bC^\beta_{\omega_0}\bC^\delta_{\rho_0}$, set
$u_\theta(t,x):=\cT^\theta_t\widetilde\varphi_\theta(x).$
Using \eqref{Erg121} with $\varphi=\varphi_{\theta}$, we obtain that for some $(\rho,\omega)\in\sW_d\times\sW_\vartheta$,
\[
\|u_{\theta}(t)\|_{\bC^{2+\gamma}_{\rho}}
\leq\ell(t)\,\omega(\theta)\,\|\varphi_{\theta}\|_{\bC^\delta_{\rho_0}},\ \ t>0,\ \theta\in\mR^\vartheta.
\]
Applying $\cD_h^\theta$ to \eqref{eqV}, we have
\[
\p_t \cD^\theta_hu_\theta-\sL_\theta \cD^\theta_h u_\theta=(\cD^\theta_h\sL_\theta)\,\tau^\theta_h u_\theta=:f^h_\theta.
\]
Note that
\[
(\cD^\theta_h\sL_\theta)\psi
=\cD^\theta_h(\sigma_\theta\sigma_\theta^*):\nabla_x^2\psi
+(\cD^\theta_h b_\theta)\cdot\nabla_x\psi.
\]
By $b,\sigma\in\bC^\beta_{\rm p}\bC^\alpha_{\rm p}$, the product rule,
 and shrinking/enlarging weights if needed,
there exist weights $(\rho,\omega)\in\sW_d\times\sW_\vartheta$
such that for all $|h|\leq 1$, $t>0$ and $\theta\in\mR^\vartheta$,
\begin{align}\label{FF0b}
\|f^h_\theta(t)\|_{\bC^{\alpha}_\rho}
\leq|h|^\beta\,\omega(\theta)\,\ell(t)\,\|\varphi\|_{\bC^0_{\omega_0}\bC^\delta_{\rho_0}}.
\end{align}
By Duhamel's formula,
\begin{align*}
\cD^\theta_hu_\theta(t)&=\cT^\theta_t(\cD^\theta_h\widetilde\varphi_\theta)+\int_0^t\cT^\theta_{t-s}f^h_\theta(s)\,\dif s\\
&=: \sI_1(t)+\sI_2(t)+\sI_3(t),
\end{align*}
where
\begin{align*}
\sI_1(t)&:=\cT^\theta_t(\cD^\theta_h\widetilde\varphi_\theta)-\mu_\theta(\cD^\theta_h\widetilde\varphi_\theta),\\
\sI_2(t)&:=\int^t_0\Big(\cT^\theta_{t-s}f^h_\theta(s)-\mu_\theta(f^h_\theta(s))\Big)\dif s,\\
\sI_3(t)&:=\mu_\theta(\cD^\theta_h\widetilde\varphi_\theta)+\int^t_0\mu_\theta(f^h_\theta(s))\dif s.
\end{align*}
{\it (i) Estimate of $\sI_1(t)$.}
From \eqref{Erg11} and \eqref{u12},
\[
\|\sI_1(t)\|_{\bC^0_{\rho_1}}
\leq \omega(\theta)\,\|\cD^\theta_h\widetilde\varphi_\theta\|_{\bC^0_{\rho_0}}
\leq |h|^\beta\,\omega(\theta)\,\|\varphi\|_{\bC^\beta_{\omega_0}\bC^\delta_{\rho_0}}.
\]
Similarly, by \eqref{Erg121}  and \eqref{u12},
\[
\|\sI_1(t)\|_{\bC^{2+\gamma}_{\rho_1}}
\leq \ell(t)\,\omega(\theta)\,\|\cD^\theta_h\widetilde\varphi_\theta\|_{\bC^\delta_{\rho_0}}
\leq |h|^\beta\,\ell(t)\,\omega(\theta)\,\|\varphi\|_{\bC^\beta_{\omega_0}\bC^\delta_{\rho_0}}.
\]

\medskip\noindent{\it (ii) Estimate of $\sI_2(t)$.}
Using \eqref{Erg11} and \eqref{FF0b},
\begin{align*}
\|\sI_2(t)\|_{\bC^0_{\rho_1}}
&\leq \int_0^t\ell(t-s) \omega(\theta)\,\|f^h_\theta(s)\|_{\bC^0_\rho}\,\dif s\\
&\leq |h|^\beta\,\omega(\theta)\,\|\varphi\|_{\bC^\beta_{\omega_0}\bC^\delta_{\rho_0}} \int_0^\infty \ell(s)\,\dif s.
\end{align*}
Similarly, by \eqref{Erg121},
\[
\|\sI_2(t)\|_{\bC^{2+\gamma}_{\rho_1}}
\leq \int_0^t \ell(t-s)\,\omega(\theta)\,\|f^h_\theta(s)\|_{\bC^\delta_\rho}\,\dif s
\leq |h|^\beta\,\omega(\theta)\,\|\varphi\|_{\bC^\beta_{\omega_0}\bC^\delta_{\rho_0}}
\int_0^t \ell(t-s)\ell(s)\,\dif s.
\]
Since $\ell$ is non-increasing,
\[
\int_0^t \ell(t-s)\ell(s)\,\dif s=\left(\int^{t/2}_0+\int^t_{t/2}\right)\ell(t-s)\ell(s)\,\dif s
\leq 2\,\ell(t/2)\int_0^\infty \ell(s)\,\dif s,
\]
and hence
\[
\|\sI_2(t)\|_{\bC^{2+\gamma}_{\rho_1}}
\leq |h|^\beta\,\omega(\theta)\,\ell(t/2)\,\|\varphi\|_{\bC^\beta_{\omega_0}\bC^\delta_{\rho_0}}.
\]
\medskip\noindent{\it (iii) Estimate of $\sI_3(t)$.}
Using the Poisson representation \eqref{pou} with $f=\widetilde\varphi_{\theta}$, i.e.,
$$
\sL_\theta U_\theta=-\widetilde\varphi_\theta\quad \text{with} \quad U_\theta=\int_0^\infty \cT^\theta_s\widetilde\varphi_\theta\,\dif s,
 $$
 and the identity (which is justified by \eqref{Erg11})
\[
\int_0^\infty \sL_\theta \cT^\theta_s \widetilde\varphi_\theta\,\dif s
=\lim_{T\to\infty}\Big(\cT^\theta_T\widetilde\varphi_\theta-\widetilde\varphi_\theta\Big)
=-\widetilde\varphi_\theta
\]
we obtain
\[
\mu_\theta(\cD^\theta_h\widetilde\varphi_\theta)
=-\int_0^\infty \mu_\theta\big(\cD^\theta_h(\sL_\theta u_\theta(s))\big)\,\dif s
=-\int_0^\infty \mu_\theta\big(f^h_\theta(s)\big)\,\dif s,
\]
where we used $\mu_\theta(\sL_\theta(\cD^\theta_h u_\theta(s)))=0$ and $(\cD^\theta_h\sL_\theta)\tau^\theta_h u_\theta(s)=f^h_\theta(s)$. Consequently,
$$
\sI_3(t)=-\int_t^\infty \mu_\theta\big(f^h_\theta(s)\big)\,\dif s,
$$
and therefore, by \eqref{FF0b},
\[
|\sI_3(t)|
\leq \int_t^\infty |\mu_\theta(f^h_\theta(s))|\,\dif s
\leq |h|^\beta\,\omega(\theta)\,\|\varphi\|_{\bC^\beta_{\omega_0}\bC^\delta_{\rho_0}}
\int_t^\infty \ell(s)\,\dif s.
\]

Collecting the bounds for $\sI_1,\sI_2,\sI_3$ and dividing by $|h|^\beta$ gives that for
some $\omega\in\sW_\vartheta$ and all $\theta\in\mR^\vartheta$ and $t>0$,
\[
\sup_{|h|\leq 1}\|\cD^\theta_h u_\theta(t)\|_{\bC^0_{\rho_1}}/|h|^\beta
\leq \omega(\theta)\|\varphi\|_{\bC^\beta_{\omega_0}\bC^\delta_{\rho_0}},
\]
and
\[
\sup_{|h|\leq 1}\|\cD^\theta_h u_\theta(t)\|_{\bC^{2+\gamma}_{\rho_1}}/|h|^\beta
\leq \omega(\theta)\Big(\ell(t/2)+\int_t^\infty \ell(s)\,\dif s\Big)\|\varphi\|_{\bC^\beta_{\omega_0}\bC^\delta_{\rho_0}}.
\]
These are exactly \eqref{cp0}--\eqref{cp1} by the characterization \eqref{Ho1} in the $\theta$-variable.
\end{proof}
\br\label{Re35}
Since $\int^\infty_0(1\vee t)\ell(t)\dif t<\infty$, by integration by parts it is easy to see that
$$
\int^\infty_0\left(\ell(t/2)+\int_t^\infty \ell(s)\,\dif s\right)\dif t<\infty.
$$
This integrability will be crucial in Section 4 for controlling the relaxation of the fast variable to its conditional equilibrium.
\er

	\section{Uniform-in-time diffusion approximation for multi-scale systems}

This section is devoted to the proof of our main result, Theorem \ref{main}. The argument proceeds in two stages. First, we establish a uniform-in-time diffusion approximation for the slow variable $Y_t^\eps$ alone (see Theorem \ref{thm:slow-variable}).
This step relies on regularity estimates  for parameter-dependent Poisson equations in the fast variable  developed in Subsection 3.2, together with  long-time decay estimates for the semigroup corresponding to the homogenized equation. Second, we extend this convergence to the joint process
$Z_t^\eps=(X_t^\eps,Y_t^\eps)$ by a novel decomposition, and carefully handle the initial layer--in which the fast variable relaxes to its local equilibrium--using the infinite-horizon parameter-dependent Kolmogorov equation estimates from Subsection 3.3.

\medskip

Recall the compact form \eqref{SDE1}  of the multiscale system, and we always write $z=(x,y)\in\mR^d\times\mR^\vartheta$.
For a function $\psi:[0,\infty)\times\mR^d\times\mR^\vartheta\to\mR$ that is $C^1$ in time and $C^2$ in space, and satisfies appropriate growth conditions, It\^o's formula yields
\begin{align}\label{Ito1}
\mE\,\psi(t, Z^\eps_t)=\psi(0, z)+\int^t_0\mE\Big[(\p_s\psi+\sL_\eps\psi)(s,Z^\eps_s)\Big]\dif s,
\end{align}
where
\[
\sL_\eps\psi
:=
(B_\eps+\eps^{-1}D)\cdot\nabla_z\psi+\Theta_\eps\Theta^*_\eps:\nabla^2_z\psi
\]
is the generator of the joint process. Expanding the coefficients according to system \eqref{sde0}, we decompose
\begin{align}\label{Dz1}
\sL_\eps
=\eps^{-1}\sL_0+\eps^{-2}\sL_1+\sL_2,
\end{align}
with
\begin{align*}
\sL_0\psi&:=D\cdot\nabla_z\psi+2(\sigma G^*):\nabla_x\nabla_y\psi,\\
\sL_1\psi&:=(\sigma\sigma^*):\nabla_x^2\psi+b\cdot\nabla_x\psi,\\
\sL_2\psi&:=(GG^*):\nabla_y^2\psi+F\cdot\nabla_y\psi.
\end{align*}
We emphasize that $\sL_1$ acts only on the fast variable $x$, whereas $\sL_2$ acts only on the slow variable $y$.

\medskip
Throughout the remainder of this section, we adopt the convention that $(\rho,\omega)\in\sW_d\times\sW_\vartheta$ denote generic polynomial weights whose precise values may change from line to line. This convention applies to all lemmas and proofs that follow.

\subsection{Convergence of the slow variable}

The main result of this subsection is the following uniform-in-time diffusion approximation for the slow component.

\bt\label{thm:slow-variable}
Under the assumptions of Theorem~\ref{main},  for any $\omega_0\in \sW_\vartheta$ and any $\gamma\in(\beta,1]$,
there exists a weight $\varrho\in\sW_{d+\vartheta}$ such that for every
$\varphi\in\bC^\gamma_{\omega_0}(\mR^\vartheta)$,   $\eps\in(0,1)$, and $z=(x,y)\in\mR^d\times\mR^\vartheta$,
\[
\sup_{t\geq 0}\Big|\mE\varphi(Y^\eps_t(z))-\mE\varphi(\bar Y_t(y))\Big|
\leq \varrho(z)\eps^\beta\|\varphi\|_{\bC^\gamma_{\omega_0}},
\]
where $\bar Y_t(y)$ solves the homogenized limit (\ref{sde000}).
\et

In the following, we always work under the assumptions of Theorem~\ref{main}.

\subsubsection{Estimates for the homogenized coefficients and semigroup}
We first prepare two preparatory results concerning the homogenized system (\ref{sde000}). The following lemma establishes the basic regularity of the corrector and the effective coefficients.

\bl\label{Le42}
Let $\cF,\cG$ and $\Gamma_1,\Gamma_2$ be defined by \eqref{bcf}-\eqref{bcf0}, respectively.
Then there exist weights $(\rho,\omega)\in\sW_d\times\sW_\vartheta$ such that
\begin{align*}
\Gamma_1,\Gamma_2\in\bC^\alpha_\rho\bC^\beta_\omega,
\quad
\cF,\cG\in \bC^\beta_\omega(\mR^\vartheta),
\end{align*}
and $\cG(y)$ is symmetric and positive semidefinite for every $y\in\mR^\vartheta$.
\el

\begin{proof}
\noindent
{\it (i) regularity of $\Gamma_1,\Gamma_2$.}
Recall that $\Phi$ solves the Poisson equation \eqref{pde10}.
By Theorem \ref{p1} and the equivalence \eqref{Equi2}, we have $\Phi\in\bC^{2+\alpha}_\rho\bC^{1+\beta}_\omega$
for some $(\rho,\omega)\in\sW_d\times\sW_\vartheta$.
By definition,
\[
\Gamma_1=D\cdot\nabla_z\Phi+2\sigma G^*:\nabla_x\nabla_y\Phi,
\quad
\Gamma_2=H\otimes\Phi+2G\sigma^*\cdot\nabla_x\Phi.
\]
Using the assumptions on $(b,\sigma,H,c,F,G)$, the product rule, and the closure properties of polynomial weights in $\sW$
(cf.\ \eqref{Da49}), we can choose $(\rho,\omega)\in\sW_d\times\sW_\vartheta$ such that
\(
\Gamma_1,\Gamma_2\in\bC^\alpha_\rho\bC^\beta_\omega.
\)

\smallskip
\noindent
{\it (ii) regularity of $\cF,\cG$.}
By definition \eqref{bcf},
\[
\cF(y)=\mu_y\big((F+\Gamma_1)(\cdot,y)\big),
\quad
\cG(y)=\mu_y\Big(\big(GG^*+(\Gamma_2+\Gamma_2^*)/2\big)(\cdot,y)\Big).
\]
By the mapping property encoded in \eqref{u12}, the regularity of $F,\Gamma_1,GG^*,\Gamma_2$ implies
$\cF,\cG\in\bC^\beta_\omega(\mR^\vartheta)$ for some $\omega\in\sW_\vartheta$.

\smallskip
\noindent
{\it (iii) symmetry and positivity of $\cG$.}
Symmetry is immediate. To show positivity, fix $\xi\in\mR^\vartheta$ and define
\[
\Phi_\xi(x,y):=\langle \Phi(x,y),\xi\rangle,\quad
G_\xi(x,y):=G(x,y)^*\xi,\quad
H_\xi(x,y):=\langle H(x,y),\xi\rangle.
\]
From the Poisson equation \eqref{pde10}, we have $\sL_1\Phi_\xi=-H_\xi$.
A direct computation gives
\[
\sL_1(\Phi_\xi^2)
=
2\Phi_\xi\,\sL_1\Phi_\xi+2|\sigma^*\nabla_x\Phi_\xi|^2
=
-2H_\xi\Phi_\xi+2|\sigma^*\nabla_x\Phi_\xi|^2.
\]
Integrating against $\mu_y$ and using $\sL_1^*\mu_y=0$ yields
\[
0=\mu_y(\sL_1(\Phi_\xi^2))
=
-2\mu_y(H_\xi\Phi_\xi)+2\mu_y(|\sigma^*\nabla_x\Phi_\xi|^2),
\]
hence
\begin{align}\label{KeyIdentityMu}
\mu_y(H_\xi\Phi_\xi)=\mu_y(|\sigma^*\nabla_x\Phi_\xi|^2)\geq0.
\end{align}
Now compute $\langle \xi,\cG(y)\xi\rangle$ using \eqref{bcf}--\eqref{bcf0}:
\[
\langle \xi,\cG(y)\xi\rangle
=
\mu_y(|G_\xi|^2)+\mu_y(H_\xi\Phi_\xi)+2\mu_y(\langle G_\xi,\sigma^*\nabla_x\Phi_\xi\rangle).
\]
Combining with \eqref{KeyIdentityMu} we obtain
\begin{align*}
\langle \xi,\cG(y)\xi\rangle
&=
\mu_y(|G_\xi|^2)+\mu_y(|\sigma^*\nabla_x\Phi_\xi|^2)+2\mu_y(\langle G_\xi,\sigma^*\nabla_x\Phi_\xi\rangle)\\
&=
\mu_y(|G_\xi+\sigma^*\nabla_x\Phi_\xi|^2)\geq0.
\end{align*}
Therefore $\cG(y)$ is positive semidefinite.
\end{proof}

The next lemma, which is a direct consequence of \cite[Corollary~4.10]{XXZ26}, quantifies the smoothing effect of the homogenized semigroup
\begin{align}\label{bart}
\bar \cT_t\varphi(y):=\mE\varphi(\bar Y_t(y)),\quad y\in \mR^\vartheta.
\end{align}
 Especially, the long-time estimates for the first and second derivative will be crucial  for our uniform-in-time analysis.

\bl
For any $\gamma\in(\beta,2)$ and $\omega_0\in\sW_\vartheta$,
there exist a decreasing function $\ell\in L^1(\mR_+)$ and weights $\omega_1,\omega_2\in\sW_\vartheta$ such that for all $t>0$,
\begin{align}\label{ny3}
\|\bar \cT_t\varphi\|_{\bC^{\beta}_{\omega_1}}
\leq \|\varphi\|_{\bC^\beta_{\omega_0}},
\quad
\|\nabla_y\bar \cT_t\varphi\|_{\bC^{\beta}_{\omega_2}}+\|\nabla^2_y\bar \cT_t\varphi\|_{\bC^{\beta}_{\omega_2}}
\leq \ell(t)\,\|\varphi\|_{\bC^\gamma_{\omega_0}}.
\end{align}
\el
\br
If $\gamma=\beta$, then the function $\ell(t)$ behaves like $t^{-1}$ as $t\to0$ and is therefore not integrable near zero; this explains the need for the strict inequality $\gamma>\beta$ in our main results.
\er

\subsubsection{A key uniform-in-time fluctuation estimate}

The following lemma provides a uniform bound for integrals involving centered fluctuations of the multiscale system. This is the core technical estimate that allows us to control the error in the diffusion approximation.

\bl\label{Le44}
For $(\rho_1,\omega_1),(\rho_2,\omega_2)\in\sW_d\times\sW_\vartheta$,
let $Q_1\in \bC^{\alpha}_{\rho_1}\bC^{\beta}_{\omega_1}$ and $Q_2\in \bC^{\alpha}_{\rho_2}\bC^{\beta}_{\omega_2}$
be $\mR^\vartheta$-valued and $\mR^\vartheta\otimes\mR^\vartheta$-valued functions on $\mR^d\times\mR^\vartheta$, respectively, satisfying  the centering condition
\begin{align}\label{Dk3}
\mu_y(Q_j(\cdot,y))=0,\quad \forall\,y\in\mR^\vartheta,\ \ j=1,2.
\end{align}
For any $\omega_0\in\sW_\vartheta$ and $\gamma\in(\beta,2)$,
there exist weights $\varrho_1,\varrho_2\in\sW_{d+\vartheta}$ such that for all $\eps\in(0,1)$, all $z\in\mR^{d+\vartheta}$,
and all $\varphi\in\bC^\gamma_{\omega_0}$,
\begin{align}\label{DQ41}
\sup_{t\geq 0}\left|\mE\int^t_0(Q_j:\nabla^j_y\bar\cT_{t-s}\varphi)(Z^\eps_s(z))\dif s\right|
\leq \eps^\beta\varrho_j(z)\|\varphi\|_{\bC^\gamma_{\omega_0}},\quad j=1,2,
\end{align}
where $:$ denotes the natural contraction between vectors (for $j=1$) or matrices (for $j=2$), and $\bar\cT_{t}\varphi$ is given by \eqref{bart}.
\el

\begin{proof}
We provide the proof of \eqref{DQ41} for $j=2$. The case $j=1$ is completely analogous and in fact  simpler.

\medskip
\noindent
{\it Step 1: the Poisson corrector and mollification.}
Let $U:\mR^d\times\mR^\vartheta\to\mR^\vartheta\otimes\mR^\vartheta$ solve the Poisson equation (in $x$ with parameter $y$)
\begin{align}\label{Poi1}
\sL_1 U = Q_2.
\end{align}
By \eqref{Dk3}, \eqref{Equi2} and Theorem~\ref{p1}, there exist weights $(\rho,\omega)\in\sW_d\times\sW_\vartheta$ such that
\begin{align}\label{UU8}
\|U\|_{\bC^{2+\alpha}_{\rho}\bC^{\beta}_{\omega}}
\leq
\|Q_2\|_{\bC^{\alpha}_{\rho_2}\bC^\beta_{\omega_2}}.
\end{align}
Fix $\varphi\in\bC^\gamma_{\omega_0}$ and $t>0$.
For $s\in(0,t)$ define
\[
f(s,z):=U(z):\nabla_y^2 \bar\cT_{t-s}\varphi(y),\quad z=(x,y).
\]
To justify the application of It\^o's formula (and the appearance of $\nabla_y^3$-terms below), we mollify in the $y$-variable.
Let $U_n$ be the mollification of $U$ in $y$ and similarly for $\nabla_y^2\bar\cT_{t-s}\varphi$.
We set
\[
f_n(s,z):=U_n(z):(\nabla_y^2 \bar\cT_{t-s}\varphi)_n(y).
\]
Since $\p_r\bar\cT_r\varphi=\bar\sL\,\bar\cT_r\varphi$, we have
\[
\p_s(\nabla_y^2\bar\cT_{t-s}\varphi)
=
-\nabla_y^2\bar\sL\,\bar\cT_{t-s}\varphi,
\]
and therefore
\[
\p_s f_n(s,z)
=-
U_n(z):(\nabla_y^2\bar\sL\,\bar\cT_{t-s}\varphi)_n(y).
\]

\noindent
{\it Step 2: applying It\^o and isolating the main term.}
Applying  \eqref{Ito1} to $\psi=f_n$ and using the generator decomposition \eqref{Dz1} we get
\begin{align}\label{JH1}
\mE f_n(t, Z^\eps_{t}(z))
&=f_n(0, z)-\mE\int^{t}_0\Big[U_n:(\nabla_y^2\bar\sL \bar\cT_{t-s}\varphi)_n\Big](Z^\eps_s(z))\dif s \no\\
&\quad+\mE\int^{t}_0(\eps^{-1}\sL_0+\eps^{-2}\sL_1+\sL_2)f_n(s, Z^\eps_s(z))\dif s.
\end{align}
Note that $\sL_1$ acts only in $x$, hence
\[
\sL_1 f(s,z)
=
(\sL_1U)(z):\nabla_y^2\bar\cT_{t-s}\varphi(y)
=
Q_2(z):\nabla_y^2\bar\cT_{t-s}\varphi(y),
\]
where we used \eqref{Poi1}.
Multiply \eqref{JH1} by $\eps^2$ and rearrange terms gives
\begin{align*}
\mE\int_0^{t}\big(Q_2:\nabla_y^2 \bar\cT_{t-s}\varphi\big)(Z_s^\eps(z))\dif s
&=
\eps^2\Big(\mE f_n(t, Z^\eps_{t}(z))-f_n(0, z)\Big)\\
&+\eps^2\mE\int^{t}_0\Big[U_n:(\nabla_y^2\bar\sL \bar\cT_{t-s}\varphi)_n\Big](Z^\eps_s(z))\dif s\\
&-\mE\int^{t}_0(\eps \sL_0+\eps^2\sL_2)f_n(s, Z^\eps_s(z))\dif s\\
&+\mE\int^{t}_0\big(\sL_1f-\sL_1f_n\big)(s, Z^\eps_s(z))\dif s\\
&=:\cI^{\eps,n}_1(t)+\cI^{\eps,n}_2(t)+\cI^{\eps,n}_3(t)+\cI^{\eps,n}_4(t).
\end{align*}
Below, we bound each $\cI^{\eps,n}_k(t)$ uniformly in $t$.

\medskip
\noindent
{\it Step 3: estimate of $\cI^{\eps,n}_1(t)$ (boundary terms).}
By the mollifier estimate \eqref{App2}, \eqref{UU8} and \eqref{ny3}, there exists $\varrho\in\sW_{d+\vartheta}$
and $\omega\in\sW_\vartheta$ such that
\[
|U_n(z)|\lesssim \varrho(z),
\quad
|\nabla_y^2\psi_n(y)|\lesssim n^{2-\beta}\omega(y)\|\psi\|_{\bC^\beta_{\omega_0}}
\quad (\psi=\varphi\ \text{or}\ \psi=\bar\cT_t\varphi).
\]
Consequently,
\[
|f_n(t,z)|=|U_n(z):(\nabla_y^2\varphi)_n(y)|
\leq n^{2-\beta}\varrho(z)\|\varphi\|_{\bC^\beta_{\omega_0}},
\]
and
\[
|f_n(0,z)|
=|U_n(z):(\nabla_y^2\bar\cT_t\varphi)_n(y)|
\leq n^{2-\beta}\varrho(z)\|\varphi\|_{\bC^\beta_{\omega_0}}.
\]
Using the uniform moment bound \eqref{Mom1}, we can choose $\varrho\in\sW_{d+\vartheta}$ so that
for all $t>0$,
\[
|\cI^{\eps,n}_1(t)|
\lesssim
\eps^2 n^{2-\beta}\Big(\mE\varrho_1(Z^\eps_t(z))+\varrho_1(z)\Big)\|\varphi\|_{\bC^\beta_{\omega_0}}
\leq
\eps^2 n^{2-\beta}\varrho(z)\|\varphi\|_{\bC^\beta_{\omega_0}}.
\]

\noindent
{\it Step 4: estimate of $\cI^{\eps,n}_2(t)$ (time derivative term).}
We first bound the factor involving $(\nabla_y^2\bar\sL\bar\cT_{t-s}\varphi)_n$.
By \eqref{App2},
\[
\|(\nabla_y^2\bar\sL\bar\cT_{t-s}\varphi)_n\|_{\bC^0_\omega}
\lesssim
n^{2-\beta}\|\bar\sL\bar\cT_{t-s}\varphi\|_{\bC^\beta_\omega}.
\]
Using the coefficient regularity, \eqref{Ho11} and \eqref{ny3},
\[
\|\bar\sL\bar\cT_{t-s}\varphi\|_{\bC^\beta_\omega}
\lesssim
\|\nabla_y^2\bar\cT_{t-s}\varphi\|_{\bC^\beta_{\omega_2}}
\lesssim
\ell(t-s)\|\varphi\|_{\bC^\gamma_{\omega_0}}.
\]
Therefore,
\[
\|(\nabla_y^2\bar\sL\bar\cT_{t-s}\varphi)_n\|_{\bC^0_\omega}
\lesssim
n^{2-\beta}\ell(t-s)\|\varphi\|_{\bC^\gamma_{\omega_0}}.
\]
Combining this with $|U_n|\lesssim \varrho_1$ and \eqref{Mom1}, we obtain
\begin{align*}
|\cI^{\eps,n}_2(t)|
&\lesssim
\eps^2 n^{2-\beta}\|\varphi\|_{\bC^\gamma_{\omega_0}}
\int_0^t \ell(t-s)\mE\varrho_1(Z_s^\eps(z))\dif s\\
&\leq
\eps^2 n^{2-\beta}\varrho(z)\|\varphi\|_{\bC^\gamma_{\omega_0}}\int_0^t\ell(t-s)\dif s\\
&\leq
\eps^2 n^{2-\beta}\varrho(z)\|\ell\|_{L^1}\|\varphi\|_{\bC^\gamma_{\omega_0}}.
\end{align*}

\noindent
{\it Step 5: estimate of $\cI^{\eps,n}_3(t)$ (terms involving $\sL_0$ and $\sL_2$).}
A direct computation of the derivatives of $f_n$ gives that for $i=1,\dots,d$,
\[
\p_{x_i}f_n(s,z)=\p_{x_i}U_n(z):(\nabla_y^2 \bar\cT_{t-s}\varphi)_n(y),
\]
and for $j=1,\dots,\vartheta$,
\begin{align*}
\p_{y_j}f_n(s,z)
&=
U_n(z):\p_{y_j}(\nabla_y^2 \bar\cT_{t-s}\varphi)_n(y)
+
\p_{y_j}U_n(z):(\nabla_y^2 \bar\cT_{t-s}\varphi)_n(y),\\
\p_{y_j}\p_{x_i}f_n(s,z)
&=
\p_{x_i}U_n(z):\p_{y_j}(\nabla_y^2 \bar\cT_{t-s}\varphi)_n(y)
+
\p_{y_j}\p_{x_i}U_n(z):(\nabla_y^2 \bar\cT_{t-s}\varphi)_n(y).
\end{align*}
Using \eqref{UU8}, \eqref{App2}, and \eqref{ny3}, we get for some weights $\varrho_1,\varrho_2\in\sW_{d+\vartheta}$,
\[
\|\nabla_z f_n(s)\|_{\bC^0_{\varrho_1}}
+
\|\nabla_x\nabla_y f_n(s)\|_{\bC^0_{\varrho_1}}
\leq
n^{1-\beta}\ell(t-s)\|\varphi\|_{\bC^\gamma_{\omega_0}},
\]
and
\[
\|\nabla_y^2 f_n(s)\|_{\bC^0_{\varrho_2}}
\leq
n^{2-\beta}\ell(t-s)\|\varphi\|_{\bC^\gamma_{\omega_0}}.
\]
Recalling
$$
\sL_0 f_n=D\cdot\nabla_z f_n+2(\sigma G^*):\nabla_x\nabla_y f_n
$$
and
$$
\sL_2 f_n=GG^*:\nabla_y^2 f_n+F\cdot\nabla_y f_n,
$$
we deduce (for suitable weights $\varrho_3,\varrho_4$)
\[
\|\sL_0 f_n(s)\|_{\bC^0_{\varrho_3}}
\leq
n^{1-\beta}\ell(t-s)\|\varphi\|_{\bC^\gamma_{\omega_0}},
\]
and
\[
\|\sL_2 f_n(s)\|_{\bC^0_{\varrho_4}}
\leq
n^{2-\beta}\ell(t-s)\|\varphi\|_{\bC^\gamma_{\omega_0}}.
\]
Therefore, using \eqref{Mom1} exactly as in Step~4,
\[
|\cI^{\eps,n}_3(t)|
\leq
(\eps n^{1-\beta}+\eps^2 n^{2-\beta})\varrho(z)\|\ell\|_{L^1}\|\varphi\|_{\bC^\gamma_{\omega_0}}.
\]

\noindent
{\it Step 6: estimate of $\cI^{\eps,n}_4(t)$ (mollification error in $\sL_1$).}
We estimate
\[
\sL_1 f_n-\sL_1 f
=
(\sL_1U_n):(\nabla_y^2\bar\cT_{t-s}\varphi)_n-(\sL_1U):\nabla_y^2\bar\cT_{t-s}\varphi.
\]
Split it as
\begin{align*}
|\sL_1 f_n-\sL_1 f|
&\leq
|\sL_1U_n-\sL_1U|\cdot|(\nabla_y^2\bar\cT_{t-s}\varphi)_n|
+
|\sL_1U|\cdot|(\nabla_y^2\bar\cT_{t-s}\varphi)_n-\nabla_y^2\bar\cT_{t-s}\varphi|.
\end{align*}
By \eqref{Ho11}, \eqref{App1} and the $y$-regularity in \eqref{UU8}, we have
\[
\|\sL_1U_n-\sL_1U\|_{\bC^0_{\rho_3}\bC^0_{\omega_3}}
\leq
\|U_n-U\|_{\bC^2_{\rho}\bC^0_\omega}
\lesssim
n^{-\beta}\|U\|_{\bC^{2}_{\rho}\bC^\beta_\omega}
\stackrel{\eqref{UU8}}{\leq}
n^{-\beta}\|Q_2\|_{\bC^\alpha_{\rho_2}\bC^\beta_{\omega_2}}.
\]
Also, again by \eqref{App1} and \eqref{ny3},
\[
\|(\nabla_y^2 \bar\cT_{t-s}\varphi)_n-\nabla_y^2 \bar\cT_{t-s}\varphi\|_{\bC^0_\omega}
\lesssim
n^{-\beta}\|\nabla_y^2 \bar\cT_{t-s}\varphi\|_{\bC^\beta_{\omega_2}}
\leq
n^{-\beta}\ell(t-s)\|\varphi\|_{\bC^\gamma_{\omega_0}}.
\]
Combining these bounds yields, for some weight $\varrho_0\in\sW_{d+\vartheta}$,
\[
|\sL_1 f_n(s,z)-\sL_1 f(s,z)|
\lesssim
n^{-\beta}\ell(t-s)\varrho_0(z)\|\varphi\|_{\bC^\gamma_{\omega_0}}.
\]
Integrating in $s$ and using \eqref{Mom1}, we obtain
\[
|\cI^{\eps,n}_4(t)|
\lesssim
n^{-\beta}\varrho(z)\|\ell\|_{L^1}\|\varphi\|_{\bC^\gamma_{\omega_0}}.
\]

\noindent
{\it Step 7: optimizing in $n$.}
Combining Steps~3--6, we conclude that for all $t>0$,
\[
\left|\mE\int_0^{t}\big(Q_2:\nabla_y^2 \bar\cT_{t-s}\varphi\big)(Z_s^\eps(z))\dif s\right|
\leq
\Big(\eps^2 n^{2-\beta}+\eps n^{1-\beta}+n^{-\beta}\Big)\varrho(z)\|\varphi\|_{\bC^\gamma_{\omega_0}}.
\]
The optimal choice $n = \eps^{-1}$ balances the three error contributions and yields the desired $\eps^\beta$ rate.
\end{proof}

\subsubsection{The $H$-term and the effective operator}
The next lemma handles the   fluctuations involving the  singular term $\eps^{-1}H\cdot\nabla_y\bar\cT_{t-s}\varphi$, which is the most delicate part of the expansion.

\bl\label{Le45}
For $\omega_0\in\sW_\vartheta$ and $\gamma\in(\beta,2]$,
there exists a weight $\varrho\in\sW_{d+\vartheta}$ such that for all $\eps\in(0,1)$ and all $\varphi\in\bC^\gamma_{\omega_0}$ and $z\in\mR^d\times\mR^\vartheta$,
\begin{align}\label{DQ1}
\sup_{t\geq 0}\left|\mE\int^t_0\left[\eps^{-1}H\cdot\nabla_y\bar\cT_{t-s}\varphi-\widetilde\sL\,\bar\cT_{t-s}\varphi\right](Z^\eps_s(z))\dif s\right|
\leq \eps^\beta\varrho(z)\|\varphi\|_{\bC^\gamma_{\omega_0}},
\end{align}
where, for $\Gamma_1,\Gamma_2$ given by \eqref{bcf0},
\begin{align*}
\widetilde\sL f:=\mu_y(\Gamma_1(\cdot,y))\cdot\nabla_yf+\mu_y(\Gamma_2(\cdot,y)):\nabla^2_y f.
\end{align*}
\el

\begin{proof}
The strategy parallels Lemma~\ref{Le44}, but now the corrector is the Poisson solution $\Phi$ from \eqref{pde10}.

\medskip
\noindent
{\it Step 1: object decomposition.}
Fix $\varphi\in\bC^\gamma_{\omega_0}$ and $t>0$. For $s\in(0,t]$ set
\[
f(s,z):=\Phi(z)\cdot\nabla_y \bar\cT_{t-s}\varphi(y),\quad z=(x,y),
\]
and mollify in $y$:
\[
f_n(s,z):=\Phi_n(z)\cdot(\nabla_y \bar\cT_{t-s}\varphi)_n(y).
\]
Since $\p_r\bar\cT_r\varphi=\bar\sL\bar\cT_r\varphi$, we have
\[
\p_s f_n(s,z)=-\Phi_n(z)\cdot(\nabla_y\bar\sL \bar\cT_{t-s}\varphi)_n(y).
\]
Applying \eqref{Ito1} to $\psi=f_n$ and using \eqref{Dz1} yields
\begin{align}\label{JH91}
\mE f_n(t, Z^\eps_{t}(z))
&=f_n(0, z)-\mE\int^{t}_0\left[\Phi_n\cdot(\nabla_y\bar\sL \bar\cT_{t-s}\varphi)_n\right](Z^\eps_s(z))\dif s\no\\
&\quad+\mE\int^{t}_0(\eps^{-1}\sL_0+\sL_2+\eps^{-2}\sL_1)f_n(s, Z^\eps_s(z))\dif s.
\end{align}
Since $\sL_1\Phi=-H$, we have
\[
\sL_1 f(s,z)=(\sL_1\Phi)(z)\cdot\nabla_y\bar\cT_{t-s}\varphi(y)=-H(z)\cdot\nabla_y\bar\cT_{t-s}\varphi(y).
\]
Multiplying \eqref{JH91} by $\eps$ and rearranging, we obtain
\begin{align*}
&\mE\int^t_0\left[\eps^{-1}H\cdot\nabla_y\bar\cT_{t-s}\varphi-\widetilde\sL\,\bar\cT_{t-s}\varphi\right](Z^\eps_s(z))\dif s\\
&\quad=
\eps\Big(f_n(0, z)-\mE f_n(t, Z^\eps_{t}(z))\Big)
+\eps^{-1}\mE\int^{t}_0(\sL_1f_n-\sL_1f)(s, Z^\eps_s)\dif s\\
&\qquad+\eps\mE\int^{t}_0\left[\sL_2f_n(s)-\Phi_n\cdot(\nabla_y\bar\sL \bar\cT_{t-s}\varphi)_n\right](Z^\eps_s(z))\dif s\\
&\qquad+\mE\int^{t}_0\left[\sL_0f_n(s)-\widetilde\sL\,\bar \cT_{t-s}\varphi\right] (Z^\eps_s(z))\dif s\\
&\quad=:\cI^{\eps,n}_1(t)+\cI^{\eps,n}_2(t)+\cI^{\eps,n}_3(t)+\cI^{\eps,n}_4(t).
\end{align*}
\noindent
{\it Step 2: estimate of each terms.}
As in Lemma~\ref{Le44}, by \eqref{App2}, \eqref{ny3}, and \eqref{Mom1}, we deduce that
\[
|\cI^{\eps,n}_1(t)|
\leq
\eps n^{1-\beta}\varrho(z)\|\varphi\|_{\bC^\beta_{\omega_0}}.
\]
Using the same splitting as in Step~6 of Lemma~\ref{Le44}, together with \eqref{App1} and $\Phi\in\bC^{2+\alpha}_\rho\bC^{1+\beta}_\omega$,
we obtain
\[
\|\sL_1f_n(s)-\sL_1f(s)\|_{\bC^0_\varrho}
\leq
n^{-1-\beta}\ell(t-s)\|\varphi\|_{\bC^\gamma_{\omega_0}}.
\]
Consequently,
\[
|\cI^{\eps,n}_2(t)|
\leq
\eps^{-1} n^{-1-\beta}\varrho(z)\|\ell\|_{L^1}\|\varphi\|_{\bC^\gamma_{\omega_0}}.
\]
Differentiating $f_n=\Phi_n\cdot(\nabla_y\bar\cT_{t-s}\varphi)_n$ and using \eqref{App2} together with \eqref{ny3} gives
\[
\|\nabla^j_y f_n(s)\|_{\bC^0_\varrho}\lesssim n^{j-\beta}\ell(t-s)\|\varphi\|_{\bC^\gamma_{\omega_0}},\ j=1,2.
\]
Consequently,
\[
\|\sL_2 f_n(s)\|_{\bC^0_\varrho}\lesssim n^{1-\beta}\ell(t-s)\|\varphi\|_{\bC^\gamma_{\omega_0}},
\]
and similarly
\[
\|\Phi_n\cdot(\nabla_y\bar\sL\bar\cT_{t-s}\varphi)_n\|_{\bC^0_\varrho}
\lesssim
n^{1-\beta}\ell(t-s)\|\varphi\|_{\bC^\gamma_{\omega_0}}.
\]
Hence, using \eqref{Mom1}, we have for some $\varrho\in\sW_{d+\vartheta}$,
\[
|\cI^{\eps,n}_3(t)|
\leq\eps n^{1-\beta}\varrho(z)\|\ell\|_{L^1}\|\varphi\|_{\bC^\gamma_{\omega_0}}.
\]
Finally,
recall that
\[
\sL_0f_n=D\cdot\nabla_zf_n+2(\sigma G^*):(\nabla_x\nabla_y f_n).
\]
A direct computation yields
\[
\sL_0 f_n
=
\Gamma^{(n)}_1\cdot(\nabla_y \bar\cT_{t-s}\varphi)_n
+
\Gamma^{(n)}_2:\nabla^2_y(\bar\cT_{t-s}\varphi)_n,
\]
where
\[
\Gamma^{(n)}_1:=D\cdot \nabla_z\Phi_n+2(\sigma G^*):\nabla_x\nabla_y\Phi_n,
\quad
\Gamma^{(n)}_2:=H\otimes \Phi_n+2G\sigma^*\cdot\nabla_x \Phi_n.
\]
By definition \eqref{bcf0}, the pointwise limits are $\Gamma^{(n)}_1\to\Gamma_1$ and $\Gamma^{(n)}_2\to\Gamma_2$ as $n\to\infty$,
with quantitative rate $n^{-\beta}$ in weighted H\"older norms by \eqref{App1}.
Thus we decompose
\begin{align*}
\sL_0 f_n-\widetilde\sL\,\bar\cT_{t-s}\varphi
&=
\big(\Gamma^{(n)}_1\cdot(\nabla_y \bar\cT_{t-s}\varphi)_n-\Gamma_1\cdot\nabla_y \bar\cT_{t-s}\varphi\big)
+\big(\Gamma_1-\mu_y(\Gamma_1(\cdot,y))\big)\cdot\nabla_y \bar\cT_{t-s}\varphi\\
&+
\big(\Gamma^{(n)}_2:\nabla_y^2(\bar\cT_{t-s}\varphi)_n-\Gamma_2:\nabla_y^2 \bar\cT_{t-s}\varphi\big)
+\big(\Gamma_2-\mu_y(\Gamma_2(\cdot,y))\big):\nabla_y^2 \bar\cT_{t-s}\varphi\\
&=:\sJ^{\eps,n}_{1,s}+\sJ_{2,s}+\sJ^{\eps,n}_{3,s}+\sJ_{4,s}.
\end{align*}
The terms $\sJ^{\eps,n}_{1,s}$ and $\sJ^{\eps,n}_{3,s}$ are pure mollification errors; by \eqref{App1}--\eqref{App2} and \eqref{ny3},
\[
\mE\int_0^t \big|\sJ^{\eps,n}_{1,s}+\sJ^{\eps,n}_{3,s}\big|(Z_s^\eps(z))\dif s
\lesssim
n^{-\beta}\varrho(z)\|\varphi\|_{\bC^\gamma_{\omega_0}}.
\]
For $\sJ_{2,s}$ and $\sJ_{4,s}$, define the centered fields
\[
Q_1:=\Gamma_1-\mu_y(\Gamma_1(\cdot,y)),
\quad
Q_2:=\Gamma_2-\mu_y(\Gamma_2(\cdot,y)).
\]
By Lemma~\ref{Le42} and \eqref{u12}, for some $(\rho,\omega)$ we have $Q_1,Q_2\in\bC^\alpha_\rho\bC^\beta_\omega$ and
$\mu_y(Q_j(\cdot,y))=0$. Hence Lemma~\ref{Le44} applies and yields
\[
\mE\int_0^t \big|\sJ_{2,s}+\sJ_{4,s}\big|(Z_s^\eps(z))\dif s
\lesssim
\eps^\beta\varrho(z)\|\varphi\|_{\bC^\gamma_{\omega_0}}.
\]
Combining the four pieces gives
\[
|\cI_4^{\eps,n}(t)|
\lesssim
\big(n^{-\beta}+\eps^\beta\big)\varrho(z)\|\varphi\|_{\bC^\gamma_{\omega_0}}.
\]

\medskip
\noindent
{\it Step 3: collecting estimates and optimizing.}
Combining the above computations, we arrive at
\begin{align*}
\sup_{t\geq0}
&\left|\mE\int^t_0\Big[\eps^{-1}H\cdot\nabla_y\bar\cT_{t-s}\varphi-\widetilde\sL\,\bar\cT_{t-s}\varphi\Big](Z^\eps_s(z))\dif s\right|\\
&\quad\leq
\Big(\eps n^{1-\beta}+\eps^{-1}n^{-1-\beta}+n^{-\beta}+\eps^\beta\Big)\varrho(z)\|\varphi\|_{\bC^\gamma_{\omega_0}}.
\end{align*}
Choosing $n=\eps^{-1}$ yields the desired \eqref{DQ1}.
\end{proof}

\subsubsection{Proof of the slow convergence theorem}
With the lemmas established above, we now proceed to give:

\begin{proof}[\bf Proof of Theorem~\ref{thm:slow-variable}]
Fix $t>0$ and apply \eqref{Ito1} with $\psi(s,z):=\bar\cT_{t-s}\varphi(y)$.
Since $\psi$ depends only on $y$, we have $\sL_1\psi=0$ and
\[
\sL_0\psi=H\cdot\nabla_y\psi,
\quad
\sL_2\psi=GG^*:\nabla_y^2\psi+F\cdot\nabla_y\psi.
\]
Moreover, $\p_s\bar\cT_{t-s}\varphi=-\bar\sL\,\bar\cT_{t-s}\varphi$, i.e.
\[
\p_s\bar\cT_{t-s}\varphi
=
-\Big(\cG:\nabla^2_y\bar\cT_{t-s}\varphi+\cF\cdot\nabla_y\bar\cT_{t-s}\varphi\Big).
\]
Inserting these expressions into \eqref{Ito1} yields
\begin{align*}
\mE\varphi(Y^\eps_t(z))-\bar\cT_{t}\varphi(y)
&= \mE\int^t_0\Big(GG^*(Z^\eps_s)-\cG(Y^\eps_s)\Big):\nabla^2_y\bar\cT_{t-s}\varphi(Y^\eps_s)\dif s\\
&\quad + \mE\int^t_0\Big(\eps^{-1}H(Z^\eps_s)+F(Z^\eps_s)-\cF(Y^\eps_s)\Big)\cdot\nabla_y\bar\cT_{t-s}\varphi(Y^\eps_s)\dif s.
\end{align*}
We further decompose the second line into centered fluctuations and the $\widetilde\sL$-term:
\[
\mE\varphi(Y^\eps_t(z))-\bar\cT_{t}\varphi(y) = \cR^\eps_1(t) + \cR^\eps_2(t) + \cR^\eps_3(t),
\]
where
\begin{align*}
\cR^\eps_1(t) &:= \mE\int^t_0\big(\widetilde G:\nabla^2_y\bar\cT_{t-s}\varphi\big)(Z^\eps_s)\dif s,\\
\cR^\eps_2(t) &:= \mE\int^t_0\big[\widetilde F\cdot\nabla_y\bar\cT_{t-s}\varphi\big](Z^\eps_s)\dif s,\\
\cR^\eps_3(t) &:= \mE\int^t_0 \left[\eps^{-1}H\cdot\nabla_y - \widetilde\sL\right]\bar \cT_{t-s}\varphi(Z^\eps_s)\dif s,
\end{align*}
with
\[
\widetilde F(z):=F(z)-\mu_y(F(\cdot,y)),
\quad
\widetilde G(z):=GG^*(z)-\mu_y(GG^*(\cdot,y)),
\]
and
$$
\widetilde\sL:=\mu_y(\Gamma_1(\cdot,y))\cdot\nabla_y+\mu_y(\Gamma_2(\cdot,y)):\nabla^2_y .
$$
Applying Lemma~\ref{Le44}  to $Q_1=\widetilde F$ and $Q_2=\widetilde G$, we obtain
\[
\sup_{t\geq0}\big(|\cR^\eps_1(t)|+|\cR^\eps_2(t)|\big)
\leq
\eps^\beta\varrho(z)\|\varphi\|_{\bC^\gamma_{\omega_0}}.
\]
By Lemma~\ref{Le45}, we also have
\[
\sup_{t\geq0}|\cR^\eps_3(t)|
\leq
\eps^\beta\varrho(z)\|\varphi\|_{\bC^\gamma_{\omega_0}}.
\]
Combining these bounds completes the proof.
\end{proof}

\subsection{Limit of the joint distribution}

We now upgrade the convergence of the slow variable to a full description of the joint law, completing the proof of Theorem \ref{main}.

\begin{proof}[\bf Proof of Theorem \ref{main}]
Let $(\rho_0,\omega_0)\in \sW_d\times\sW_\vartheta$ and fix $\gamma\in(\beta,2)$.
Take $\varphi\in\bC^\alpha_{\rho_0}\bC^\gamma_{\omega_0}$ and define
\begin{align*}
\bar\varphi(y):=\mu_y(\varphi(\cdot,y)),\quad
\widetilde\varphi(x,y):=\varphi(x,y)-\bar\varphi(y).
\end{align*}
Fix $t\geq0$ and $z=(x,y)$. Let $X_t^y(x)$ solve the frozen equation \eqref{sde2}.
We decompose the error into three parts:
\begin{align*}
\mE\varphi(Z^\eps_t(z))-\mE\bar\varphi(\bar Y_t(y))
&=\Big(\mE\bar\varphi(Y_t^\eps(z))-\mE\bar\varphi(\bar Y_t(y))\Big)
+\mE\big[\widetilde\varphi(X^y_{t/\eps^2}(x),y)\big]\\
&\quad+\Big(\mE\widetilde\varphi(Z_t^\eps(z))-\mE\big[\widetilde\varphi(X^y_{t/\eps^2}(x),y)\big]\Big)\\
&=:\sI^\eps_1(t,z)+\sI^\eps_2(t,z)+\sI^\eps_3(t,z).
\end{align*}
Choose $\eta\in(\beta,1\wedge\gamma)$. Since $b,\sigma\in\bC^\alpha_{\rm p}\bC^{1+\beta}_{\rm p}$, Theorem~\ref{p1} (applied with parameter regularity exponent $1+\beta$) together with \eqref{Equi2} implies that there exists $\bar\omega\in\sW_\vartheta$ such that
\[
\|\bar\varphi\|_{\bC^\eta_{\bar\omega}}
\lesssim
\|\varphi\|_{\bC^\alpha_{\rho_0}\bC^\eta_{\omega_0}}
\lesssim
\|\varphi\|_{\bC^\alpha_{\rho_0}\bC^\gamma_{\omega_0}}.
\]
Therefore, by Theorem~\ref{thm:slow-variable}, there exist $\varrho_1,\varrho_2\in\sW_{d+\vartheta}$ such that
\[
|\sI^\eps_1(t,z)|
\leq
\eps^\beta\varrho_1(z)\|\bar\varphi\|_{\bC^\eta_{\bar\omega}}
\lesssim
\eps^\beta\varrho_2(z)\|\varphi\|_{\bC^\alpha_{\rho_0}\bC^\gamma_{\omega_0}}.
\]
By \eqref{xy}, there exists $\varrho\in\sW_{d+\vartheta}$ such that
\[
|\sI^\eps_2(t,z)|
\leq\ell_0(t/\eps^2)\varrho(z)\|\widetilde\varphi\|_{\bC^0_{\rho_0}\bC^0_{\omega_0}}.
\]
We proceed to estimate the remaining term $\sI^\eps_3(t,z)$.
Define the frozen semigroup action on the centered test function:
\[
u(t,x,y):=\cT^y_t\big(\widetilde\varphi(\cdot,y)\big)(x)
=\mE\big[\widetilde\varphi(X^y_t(x),y)\big].
\]
Then $\p_t u=\sL_1 u$ and $u(0,\cdot)=\widetilde\varphi$.
For fixed $t>0$ and $\eps\in(0,1)$, set
\[
\widetilde u(s,z):=u((t-s)/\eps^2,z),\quad s\in[0,t].
\]
A direct calculation shows that
\[
\sI^\eps_3(t,z)=\mE\widetilde u(t,Z_t^\eps)-\widetilde u(0,z).
\]
Mollify $\widetilde u$ in the slow variable:
\[
\widetilde u_n(s,x,y):=\big[\widetilde u(s,x,\cdot)*\phi_n\big](y).
\]
Since $\p_s\widetilde u=-\eps^{-2}\sL_1u((t-s)/\eps^2,\cdot)$, we have
\[
\p_s\widetilde u_n(s,z)
=
-\eps^{-2}(\sL_1u)_n((t-s)/\eps^2,z).
\]
Applying It\^o's formula \eqref{Ito1} to $\psi=\widetilde u_n$ and using \eqref{Dz1}, we obtain
\begin{align*}
\sI^\eps_3(t,z)
&=
\mE\big[(\widetilde u-\widetilde u_n)(t,Z^\eps_t)\big]+\big(\widetilde u_n(0,z)-\widetilde u(0,z)\big)\\
&\quad-\mE\int^{t}_0(\eps^{-1}\sL_0+\sL_2)\widetilde u_n(s, Z^\eps_s)\dif s\\
&\quad+\eps^{-2}\mE\int^{t}_0\Big[(\sL_1 u)_n-\sL_1u_n\Big]((t-s)/\eps^2, Z^\eps_s)\dif s\\
&=:\sI^{\eps,n}_{31}(t)+\sI^{n}_{32}+\sI^{\eps,n}_{33}(t)+\sI^{\eps,n}_{34}(t).
\end{align*}
We estimate each of these four terms separately.

\smallskip
\noindent
{\it (i) mollification errors $\sI^{\eps,n}_{31}$ and $\sI^{n}_{32}$.}
From the mollifier estimate \eqref{App1} and the regularity of $\widetilde\varphi$ ensured by \eqref{u12}, we have
\[
\|\widetilde u_n(t)-\widetilde u(t)\|_{\bC^0_\rho\bC^0_\omega}
=\|(\widetilde \varphi)_n-\widetilde \varphi\|_{\bC^0_\rho\bC^0_\omega}
\leq
n^{-\beta}\|\widetilde\varphi\|_{\bC^0_{\rho_0}\bC^\beta_{\omega_0}}
\lesssim
n^{-\beta}\|\varphi\|_{\bC^\alpha_{\rho_0}\bC^\beta_{\omega_0}}.
\]
Similarly, using \eqref{App1} together with the semigroup estimate \eqref{cp0} in Theorem~\ref{cp},
\begin{align*}
\|\widetilde u_n(0)-\widetilde u(0)\|_{\bC^0_{\rho}\bC^0_{\omega}}
&=
\|u_n(t/\eps^2)-u(t/\eps^2)\|_{\bC^0_{\rho}\bC^0_{\omega}}
\lesssim n^{-\beta}\|u(t/\eps^2)\|_{\bC^0_{\rho}\bC^\beta_{\omega}}
\leq
 n^{-\beta}\|\varphi\|_{\bC^\alpha_{\rho_0}\bC^\beta_{\omega_0}}.
\end{align*}
Thus, by the uniform moment estimate \eqref{Mom1}, we obtain for some $\varrho\in\sW_{d+\vartheta}$,
\[
|\sI^{\eps,n}_{31}(t)|+|\sI^{n}_{32}|
\leq
n^{-\beta}\varrho(z)\|\varphi\|_{\bC^\alpha_{\rho_0}\bC^\beta_{\omega_0}}.
\]

\smallskip
\noindent
{\it (ii) estimate of the $\sL_0,\sL_2$ terms $\sI^{\eps,n}_{33}(t)$.}
By \eqref{cp1} in Theorem~\ref{cp} and Remark \ref{Re35},
there exists a decreasing function $h\in L^1([0,\infty))$ such that for all $r>0$,
\begin{align}\label{Sq1}
\|u(r)\|_{\bC^2_\rho\bC^\beta_\omega}\leq h(r)\|\varphi\|_{\bC^\alpha_{\rho_0}\bC^\beta_{\omega_0}}.
\end{align}
Combining this with the mollifier estimate \eqref{App2} gives, for some $\varrho_1\in\sW_{d+\vartheta}$,
\[
|\sL_0\widetilde u_n(s,z)|
\leq
h((t-s)/\eps^2)\,n^{1-\beta}\,\varrho_1(z)\|\varphi\|_{\bC^\alpha_{\rho_0}\bC^\beta_{\omega_0}},
\]
and
\[
|\sL_2\widetilde u_n(s,z)|
\leq
h((t-s)/\eps^2)\,n^{2-\beta}\,\varrho_1(z)\|\varphi\|_{\bC^\alpha_{\rho_0}\bC^\beta_{\omega_0}}.
\]
Therefore, using \eqref{Mom1}, we obtain $\varrho\in\sW_{d+\vartheta}$,
\begin{align*}
|\sI^{\eps,n}_{33}(t)|
&\leq
(\eps^{-1}n^{1-\beta}+n^{2-\beta})
\int_0^t h((t-s)/\eps^2)\mE\varrho_1(Z_s^\eps)\dif s\,
\|\varphi\|_{\bC^\alpha_{\rho_0}\bC^\beta_{\omega_0}}\\
&\leq
(\eps^{-1}n^{1-\beta}+n^{2-\beta})\varrho(z)
\int_0^t h((t-s)/\eps^2)\dif s\,
\|\varphi\|_{\bC^\alpha_{\rho_0}\bC^\beta_{\omega_0}}\\
&=
(\eps n^{1-\beta}+\eps^2 n^{2-\beta})\varrho(z)
\int_0^{t/\eps^2} h(r)\dif r\,
\|\varphi\|_{\bC^\alpha_{\rho_0}\bC^\beta_{\omega_0}}.
\end{align*}

\smallskip
\noindent
{\it(iii) estimate of the commutator term $\sI^{\eps,n}_{34}(t)$.}
Using \eqref{App1} together with the coefficient regularity \eqref{Red1} and \eqref{Sq1}, we have
\begin{align*}
\|(\sL_1 u(s))_n-\sL_1u_n(s)\|_{\bC^0_\rho\bC^0_\omega}
&\leq
\|(\sL_1 u(s))_n-\sL_1 u(s)\|_{\bC^0_\rho\bC^0_\omega}
+\|\sL_1(u(s)-u_n(s))\|_{\bC^0_\rho\bC^0_\omega}\\
&\lesssim
n^{-\beta}\|u(s)\|_{\bC^2_\rho\bC^\beta_\omega}
\leq
n^{-\beta}h(s)\|\varphi\|_{\bC^\alpha_{\rho_0}\bC^\beta_{\omega_0}}.
\end{align*}
Consequently,  for some $\varrho_1\in\sW_{d+\vartheta}$, we derive
\begin{align*}
|\sI^{\eps,n}_{34}(t)|
&\lesssim
n^{-\beta}\eps^{-2}\int_0^t h((t-s)/\eps^2)\mE\varrho_1(Z_s^\eps)\dif s\,
\|\varphi\|_{\bC^\alpha_{\rho_0}\bC^\beta_{\omega_0}}\\
&\leq
n^{-\beta}\varrho(z)\int_0^{t/\eps^2} h(r)\dif r\,
\|\varphi\|_{\bC^\alpha_{\rho_0}\bC^\beta_{\omega_0}}.
\end{align*}

\smallskip
Collecting the above estimates, we obtain
\[
|\sI^\eps_3(t,z)|
\leq
\Big(n^{-\beta}+\eps n^{1-\beta}+\eps^2 n^{2-\beta}\Big)\varrho(z)
\left(\int_0^\infty h(r)\dif r\right)
\|\varphi\|_{\bC^\alpha_{\rho_0}\bC^\beta_{\omega_0}}.
\]
Choosing $n=\eps^{-1}$ yields
\[
\sup_{t\geq0}|\sI^\eps_3(t,z)|
\lesssim
\eps^\beta\varrho(z)\|\varphi\|_{\bC^\alpha_{\rho_0}\bC^\beta_{\omega_0}}.
\]
Finally, combining the estimates for $\sI^\eps_1$, $\sI^\eps_2$, and $\sI^\eps_3$, we arrive at
\[
\sup_{t\geq0}\big|\mE\varphi(Z_t^\eps(z))-\bar\cT_t\bar\varphi(y)\big|
\leq
\varrho(z)\Big(\eps^\beta\|\varphi\|_{\bC^\alpha_{\rho_0}\bC^\gamma_{\omega_0}}
+\ell_0(t/\eps^2)\|\varphi-\bar\varphi\|_{\bC^0_{\rho_0}\bC^0_{\omega_0}}\Big),
\]
which is exactly \eqref{XXX}. This completes the proof.
\end{proof}
	
\section{Uniform-in-time averaging principle}

In this  section, we demonstrate how our general framework applies to the classical averaging principle.
Consider the slow-fast stochastic system
\begin{equation*}
\left\{
\begin{aligned}
\dif X^\eps_t
&=\eps^{-2} b(X^\eps_t,Y^\eps_t)\dif t+\sqrt{2}\,\eps^{-1}\sigma(X^\eps_t,Y^\eps_t)\dif W_t,\ &X^\eps_0=x,\\
\dif Y^\eps_t
&=F(X^\eps_t,Y^\eps_t)\dif t+\sqrt{2}\,G(X^\eps_t,Y^\eps_t)\dif W_t,\ &Y^\eps_0=y,
\end{aligned}
\right.
\end{equation*}
where $W$ is an $\mR^m$-valued Brownian motion. This corresponds to the general framework \eqref{sde0} with $c\equiv0$ and $H\equiv0$, so that the $\cO(\eps^{-1})$ terms vanish in  both the fast and the slow equations.

We impose the following assumptions on the coefficients.

\begin{enumerate}[{\bf (A$^\alpha_\beta$)$'$}]
\item
For some $\alpha,\beta\in(0,2]$, we have
$b,\sigma,F,G\in\bC^{\alpha}_{\rm p}\bC^{\beta}_{\rm p}$.
Moreover, the following conditions hold.

\begin{enumerate}[(a)]
\item (Fast component) Assume that one of the following holds.

\smallskip
\noindent
{\it (i) (Exponential ergodicity)} There exist $\lambda,\Lambda\in\sW_d$ such that for all $(x,y)\in\mR^d\times\mR^\vartheta$,
\[
\lambda(x)|\xi|^{2} \leq |\sigma(x,y)\xi|^2 \leq \Lambda(x) |\xi|^{2},
\quad \forall \,\xi \in \mathbb{R}^m,
\]
and for any $\kappa>0$, there exists  constants $c_\kappa,C_\kappa>0$
\[
\kappa|\sigma:\sigma|(x,y)+\langle x,b(x,y)\rangle\leq -c_\kappa|x|^2+C_\kappa.
\]

\smallskip
\noindent
{\it (ii) (Polynomial ergodicity)} There exist constants $0<\lambda<\Lambda<\infty$ such that for all $(x,y)\in\mR^d\times\mR^\vartheta$,
\[
\lambda|\xi|^{2} \leq |\sigma(x,y)\xi|^2 \leq \Lambda|\xi|^{2},
\quad \forall \,\xi \in \mathbb{R}^m,
\]
and
\begin{align*}
\lim_{|x|\to\infty}\sup_{y\in \mR^\vartheta}\langle x,b(x,y)\rangle =-\infty.
\end{align*}

\item (Slow component) Assume that one of the following holds.

\smallskip
\noindent
{\it (i) (Exponential ergodicity)} There exist $\lambda,\Lambda\in\sW_\vartheta$  such that for all $(x,y)\in\mR^d\times\mR^\vartheta$,
\[
\lambda(y)|\xi|^{2} \leq |G(x,y)\xi|^2 \leq \Lambda(y) |\xi|^{2},
\quad \forall \,\xi \in \mathbb{R}^m,
\]
and for any $\kappa>0$, there exist constants $c_\kappa,C_\kappa>0,m\geq 2$ such that
\[
\kappa|G:G|(x,y)+\langle y,F(x,y)\rangle\leq -c_\kappa|y|^2+C_\kappa(1+|x|^m).
\]

\smallskip
\noindent
{\it (ii) (Polynomial ergodicity)} There exist constants $0<\lambda<\Lambda<\infty$ such that for all $(x,y)\in\mR^d\times\mR^\vartheta$,
\[
\lambda|\xi|^{2} \leq |G(x,y)\xi|^2 \leq \Lambda|\xi|^{2},
\quad \forall \,\xi \in \mathbb{R}^m,
\]
and
\[
\lim_{|y|\to\infty}\sup_{x\in \mR^d}\langle y,F(x,y)\rangle =-\infty.
\]
\end{enumerate}
\end{enumerate}

Under condition (a) of {\bf (A$^\alpha_\beta$)$'$}, for each $y\in\mR^\vartheta$, the frozen fast equation \eqref{sde2}
admits a unique invariant probability measure $\mu_y(\dif x)$
(see Examples 1.6 and 1.7 in \cite{XXZ26}).
In this setting, the effective coefficients in (\ref{bcf}) reduce to
\[
\cF(y)=\mu_y\big(F(\cdot,y)\big),
\quad
\cG(y)=\mu_y\big((GG^*)(\cdot,y)\big).
\]
The averaged equation for the slow component is therefore
\begin{equation}\label{sde-avg-limit}
\dif \bar Y_t = \cF(\bar Y_t)\dif t + \sqrt{2}\,\cG(\bar Y_t)^{1/2}\dif \bar W_t,\qquad \bar Y_0=y,
\end{equation}
where $\bar W$ is a $\vartheta$-dimensional Brownian motion, and $\cG^{1/2}$ denotes the (symmetric) matrix square root.
As a result of Theorem~\ref{main}, we have  the following uniform-in-time averaging result.

\bt\label{aver}
Assume {\bf (A$^\alpha_\beta$)$'$}. Let $(\rho_0,\omega_0)\in\sW_d\times\sW_\vartheta$, $\alpha\in(0,1]$, $\gamma\in(\beta,2]$, and let $\varphi\in\bC^\alpha_{\rho_0}\bC^\gamma_{\omega_0}$.
Then there exist a weight  $\varrho\in\sW_{d+\vartheta}$ such that for all $\eps\in(0,1)$, $z=(x,y)$ and  all $t\geq0$,
\begin{align}\label{uit-av-joint}
\Big|\mE\big[\varphi(X_t^\eps,Y_t^\eps)(z)\big]-\mE\big[\bar\varphi(\bar Y_t(y))\big]\Big|
\leq \varrho(z)\Big(\eps^\beta\|\varphi\|_{\bC^\alpha_{\rho_0}\bC^\gamma_{\omega_0}}
+\ell_0(t/\eps^2)\|\varphi-\bar\varphi\|_{\bC^0_{\rho_0}\bC^0_{\omega_0}}\Big),
\end{align}
where $\bar\varphi(y):=\mu_y(\varphi(\cdot,y))$, $\bar Y_t$ solves the averaged equation \eqref{sde-avg-limit}, and the mixing rate $\ell_0(t)$ is given by:
\begin{itemize}
    \item $\ell_0(t)=C\e^{-\gamma t}$ for some $C,\gamma>0$ under assumption {\rm (a)} (i);
    \item $\ell_0(t)=C_m (1\wedge t^{-m})$ for any prescribed $m>2$ under assumption {\rm (a)} (ii).
\end{itemize}
\et

\begin{remark}
(i) Classical averaging principles are typically established on finite time intervals $[0,T]$ (see, e.g., \cite{KY2,GR}). Uniform-in-time estimates have only been obtained recently under strong dissipativity and regularity assumptions \cite{BDOZ,Cr-Do-Go-Ot-So}.  Theorem \ref{aver} improves these results in several directions:
\begin{enumerate}[--]
  \item  It allows for unbounded and irregular coefficients as well as test functions;
  \item  It requires only weak (polynomial) mixing properties for the dynamics;
  \item  It provides an explicit convergence rate that depends only on the regularity of the coefficients in the slow variable;
  \item It establishes convergence for the joint law of $(X_t^\eps,Y_t^\eps)$, not merely the marginal law of the slow component.
\end{enumerate}

(ii) In the general diffusion approximation setting of Theorem \ref{main}, we assume $b,\sigma\in\bC^{\alpha}_{\rm p}\bC^{1+\beta}_{\rm p}$ with $\beta\in(0,1]$  due to the need for solving the Poisson equation \eqref{pde10} and the associated regularity estimates for the corrector $\Phi$. In the averaging regime considered here, we have $H\equiv0$, so  $\Phi\equiv0$, and all estimates involving $\Phi$ become trivial. Consequently,  the regularity assumptions on the coefficients can be relaxed to $b,\sigma\in\bC^{\alpha}_{\rm p}\bC^{\beta}_{\rm p}$ with $\beta\in(0,2]$,  and the convergence rate improves accordingly.
\end{remark}

\begin{proof}
We verify the hypotheses of Theorem \ref{main}. Under {\bf (A\(^\alpha_\beta\))$'$}, the frozen fast dynamics \eqref{sde2} admits a unique invariant measure $\mu_y$, and the ergodic estimate \eqref{xy} holds with $\ell_0(t)$ as specified in Theorem \ref{aver}: exponential decay under (a) {\it (i)} and polynomial decay under (a) {\it (ii)}, see Examples 1.6 and 1.7 in \cite{XXZ26}. Thus {\bf (H\(_1\))} is satisfied.

Moreover, under condition (b), one readily checks that either
\[
\lambda(y)|\xi|^{2} \leq \langle \cG(y)\xi,\xi\rangle \leq \Lambda(y) |\xi|^{2},
\quad \forall \xi\in\mR^\vartheta,
\]
together with the coercivity bound
\[
\langle y,\cF(y)\rangle+\mathrm{tr}\,\cG(y)\leq -c_0|y|^2+c_1,
\]
or, in the uniform case,
\[
\lambda|\xi|^{2} \leq \langle \cG(y)\xi,\xi\rangle \leq \Lambda|\xi|^{2},
\quad \forall \xi\in\mR^\vartheta,
\]
and
\[
\lim_{|y|\to\infty}\langle y,\cF(y)\rangle=-\infty.
\]
Hence the averaged limit  \eqref{sde-avg-limit} satisfies {\bf (H$_2$)}.

It remains to verify {\bf (H\(_0\))}. Under (a), we have by \cite[Lemma~1]{Ve1} (see also \cite[Proposition 1]{Pa-Ve1}) that for any $m>0$,
\begin{align}\label{esX}
\sup_{\eps\in(0,1)}\sup_{t\geq 0}\mE|X_t^\eps|^m\leq C_m(1+|x|^m)<\infty.
\end{align}
Under (b) {\it (ii)}, the same bound holds directly for $Y_t^\eps$.
Under (b) {\it (i)},  applying It\^o's formula and using Young's inequality, we deduce that for any $p\geq 1$,
\begin{align*}
\frac{d}{dt}\mE|Y_t^\eps|^{2p}
&\leq \mE\Big(\big[2p(2p-1)|G:G|(X_t^\eps,Y_t^\eps)+\langle Y_t^\eps,F(X_t^\eps,Y_t^\eps)\rangle\big] |Y_t^\eps|^{2p-2}\Big)\\
&\leq \mE\Big(\big[-c_p|Y_t^\eps|^2+C_p(1+|X_t^\eps|^m)\big]\,|Y_t^\eps|^{2p-2}\Big)\\
&\leq -\tilde c_p\,\mE|Y_t^\eps|^{2p}+\tilde C_p\bigl(1+\mE|X_t^\eps|^{mp}\bigr),
\end{align*}
which in turn implies by  (\ref{esX}) that
\begin{align*}
\sup_{\eps\in(0,1)}\sup_{t\geq 0}\mE|Y_t^\eps|^{2p}\leq C_p(1+|x|^{mp}+|y|^{2p}).
\end{align*}
Thus {\bf (H$_0$)} holds.   Note that in the present setting, the vanishing corrector $\Phi\equiv0$ allows us to relax  the regularity requirements  to $b,\sigma\in\bC^{\alpha}_{\rm p}\bC^{\beta}_{\rm p}$ with $\beta\in(0,2]$, compared to the general case of Theorem \ref{main} that $b,\sigma\in\bC^{\alpha}_{\rm p}\bC^{1+\beta}_{\rm p}$ with $\beta\in(0,1]$.  Applying Theorem \ref{main} directly yields the desired estimate \eqref{uit-av-joint}.
\end{proof}

\section{Uniform S-K approximation with state-dependent matrix friction}
\label{sec:SK-generalA}

In this section, we study the small-mass limit for Langevin equations (\ref{lav}) with state-dependent, possibly non-symmetric friction matrices. Our goal is to derive a uniform-in-time convergence for the joint position-scaled velocity law as a direct application of Theorem \ref{main}. We remark that  the analysis developed below extends naturally to a broader class of stochastic Hamiltonian systems with multiple time scales, see e.g. \cite{B18,BW18}.

\subsection{The kinetic system and assumptions}

We consider the second-order stochastic system
\begin{equation}\label{eq:SK-second-order-generalA}
\eps^2 \ddot Y_t^\eps
=
-\nabla U(Y_t^\eps)-A(Y_t^\eps)\dot Y_t^\eps
+\sqrt{2}\,\sigma(Y_t^\eps)\,\dot W_t,
\qquad
(\dot Y_0^\eps,Y_0^\eps)=(v,y)\in\mR^d\times\mR^d,
\end{equation}
where
\[
U:\mR^d\to[0,\infty),
\qquad
A,\sigma:\mR^d\to\mR^d\otimes\mR^d.
\]
Introducing the scaled velocity
\[
X_t^\eps:=\eps\,\dot Y_t^\eps,
\]
we rewrite \eqref{eq:SK-second-order-generalA} as the first-order  system
\begin{equation}\label{eq:SK-first-order-generalA}
\left\{
\begin{aligned}
\dif X_t^\eps
&=
-\eps^{-2}A(Y_t^\eps)X_t^\eps\,\dif t
-\eps^{-1}\nabla U(Y_t^\eps)\,\dif t
+\eps^{-1}\sqrt{2}\,\sigma(Y_t^\eps)\,\dif W_t,\\
\dif Y_t^\eps
&=
\eps^{-1}X_t^\eps\,\dif t,
\qquad
(X_0^\eps,Y_0^\eps)=(\eps v,y).
\end{aligned}
\right.
\end{equation}
This is a particular case of the general model \eqref{sde0} with $\vartheta=d$ and
\[
G\equiv0,
\quad
F\equiv0,
\quad
H(x,y)=x,
\quad
b(x,y)=-A(y)x,
\quad
c(x,y)=-\nabla U(y).
\]
\begin{remark}
Under the Smoluchowski--Kramers scaling, the natural initial datum for \eqref{eq:SK-first-order-generalA} is the scaled state
\[
(X_0^\eps,Y_0^\eps)=(\eps v,y).
\]
Accordingly, when applying the abstract averaging theorem, the initial variable $z=(x,y)$ should always be understood as the scaled initial state, namely
\[
x=X_0^\eps=\eps v.
\]
\end{remark}

We impose the following assumptions  on the coefficients.

\smallskip
\begin{enumerate}[{\bf (A$^{\beta}_{\rm SK}$)}]
\item
Let $\beta\in(0,1]$. Assume that
$U,A,\sigma\in \bC_{\mathrm p}^{1+\beta}(\mR^d)$.
Moreover, there exist constants $\kappa_0,\kappa_1,\lambda,\Lambda>0$ such that for all $y,\xi\in\mR^d$,
\begin{equation}\label{eq:ASK-elliptic-A}
\kappa_0|\xi|^2\leq \langle A(y)\xi,\xi\rangle\leq \kappa_0^{-1}|\xi|^2,\quad
\lambda |\xi|^2\leq |\sigma(y)\xi|^2\leq \Lambda |\xi|^2,
\quad
\|\nabla A(y)\|\leq\kappa_1,
\end{equation}
and there exist constants $c_0,C_U>0$ such that
\begin{equation}\label{eq:U}
\big\langle y, A(y)^{-1}\nabla U(y)\big\rangle\geq c_0|y|^2-C_U.
\end{equation}
Furthermore, there exists  a function $h\in C^2(\mR^d;[0,\infty))$ such that for some constants $\kappa_2,\kappa_3,C_A,c_1>0$,
\begin{equation}\label{eq:ASK-h}
\kappa_2|y|^2\leq h(y)\leq \kappa_3(1+|y|^2),
\end{equation}
and, with
$\cA_h(y):=(A(y)^{-1})^*\nabla h(y)$,
\begin{equation}\label{eq:ASK-Ah}
\frac{|\cA_h(y)|}{1+|y|}+\|\nabla \cA_h(y)\|\leq C_A,
\qquad y\in\mR^d,
\end{equation}
and
\begin{equation}\label{eq:ASK-U}
\big\langle \cA_h(y),\nabla U(y)\big\rangle
\geq c_0|y|^2+c_1U(y)-C_U,
\qquad y\in\mR^d.
\end{equation}
\end{enumerate}

\smallskip
\begin{example}
Let $\hbar:\mR_+\to\mR_+$ satisfy
\[
\kappa_0\leq \hbar(s)\leq \kappa_1,
\qquad
|\hbar'(s)|\leq \kappa_2,
\qquad s\geq0,
\]
and set
\[
A(y)=\hbar(|y|^2)I_d.
\]
Define
\[
h(y):=\frac{1}{2}\int_0^{|y|^2}\hbar(s)\,\dif s.
\]
Then conditions \eqref{eq:ASK-elliptic-A} and \eqref{eq:ASK-h} hold, and we readily obtain
\[
\cA_h(y)=y.
\]
Hence, \eqref{eq:ASK-Ah} is automatically satisfied, while \eqref{eq:ASK-U} reduces to the standard coercivity condition
\[
\langle y,\nabla U(y)\rangle\geq c_0|y|^2+c_1U(y)-C_U.
\]
\end{example}

\begin{example}
Suppose   $A(y)$ satisfies \eqref{eq:ASK-elliptic-A} alongside the bound
\[
\|\nabla A(y)\|\leq \frac{C}{1+|y|}.
\]
Let $U(y)=\frac{1}{2}u(|y|^2)$ with
\[
su'(s)\geq c_0s+c_1u(s)-C_u, \qquad s>0.
\]
Then \eqref{eq:ASK-Ah}, \eqref{eq:U}, and \eqref{eq:ASK-U} hold with the choice
\[
h(y)=\frac{1}{2}|y|^2.
\]
Indeed, observing that
\[
\nabla h(y)=y, \qquad \nabla U(y)=u'(|y|^2)\,y, \qquad \cA_h(y)=(A(y)^{-1})^*y,
\]
we immediately obtain
\[
\|\nabla\cA_h(y)\|\leq\|(A(y)^{-1})^*\|+\|\nabla (A(y)^{-1})\|\,|y|\leq C_A.
\]
Moreover, applying the uniform ellipticity condition,
\[
\big\langle \cA_h(y),\nabla U(y)\big\rangle
=
u'(|y|^2)\big\langle A(y)^{-1}y,y\big\rangle
\geq \kappa_0|y|^2u'(|y|^2).
\]
Thus, \eqref{eq:U} and \eqref{eq:ASK-U} are both satisfied.
\end{example}

\subsection{Frozen fast dynamics  and homogenized limit}

For each frozen configuration $y\in\mR^d$, the frozen system (\ref{sde2}) reduces to an Ornstein--Uhlenbeck equation
\begin{equation}\label{eq:frozen-fast-generalA}
\dif X_t^y=-A(y)X_t^y\,\dif t+\sqrt{2}\,\sigma(y)\,\dif W_t,
\qquad
X_0^y=x.
\end{equation}

\begin{lemma}\label{lem:frozen-generalA}
Under {\bf (A$^\beta_{\rm SK}$)}, for each $y\in\mR^d$, equation \eqref{eq:frozen-fast-generalA} admits the explicit solution
\[
X_t^y(x)
=
\e^{-tA(y)}x+\sqrt{2}\int_0^t \e^{-(t-s)A(y)}\sigma(y)\,\dif W_s.
\]
The unique invariant probability measure is the Gaussian distribution
\[
\mu_y=\sN(0,\Sigma(y)),
\]
where $\Sigma(y)$ solves the Lyapunov equation
\begin{equation}\label{eq:Sigma-Lyapunov-generalA}
A(y)\Sigma(y)+\Sigma(y)A(y)^*=2\,\sigma(y)\sigma(y)^*.
\end{equation}
Equivalently,
\begin{equation}\label{eq:Sigma-integral-generalA}
\Sigma(y)=2\int_0^\infty \e^{-sA(y)}\sigma(y)\sigma(y)^*\e^{-sA(y)^*}\,\dif s.
\end{equation}
Moreover,  for every polynomial weight $\rho_0\in \sW_d$, there exist a polynomial weight $\rho_1\in \sW_d$ and a constant $C>0$ such that
\begin{equation}\label{eq:mixing-generalA}
\big|\mE[\varphi(X^y_t(x))]-\mu_y(\varphi)\big|
\leq
C\e^{-\kappa_0 t}\rho_1(x)\|\varphi\|_{\bC^0_{\rho_0}},
\qquad
t\ge0,\ x,y\in\mR^d,
\end{equation}
for all $\varphi\in\bC^0_{\rho_0}(\mR^d)$. In particular, the centering condition holds:
\[
\mu_y(H(\cdot,y))=\int_{\mR^d}x\,\mu_y(\dif x)=0.
\]
\end{lemma}

\begin{proof}
The  explicit formula  follows from the variation-of-constants formula. By \eqref{eq:ASK-elliptic-A}, the symmetric part of $A(y)$ is uniformly coercive, which implies
\[
\|\e^{-tA(y)}\|\leq \e^{-\kappa_0 t},
\qquad t\ge0,\ y\in\mR^d.
\]
This guarantees the absolute convergence of the integral in \eqref{eq:Sigma-integral-generalA} and the exponential stability of the frozen semigroup. The existence and uniqueness of the invariant Gaussian law are standard features of Ornstein--Uhlenbeck processes, and its covariance is completely characterized by \eqref{eq:Sigma-Lyapunov-generalA}. Estimate \eqref{eq:mixing-generalA} is a consequence of the explicit Gaussian transition formula, the exponential decay of the mean, and the uniform bounds on the covariance. The centering property holds trivially since  $\mu_y$ has zero mean.
\end{proof}

Let
\[
\sL_1:=(\sigma\sigma^*)(y):\nabla_x^2-A(y)x\cdot\nabla_x
\]
be the  generator of the frozen SDE \ref{eq:frozen-fast-generalA}.
We proceed to identify the corrector and the effective coefficients.

\begin{lemma}\label{lem:effective-generalA}
Under {\bf (A$^\beta_{\rm SK}$)}, the unique centered solution to the Poisson equation
\begin{equation*}
\sL_1\Phi(\cdot,y)=-H(\cdot,y)=-x
\end{equation*}
is given by
\begin{equation}\label{eq:corrector-generalA}
\Phi(x,y)=A(y)^{-1}x.
\end{equation}
Consequently, the effective diffusion matrix and drift are
\begin{align}
\cG(y)&=A(y)^{-1}\sigma(y)\sigma(y)^*(A(y)^{-1})^*,\label{eq:G-generalA}\\
\cF(y)&=-A(y)^{-1}\nabla U(y)+\mathfrak B(y),\label{eq:F-generalA}
\end{align}
where $\mathfrak B(y)$ is called the noise-induced term given componentwise by
\begin{equation}\label{eq:B-generalA}
\mathfrak B_i(y)
=
\sum_{j,k=1}^d \Sigma_{jk}(y)\,\partial_{y_j}(A(y)^{-1})_{ik},
\qquad i=1,\dots,d.
\end{equation}
\end{lemma}

\begin{proof}
Since $\sL_1x=-A(y)x$, the expression \eqref{eq:corrector-generalA} is easily verified. For $\xi\in\mR^d$, define
\[
\Phi_\xi(x,y):=\langle \Phi(x,y),\xi\rangle
=
\langle x,(A(y)^{-1})^*\xi\rangle.
\]
Then
\[
\nabla_x\Phi_\xi(x,y)=(A(y)^{-1})^*\xi,
\]
 and the abstract formula \eqref{bcf} for the effective covariance yields
\[
\langle \cG(y)\xi,\xi\rangle
=
\mu_y\big(|\sigma(y)^*\nabla_x\Phi_\xi(\cdot,y)|^2\big)
=
\big\langle A(y)^{-1}\sigma(y)\sigma(y)^*(A(y)^{-1})^*\xi,\xi\big\rangle,
\]
which proves \eqref{eq:G-generalA}.
For the effective drift, recall that $D=(c,H)=(-\nabla U(y),x)$. A direct computation gives
\[
\Gamma_1=D\cdot\nabla_z\Phi
=
c\cdot\nabla_x\Phi+H\cdot\nabla_y\Phi.
\]
Since $\nabla_x\Phi=A(y)^{-1}$, the first component evaluates to
\[
c\cdot\nabla_x\Phi=-A(y)^{-1}\nabla U(y).
\]
Moreover,
\[
\partial_{y_j}\Phi_i(x,y)
=
\sum_{k=1}^d\partial_{y_j}(A(y)^{-1})_{ik}\,x_k,
\]
and hence
\[
(H\cdot\nabla_y\Phi)_i
=
\sum_{j,k=1}^d x_jx_k\,\partial_{y_j}(A(y)^{-1})_{ik}(y).
\]
Taking expectation with respect to $\mu_y=\sN(0,\Sigma(y))$, we obtain
\[
\mu_y\big((H\cdot\nabla_y\Phi)_i\big)
=
\sum_{j,k=1}^d \Sigma_{jk}(y)\,\partial_{y_j}(A(y)^{-1})_{ik}.
\]
This confirms \eqref{eq:F-generalA} and \eqref{eq:B-generalA}.
\end{proof}

Consequently, the homogenized limit for the slow motion is described by the It\^o SDE
\begin{equation}\label{eq:effective-SK-generalA}
\dif \bar Y_t
=
\cF(\bar Y_t)\,\dif t
+
\sqrt{2}\,\cG(\bar Y_t)^{1/2}\,\dif \bar W_t,
\qquad
\bar Y_0=y.
\end{equation}

\begin{remark}
If the friction matrix acts as a scalar multiplier, i.e., $A(y)=h(y)I_d$, then \eqref{eq:Sigma-Lyapunov-generalA} reduces symmetrically to
\[
\Sigma(y)=\frac{\sigma(y)\sigma(y)^*}{h(y)},
\]
 and the effective coefficients simplify to
\[
\cG(y)=\frac{\sigma(y)\sigma(y)^*}{h(y)^2},\qquad
\cF(y)=-\frac{\nabla U(y)}{h(y)}-\frac{\sigma(y)\sigma(y)^*\,\nabla h(y)}{h(y)^3}.
\]
This elegant reduction perfectly recovers the classical Smoluchowski--Kramers formulas characterizing the scalar-friction regime.
\end{remark}

\subsection{Convergence of   joint  position-scaled
velocity law}

Denote $Z_t^\eps:=(X_t^\eps,Y_t^\eps)\in\mR^{2d}$.
We can now state the main result of this section.

\begin{theorem}[Uniform Smoluchowski--Kramers approximation]\label{thm:uniform-SK-generalA}
Assume {\bf (A$^\beta_{\rm SK}$)}. Let $\gamma\in(\beta,2)$, $\alpha\in(0,1]$, and let $(\rho_0,\omega_0)\in \sW_d\times \sW_d$. Then there exists a weight $\varrho\in \sW_{2d}$ such that, for every $\varphi\in \bC^\alpha_{\rho_0}\bC^\gamma_{\omega_0}$,   $\eps\in(0,1)$ and all $t\geq0$,
\begin{equation}\label{eq:uniform-SK-generalA}
\Big|\mE\big[\varphi(Z_t^\eps(z))\big]-\bar\cT_t\bar\varphi(y)\Big|
\leq
\varrho(z)\Big(
\eps^\beta\|\varphi\|_{\bC^\alpha_{\rho_0}\bC^\gamma_{\omega_0}}
+
\e^{-\kappa_0 t/\eps^2}
\|\varphi-\bar\varphi\|_{\bC^0_{\rho_0}\bC^0_{\omega_0}}
\Big),
\end{equation}
where
\[
\bar\varphi(y):=\mu_y(\varphi(\cdot,y)),
\qquad
\bar\cT_t\bar\varphi(y):=\mE\big[\bar\varphi(\bar Y_t(y))\big],
\]
$\mu_y=\sN(0,\Sigma(y))$ is given by Lemma~\ref{lem:frozen-generalA}, and $\bar Y$ solves \eqref{eq:effective-SK-generalA}.
\end{theorem}

\begin{proof}
We  verify the hypotheses of Theorem~\ref{main}.

\smallskip
\noindent{\it (i) Verification of ${\bf(H_1)}$.}
Lemma~\ref{lem:frozen-generalA} establishes the existence of the unique invariant measure $\mu_y$ with  exponential mixing rate \eqref{eq:mixing-generalA}, and confirms the centering condition
\[
\mu_y(H(\cdot,y))=0.
\]
The uniform ellipticity aspect of ${\bf(H_1)}$ is a direct consequence of \eqref{eq:ASK-elliptic-A}.

\smallskip
\noindent{\it (ii) Verification of ${\bf(H_2)}$.}
Lemma~\ref{lem:effective-generalA} identifies the corrector $\Phi$ and the effective coefficients $\cF$ and $\cG$. Since $A,\sigma\in\bC_{\mathrm p}^{1+\beta}(\mR^d)$ and $A$ is uniformly invertible by \eqref{eq:ASK-elliptic-A}, standard composition estimates imply that $A^{-1}$ and $\Sigma$ also belong to $\bC_{\mathrm p}^{1+\beta}(\mR^d)$. Consequently, $\mathfrak B$, $\cF$, and $\cG$ share this regularity.
Moreover, by \eqref{eq:ASK-elliptic-A},
\[
\kappa_0|\xi|
\leq \big|A(y)^{-1}\xi\big|
\leq \kappa_0^{-1}|\xi|,
\qquad y,\xi\in\mR^d,
\]
and combining this with \eqref{eq:G-generalA}, we obtain
\begin{equation*}
\lambda\kappa_0^2|\xi|^2
\leq
\langle \cG(y)\xi,\xi\rangle
\leq
\Lambda\kappa_0^{-2}|\xi|^2,
\qquad y,\xi\in\mR^d.
\end{equation*}
Hence, $\cG$ is uniformly elliptic.

Next, define the Lyapunov function
\[
W(y):=1+|y|^2.
\]
Let $\bar{\sL}$ be the generator of \eqref{eq:effective-SK-generalA}, namely
\[
\bar{\sL}f=\cF\cdot\nabla f+\cG:\nabla^2f.
\]
Using \eqref{eq:F-generalA} alongside \eqref{eq:U} and the fact that $\mathfrak{B}$ and $\cG$ are globally bounded (since $A^{-1}$, $\nabla A$, and $\Sigma$ are globally bounded by our assumptions), we have for some constants $c,C>0$,
\begin{align*}
\bar{\sL}W(y)&=
-2\big\langle y,A(y)^{-1}\nabla U(y)\big\rangle
+2\langle y,\mathfrak B(y)\rangle
+2\operatorname{tr}\cG(y)\\
&\leq
-2c_0|y|^2+2C_U+2|y|\|\mathfrak B(y)\|+2\operatorname{tr}\cG(y)\\
&\leq -c\,W(y)+C,
\end{align*}
which implies the exponential ergodicity of the effective equation \eqref{eq:effective-SK-generalA}.

\smallskip
\noindent{\it (iii) Uniform moment estimates for the slow variable.}
Fix $\eta\in(0,1]$ and define
\begin{equation*}
V_{\eps,\eta}(x,y)
:=
\eps^2\bigl(|x|^2+2U(y)\bigr)
+\eta\eps\langle x,\cA_h(y)\rangle
+\eta h(y)+1,
\end{equation*}
where $\cA_h(y)$ is given in \eqref{eq:ASK-Ah}.
Recall that the full generator decomposes as
\[
\sL_\eps=\eps^{-1}\sL_0+\eps^{-2}\sL_1,
\]
where
\[
\sL_0=x\cdot\nabla_y-\nabla U(y)\cdot\nabla_x,
\qquad
\sL_1=(\sigma\sigma^*)(y):\nabla_x^2-A(y)x\cdot\nabla_x.
\]
We first compute the contribution from $\sL_0$. Observe that
\[
\sL_0(\eps^2|x|^2)
=
-2\eps^2\langle \nabla U(y),x\rangle,
\qquad
\sL_0(2\eps^2U(y))
=
2\eps^2\langle x,\nabla U(y)\rangle,
\]
so these two cross-terms cancel precisely. Furthermore,
\[
\sL_0\langle x,\cA_h(y)\rangle
=
(x\otimes x):\nabla\cA_h(y)-\langle \cA_h(y),\nabla U(y)\rangle,
\]
and
\[
\sL_0(\eta h(y))=\eta\langle x,\nabla h(y)\rangle.
\]
We next compute the contribution from $\sL_1$. Noting that $\cA_h(y)=(A(y)^{-1})^*\nabla h(y)$, we have
\[
\nabla_x V_{\eps,\eta}
=
2\eps^2x+\eta\eps\cA_h(y),
\qquad
\nabla_x^2 V_{\eps,\eta}=2\eps^2I_d.
\]
This yields
\[
(\sigma\sigma^*)(y):\nabla_x^2V_{\eps,\eta}
=
2\eps^2\operatorname{tr}\bigl(\sigma(y)\sigma(y)^*\bigr),
\]
and
\begin{align*}
-A(y)x\cdot\nabla_xV_{\eps,\eta}
&=
-A(y)x\cdot\Bigl(2\eps^2x+\eta\eps(A(y)^{-1})^*\nabla h(y)\Bigr)\\
&=
-2\eps^2\langle A(y)x,x\rangle
-\eta\eps\langle A(y)x,(A(y)^{-1})^*\nabla h(y)\rangle\\
&=
-2\eps^2\langle A(y)x,x\rangle
-\eta\eps\langle x,\nabla h(y)\rangle.
\end{align*}
Therefore,
\[
\sL_1V_{\eps,\eta}(x,y)
=
2\eps^2\operatorname{tr}\bigl(\sigma(y)\sigma(y)^*\bigr)
-2\eps^2\langle A(y)x,x\rangle
-\eta\eps\langle x,\nabla h(y)\rangle.
\]
Combining both contributions yields
\begin{align}
\sL_\eps V_{\eps,\eta}(x,y)
&=
-2\langle A(y)x,x\rangle
+\eta(x\otimes x):\nabla\cA_h(y)
\notag\\
&\quad
-\eta\langle \cA_h(y),\nabla U(y)\rangle
+2\operatorname{tr}\bigl(\sigma(y)\sigma(y)^*\bigr).
\label{eq:LV-generalA}
\end{align}
Crucially, the singular mixed term of order $\eps^{-1}$ cancels out exactly.

Now, \eqref{eq:ASK-elliptic-A} and \eqref{eq:ASK-Ah} imply that
\[
\big|(x\otimes x):\nabla\cA_h(y)\big|\leq C_A|x|^2,
\]
while \eqref{eq:ASK-U} dictates
\[
-\eta\langle \cA_h(y),\nabla U(y)\rangle
\leq
-\eta c_0|y|^2-\eta c_1U(y)+\eta C_U.
\]
Substituting these bounds into \eqref{eq:LV-generalA}, we find
\[
\sL_\eps V_{\eps,\eta}(x,y)
\leq
-\bigl(2\kappa_0-\eta C_A\bigr)|x|^2
-\eta c_0|y|^2
-\eta c_1U(y)
+C.
\]
By choosing $\eta>0$ sufficiently small, we guarantee the existence of a uniform constant $c>0$ such that
\begin{equation}\label{eq:drift-generalA}
\sL_\eps V_{\eps,\eta}(x,y)
\leq
-c\bigl(|x|^2+|y|^2+U(y)\bigr)+C,
\end{equation}
with constants strictly independent of $\eps$.

By invoking Young's inequality alongside \eqref{eq:ASK-Ah} and \eqref{eq:ASK-h}, for $\eta>0$ chosen small enough, there exist constants $0<c_1\leq c_2<\infty$, independent of $\eps\in(0,1)$, ensuring the equivalence:
\begin{equation}\label{eq:Veps-equiv-generalA}
c_1\bigl(1+|y|^2+\eps^2|x|^2+\eps^2U(y)\bigr)
\leq
V_{\eps,\eta}(x,y)
\leq
c_2\bigl(1+|y|^2+\eps^2|x|^2+\eps^2U(y)\bigr).
\end{equation}
Combining \eqref{eq:drift-generalA} and \eqref{eq:Veps-equiv-generalA}, we immediately deduce that
\begin{equation*}
\sL_\eps V_{\eps,\eta}(x,y)\leq -\delta V_{\eps,\eta}(x,y)+C
\end{equation*}
for some $\delta,C>0$ independent of $\eps$.

Let $m\geq1$. Applying the diffusion operator $\sL_\eps$ to the $m$-th power yields
\[
\sL_\eps(V_{\eps,\eta}^m)
=
mV_{\eps,\eta}^{m-1}\sL_\eps V_{\eps,\eta}
+
m(m-1)V_{\eps,\eta}^{m-2}\Gamma_\eps(V_{\eps,\eta}),
\]
where the \textit{carr\'e du champ} operator evaluates to
\[
\Gamma_\eps(V_{\eps,\eta})
=
\eps^{-2}\big\langle (\sigma\sigma^*)(y)\nabla_xV_{\eps,\eta},\nabla_xV_{\eps,\eta}\big\rangle.
\]
Since
\[
\nabla_xV_{\eps,\eta}
=
2\eps^2x+\eta\eps \cA_h(y),
\]
we have
\[
\Gamma_\eps(V_{\eps,\eta})
\leq
C\eps^{-2}\bigl(\eps^4|x|^2+\eps^2|\cA_h(y)|^2\bigr)
\leq
C\bigl(\eps^2|x|^2+|y|^2+1\bigr).
\]
Exploiting \eqref{eq:ASK-Ah} and \eqref{eq:Veps-equiv-generalA}, we infer that
\[
\Gamma_\eps(V_{\eps,\eta})\leq C\,V_{\eps,\eta}.
\]
Consequently, utilizing Young's inequality allows us to absorb lower-order terms, leading to
\[
\sL_\eps(V_{\eps,\eta}^m)\leq -c_m V_{\eps,\eta}^m+C_m,
\]
for suitable constants $c_m,C_m>0$ independent of $\eps$. Dynkin's formula and Gronwall's lemma subsequently yield the uniform moment bounds
\[
\sup_{\eps\in(0,1)}\sup_{t\geq0}
\mE\big[V_{\eps,\eta}(Z_t^\eps)^m\big]
\leq
C_m\bigl(1+V_{\eps,\eta}(z)^m\bigr).
\]
In particular, this provides uniform moment bounds for the spatial slow variable:
\begin{equation}\label{eq:uniform-Y-bound}
\sup_{\eps\in(0,1)}\sup_{t\geq0}
\mE\big[|Y_t^\eps|^m\big]
\leq
\widetilde{C}_m\bigl(1+|y|^2+|x|^2+U(y)\bigr)^{m/2}.
\end{equation}

\smallskip
\noindent{\it (iv) Uniform moment estimates for the fast variable.}
To establish the uniform bounds for the fast variable $X_t^\eps$, we define the unperturbed energy functional
\[
V(x,y):=|x|^2+2U(y).
\]
Applying $\sL_\eps$ and utilizing the coercivity condition in \eqref{eq:ASK-elliptic-A}, we obtain
\begin{align*}
\sL_\eps V(x,y)
&=
\eps^{-2}\bigl(2\operatorname{tr}(\sigma(y)\sigma(y)^*)-2\langle A(y)x,x\rangle\bigr)\\
&\leq
-2\kappa_0\eps^{-2}|x|^2+C\eps^{-2}\\
&\leq
-2\kappa_0\eps^{-2}V(x,y)+\eps^{-2}\bigl(C+4\kappa_0 U(y)\bigr),
\end{align*}
where we have used the definition of $V(x,y)$ to substitute $|x|^2 = V(x,y) - 2U(y)$.

For $m\geq 1$, we apply the diffusion operator $\sL_\eps$ to $V^m$:
\begin{align*}
\sL_\eps V^m
&=
mV^{m-1}\sL_\eps V
+
m(m-1)V^{m-2}\eps^{-2}\big\langle (\sigma\sigma^*)(y)\nabla_x V,\nabla_x V\big\rangle\\
&=
mV^{m-1}\sL_\eps V
+
4m(m-1)V^{m-2}\eps^{-2}\big\langle (\sigma\sigma^*)(y)x,x\big\rangle\\
&\leq
mV^{m-1}\eps^{-2}\bigl[-2\kappa_0 V(x,y)+C+4\kappa_0 U(y)\bigr]
+
4m(m-1)\eps^{-2}\Lambda|x|^2V^{m-2}\\
&\leq
-2\kappa_0 m\eps^{-2}V^m
+
\eps^{-2}C_m V^{m-1}\bigl(1+U(y)\bigr)\\
&\leq
-c_m\eps^{-2}V^m
+
\eps^{-2}C_m'\bigl(1+U(y)^m\bigr),
\end{align*}
where we utilized $|x|^2\leq V(x,y)$ to bound the second-order diffusion term, and applied Young's inequality to absorb the mixed term $V^{m-1}U(y)$ into $V^m$, producing strictly positive constants $c_m, C_m, C_m'$ independent of $\eps$.

By It\^o's formula, taking expectations yields the differential inequality
\begin{align*}
\frac{\dif}{\dif t} \mE\big[V^m(Z^\eps_t)\big]
&=
\mE\big[\sL_\eps V^m(Z^\eps_t)\big]
\leq
-c_m\eps^{-2}\mE\big[V^m(Z^\eps_t)\big]
+
\eps^{-2}C_m'\bigl(1+\mE\big[U(Y^\eps_t)^m\big]\bigr).
\end{align*}
Since $U\in\bC_{\mathrm p}^{1+\beta}(\mR^d)$, it exhibits at most polynomial growth. By the uniform spatial moment bounds on $Y_t^\eps$ established in \eqref{eq:uniform-Y-bound}, we deduce that $\sup_{\eps\in(0,1)}\sup_{t\geq0}\mE\big[U(Y^\eps_t)^m\big]$ is bounded by some polynomial weight $\widehat{C}_K\bigl(1+|y|^2+|x|^2+U(y)\bigr)^K$ depending only on the initial state $z=(x,y)$. Substituting this uniform bound back into the differential inequality and integrating via Gronwall's lemma gives
\[
\sup_{\eps\in(0,1)}\sup_{t\geq0}
\mE\big[V(Z_t^\eps)^m\big]
\leq
V(x,y)^m + \widetilde{C}_m\bigl(1+|y|^2+|x|^2+U(y)\bigr)^K.
\]
Since the scaled initial state is $z = (x,y)$ with $x = \eps v$, the initial evaluation $V(Z_0^\eps)^m = (|x|^2 + 2U(y))^m = V(x,y)^m$ exactly matches the initial energy, providing the required uniform bounds for the fast variable.

\smallskip
\noindent{\it (v) Conclusion.}
Combining {\it (iii)} and {\it (iv)}, we find that ${\bf(H_0)}$ is satisfied.
Thus, all assumptions of Theorem~\ref{main} are   verified. By \eqref{eq:mixing-generalA}, the frozen semigroup converges to equilibrium at the exponential rate $\e^{-\kappa_0 t}$, meaning the boundary-layer term in the abstract slow-fast estimate naturally takes the form
\[
\e^{-\kappa_0 t/\eps^2}.
\]
Applying Theorem~\ref{main}, we obtain the desired estimate \eqref{eq:uniform-SK-generalA}.
\end{proof}

\subsection{Uniform-in-time thermodynamic approximations and physical implications}

The quantitative description of the limit for the joint law of $(X_t^\eps, Y_t^\eps)$ has important physical implications. Because the Smoluchowski--Kramers approximation established in Theorem~\ref{thm:uniform-SK-generalA} is \emph{uniform-in-time}, it allows us to rigorously derive global-in-time asymptotics for key thermodynamic quantities such as total energy and entropy production rate, and to formulate a formal asymptotic description of the free energy. This provides, for the first time, a mathematically sharp version of adiabatic elimination for these observables beyond finite time intervals $[0,T]$ and up to the limit $t \to \infty$.

To formalize this, recall that for any thermodynamic observable $\varphi(x,y)$ with appropriate polynomial growth, Theorem~\ref{thm:uniform-SK-generalA} yields the global error bound:
\begin{equation}\label{eq:thermo-uniform-bound}
\Big|\mE\big[\varphi(X_t^\eps, Y_t^\eps)\big]-\mE\big[\bar\varphi(\bar Y_t)\big]\Big|
\leq
\varrho(z)\Big( \eps^\beta\|\varphi\| + \e^{-\kappa_0 t/\eps^2}\|\varphi-\bar\varphi\| \Big), \qquad \forall t \geq 0,
\end{equation}
where $z=(v, y)$ is the scaled initial state, $\bar\varphi(y) = \int_{\mR^d} \varphi(x,y)\,\mu_y(\dif x)$ is the local equilibrium expectation, and $\varrho(z)$ is a polynomial weight.

\subsubsection{Total mechanical energy}

Define the scaled mechanical energy (Hamiltonian) of the system as
\begin{equation*}
\cE_\eps(t) := \frac{1}{2}|X_t^\eps|^2 + U(Y_t^\eps).
\end{equation*}
This quantity corresponds to $\eps^2$ times the physical kinetic energy plus the potential energy, which is purely mechanical and does not explicitly involve the friction matrix $A(y)$. The observable $\varphi(x,y) = \frac{1}{2}|x|^2 + U(y)$ has quadratic growth, which is perfectly controlled by the uniform moment estimates established in Section~\ref{sec:SK-generalA}. Using the fact that the frozen equilibrium $\mu_y$ is a centered Gaussian with covariance $\Sigma(y)$, the effective macroscopic energy evaluates to
\[
\bar{\cE}(y) := \int_{\mR^d} \left( \frac{1}{2}|x|^2 + U(y) \right) \mu_y(\dif x) = \frac{1}{2}\operatorname{tr}\bigl(\Sigma(y)\bigr) + U(y).
\]
Theorem~\ref{thm:uniform-SK-generalA} then yields the following explicit uniform-in-time estimate.

\begin{corollary}[Uniform energy approximation]
Under the assumptions of Theorem \ref{thm:uniform-SK-generalA}, there exists a constant $C(z)>0$ such that
\begin{equation}\label{eq:energy-uniform}
\left| \mE[\cE_\eps(t)] - \mE\big[\bar{\cE}(\bar Y_t)\big] \right| \leq C(z)\Big( \eps^\beta + \e^{-\kappa_0 t/\eps^2} \Big), \qquad \forall t \ge 0.
\end{equation}
\end{corollary}

\paragraph{{\bf Physical interpretation:}}
\begin{itemize}
    \item \textbf{Kinetic thermalization (boundary layer):} For very short times $t = \cO(\eps^2)$, the exponential term $\e^{-\kappa_0 t/\eps^2}$ dominates the error. This characterizes the rapid initial transient phase where the fast velocity $X_t^\eps$ dynamically forgets its deterministic initial kinetic energy $\frac{1}{2}\eps^2|v|^2$ and thermalizes toward the local Maxwell--Boltzmann distribution conditioned on $Y_t^\eps$.
    \item \textbf{Uniform long-time tracking:} For macroscopic time $t \gg \eps^2$, the transient boundary layer strictly vanishes. The macroscopic expected energy tracks the effective energy $\bar{\cE}(\bar Y_t)$ uniformly up to $t \to \infty$ with a steady, non-exploding accuracy of $\cO(\eps^\beta)$.
    \item \textbf{Generalized equipartition:} The effective kinetic energy $\frac{1}{2}\operatorname{tr}\bigl(\Sigma(\bar Y_t)\bigr)$ demonstrates that the local friction matrix $A(y)$ inherently regulates the macroscopic kinetic energy via the Lyapunov equation \eqref{eq:Sigma-Lyapunov-generalA}, providing a rigorous justification of the generalized equipartition theorem for state-dependent friction, see e.g. \cite{K} and \cite[Chapter 1-2]{Z}.
\end{itemize}

\subsubsection{Entropy production rate}

For the Langevin dynamics \eqref{eq:SK-second-order-generalA} at unit temperature ($T=1$), the mean rate of entropy production in the medium (heat bath) is given by the average power dissipated by the friction force \cite{ChGa,S05, S12}:
\begin{equation*}
\dot{\cS}^\eps_{\text{med}}(t) := \mE\big[ \langle A(Y_t^\eps)\dot Y_t^\eps, \dot Y_t^\eps \rangle \big].
\end{equation*}
Using the velocity scaling $\dot Y_t^\eps = \eps^{-1} X_t^\eps$, we introduce the rescaled entropy production rate $$
\tilde{\cS}^\eps_{{\text{med}}}(t) := \eps^2 \dot{\cS}^\eps_{\text{med}}(t) = \mE\big[ \langle A(Y_t^\eps)X_t^\eps, X_t^\eps \rangle \big],
$$
which rescales the diverging physical entropy production rate to an $\cO(1)$ quantity that captures the macroscopic dissipation.
As a direct result of Theorem \ref{thm:uniform-SK-generalA}, we have:

\begin{corollary}[Uniform entropy production approximation]
Under the assumptions of Theorem \ref{thm:uniform-SK-generalA},
\begin{equation}\label{eq:epr-uniform}
\left| \tilde{\cS}^\eps_{\text{med}}(t) - \mE\Big[ \operatorname{tr}\bigl((\sigma\sigma^*)(\bar Y_t)\bigr) \Big] \right| \leq \tilde{C}(z)\Big( \eps^\beta + \e^{-\kappa_0 t/\eps^2} \Big), \qquad \forall t \geq 0.
\end{equation}
\end{corollary}

\begin{proof}
Apply \eqref{eq:thermo-uniform-bound} to the quadratic observable $\varphi(x,y)=\langle A(y)x,x\rangle$, the local equilibrium expectation evaluates to
\[
\bar\varphi(y) = \int_{\mR^d} \langle A(y)x, x \rangle \,\mu_y(\dif x) = \operatorname{tr}\bigl(A(y)\Sigma(y)\bigr).
\]
Taking the trace of the Lyapunov equation \eqref{eq:Sigma-Lyapunov-generalA} and using the cyclic property of the trace ($\operatorname{tr}(\Sigma A^*) = \operatorname{tr}(A \Sigma^T) = \operatorname{tr}(A\Sigma)$ since $\Sigma$ is symmetric), we obtain the exact algebraic identity
\begin{equation*}
2\operatorname{tr}\bigl(A(y)\Sigma(y)\bigr) = 2\operatorname{tr}\bigl((\sigma\sigma^*)(y)\bigr).
\end{equation*}
Thus, $\bar\varphi(y) = \operatorname{tr}\bigl((\sigma\sigma^*)(y)\bigr)$.
\end{proof}

\paragraph{{\bf Physical interpretation:}}
\begin{itemize}
    \item The unscaled entropy production rate diverges   as $\dot{\cS}^\eps_{\text{med}}(t) \sim \cO(\eps^{-2})$  for each fixed $t>0$, reflecting the intense thermal dissipation generated by the highly oscillatory fast velocities as mass vanishes.
    \item Strikingly, \eqref{eq:epr-uniform} shows that the macroscopic rescaled dissipation \emph{sheds its dependence on the non-symmetric friction matrix $A(y)$}. After the thermalization layer $t \gg \eps^2$, the steady dissipation is universally determined solely by the noise intensity $\operatorname{tr}(\sigma\sigma^*)$, reflecting a robust energy balance in the homogenized dynamics.
\end{itemize}

\subsubsection{Commutativity of limits and steady-state thermodynamics}

The uniformity in time has a profound mathematical consequence: it guarantees the commutativity of the limits $\eps\to0$ and $t\to\infty$.
Suppose the homogenized diffusion $\bar Y_t$ defined in \eqref{eq:effective-SK-generalA} is ergodic with a unique invariant probability measure $\bar\mu$, and   the original fully coupled system $Z_t^\eps$ admits an invariant measure $\mu_\eps$. Taking the limit $t \to \infty$ in \eqref{eq:energy-uniform}, the exponential boundary layer strictly vanishes, leaving
\begin{equation*}
\left| \int_{\mR^{2d}} \cE_\eps(x,y)\,\mu_\eps(\dif x, \dif y) - \int_{\mR^d} \bar{\cE}(y)\,\bar\mu(\dif y) \right| \leq \cO(\eps^\beta).
\end{equation*}
This establishes the rigorous relation
\[
\lim_{\eps \to 0} \lim_{t \to \infty} \mE[\cE_\eps(t)] = \lim_{t \to \infty} \lim_{\eps \to 0} \mE[\cE_\eps(t)].
\]
Consequently, the steady-state thermodynamics of the fully coupled, highly singular multiscale system can be computed with strict bounds through the geometry of the reduced effective manifold. This provides an uncompromising theoretical justification for using macroscopic homogenized equations to predict long-time physical steady states.

\subsubsection{Asymptotic Gibbs--Shannon entropy and effective free energy}

The uniform-in-time  framework also provides deep insights into the structural renormalization of information. Let $\rho_\eps(t, x, y)$ denote the joint probability density of $(X_t^\eps, Y_t^\eps)$. The existence and regularity of such densities for degenerate Langevin systems have been extensively studied, we refer to \cite{CM,DM,ZZ}.
The Gibbs--Shannon information entropy is defined as $\cH_\eps(t) := -\mE\big[ \log \rho_\eps(t, X_t^\eps, Y_t^\eps) \big]$.

\paragraph{{\bf A formal free-energy discussion.}} While $-\log \rho_\eps$ is a density-dependent observable and therefore does not strictly fall within the polynomially weighted test-function class of Theorem~\ref{thm:uniform-SK-generalA}, the uniform weak convergence established above strongly suggests the following asymptotic factorization of the density away from the initial boundary layer:
\[
\rho_\eps(t,x,y) \approx \mu_y(x)\,\bar\rho(t,y),\qquad t\gg\eps^2,
\]
 where $\bar\rho(t, y)$ is the marginal density of the effective slow process $\bar Y_t$.
Under this factorization, the joint entropy formally splits into the marginal entropy of the slow variable $\cH_{\bar Y}(t) := -\mE[\log \bar\rho(t, \bar Y_t)]$ and the conditional entropy of the frozen Gaussian fast variable $\mu_y = \sN(0, \Sigma(y))$. Calculating the exact differential entropy of the multidimensional Gaussian yields the asymptotic structural limit:
\begin{equation}\label{eq:entropy-limit}
\cH_\eps(t) \approx \cH_{\bar Y}(t) + \mE\left[ \frac{d}{2}\log(2\pi \mathrm{e}) + \frac{1}{2}\log\det\bigl(\Sigma(\bar Y_t)\bigr) \right].
\end{equation}

Defining the mean free energy as $\cF_\eps(t) := \mE[\cE_\eps(t)] - \cH_\eps(t)$ (see, e.g., \cite{Ja,Z}), we combine the asymptotic expression \eqref{eq:entropy-limit} with the uniform energy approximation to find
\begin{align*}
\cF_\eps(t) &\approx \mE\left[ \frac{1}{2}\operatorname{tr}\bigl(\Sigma(\bar Y_t)\bigr) + U(\bar Y_t) \right] - \cH_{\bar Y}(t) - \mE\left[ \frac{d}{2}\log(2\pi \mathrm{e}) + \frac{1}{2}\log\det\bigl(\Sigma(\bar Y_t)\bigr) \right] \\
&= \cF_{\bar Y}(t) + \frac{1}{2}\mE\left[ \operatorname{tr}\bigl(\Sigma(\bar Y_t)\bigr) - \log\det\bigl(\Sigma(\bar Y_t)\bigr) - d \right] - \frac{d}{2}\log(2\pi),
\end{align*}
where $\cF_{\bar Y}(t) := \mE[U(\bar Y_t)] - \cH_{\bar Y}(t)$ is the standard macroscopic free energy of the reduced slow dynamics.

\vspace{2mm}
\paragraph{\bf Physical interpretation (Kullback--Leibler renormalization):}
Notice that the explicitly emergent penalty term in the free energy,
\[
\frac{1}{2}\Big(\operatorname{tr}\bigl(\Sigma(\bar Y_t)\bigr) - \log\det\bigl(\Sigma(\bar Y_t)\bigr) - d\Big),
\]
is exactly the Kullback--Leibler divergence $D_{\mathrm{KL}}\big(\sN(0,\Sigma(\bar Y_t)) \parallel \sN(0,I_d)\big) \geq 0$. This reveals a deep principle: during   adiabatic elimination, the statistical information of the fast degrees of freedom is not lost, but systematically renormalized into an emergent thermodynamic potential landscape characterizing the reduced macroscopic system globally in time.

\section{Periodic homogenization for SDEs with distributional drifts}

In this section, we present an application of our abstract uniform-in-time diffusion-approximation framework to periodic homogenization. The key novelty of our result lies in that the highly oscillatory fast drift is allowed to be \emph{distributional} (having strictly negative H\"older regularity), yet we rigorously extract an \emph{explicit} convergence rate that is \emph{uniform in time} for suitable macroscopic observables.

We first establish the geometric notation. Let $\mT^d:=\mR^d/\mZ^d$ denote the $d$-dimensional flat torus. For $x\in\mR^d$, we denote its equivalence class by
\[
[x]:=x \pmod{\mZ^d}
\quad\Longleftrightarrow\quad
[x]-x\in\mZ^d .
\]
Given a function $f:\mT^d\to\mR$, we canonically extend it to $\mR^d$ by periodicity:
\begin{equation*}
f(x):=f([x]),\qquad x\in\mR^d .
\end{equation*}
For $\alpha\in\mR$, we denote by $\bC^\alpha(\mT^d)$ the standard H\"older--Besov spaces defined via the Littlewood--Paley decomposition (see, e.g., \cite{BCD11}).

\subsection{Periodic homogenization with H\"older coefficients}\label{subsec:holder-homo}

Before tackling the singular distributional regime, we first consider the  setting where the coefficients are H\"older continuous.  This will serve as the foundation  for the more general result. Let $\varphi\in L^\infty(\mT^d)$ be a bounded, measurable, $1$-periodic observable. Consider the Cauchy problem
\begin{equation}\label{eq:ue}
\left\{
\begin{aligned}
&\partial_t u_\eps(t,x)=\sL_\eps u_\eps(t,x),\qquad (t,x)\in \mR_+\times\mR^d,\\
&u_\eps(0,x)=\varphi(x),
\end{aligned}
\right.
\end{equation}
with $\sL_\eps$ defined by
\begin{equation*}
\sL_\eps f(x)
:=
a\!\left(\frac{x}{\eps}\right):\nabla_x^2 f(x)
+\Big[\eps^{-1}b+c\Big]\!\left(\frac{x}{\eps}\right)\cdot\nabla_x f(x),
\qquad a(x) := \big(\sigma\sigma^*\big)(x).
\end{equation*}
We impose the following structural assumptions on the periodic coefficients:

\vspace{2mm}
\begin{enumerate}[{\bf(A$^\alpha$)}]
\item
Let $\alpha\in(0,1)$. Assume that $\sigma,b,c\in \bC^\alpha(\mT^d)$ are $1$-periodic in each spatial coordinate. Furthermore, there exist constants $0<\lambda\leq \Lambda<\infty$ such that
\begin{equation*}
\lambda|\xi|^2\leq |\sigma^*(x)\xi|^2\leq \Lambda|\xi|^2,
\qquad \forall x\in\mT^d,\ \forall \xi\in\mR^d.
\end{equation*}
\end{enumerate}

\vspace{2mm}
Let $(X_t^\eps)_{t\geq0}$ be the unique weak solution to the associated SDE:
\begin{equation*}
\dif X_t^\eps
=\Big[\eps^{-1}b+c\Big]\!\left(\frac{X_t^\eps}{\eps}\right)\dif t
+\sqrt2\,\sigma\!\left(\frac{X_t^\eps}{\eps}\right)\dif W_t,
\qquad X_0^\eps=x\in\mR^d.
\end{equation*}
Then $\sL_\eps$ is the generator of $X^\eps_t$, and the solution to \eqref{eq:ue} admits the classical Feynman--Kac representation:
\begin{equation*}
u_\eps(t,x)=\mE\big[\varphi(X_t^\eps(x))\big].
\end{equation*}
Consequently, studying the limiting behavior of $u_\eps$ as $\eps\to0$ is equivalent to establishing the weak convergence of the stochastic process $X_t^\eps(x)$.

\subsubsection{Cell problem and uniform  homogenization}
Consider the auxiliary unperturbed fast SDE constrained to the torus $\mT^d$:
\begin{equation}\label{SDE9}
\dif X_t=b(X_t)\,\dif t+\sqrt2\,\sigma(X_t)\,\dif W_t,\qquad X_0=x,
\end{equation}
whose generator is given by
\begin{equation*}
\sL_0 f(x):=a(x):\nabla_x^2 f(x)+b(x)\cdot\nabla_x f(x),\qquad x\in\mT^d.
\end{equation*}
Under Assumption \textbf{(A$^\alpha$)}, the diffusion $X_t$ associated with $\sL_0$ admits a unique invariant probability measure $\mu(\dif x)$ on $\mT^d$ and exhibits  exponential ergodicity: there exist constants $C,\kappa>0$ such that
\begin{equation}\label{eq:exp-erg}
\big|\mE\big[\psi(X_t(x))\big]-\mu(\psi)\big|\leq C\e^{-\kappa t}\|\psi\|_\infty,
\quad \forall t>0,\ \forall x\in\mT^d,\ \forall \psi\in L^\infty(\mT^d).
\end{equation}

Define the averaged drift $\bar b$ and the centered fast drift $\widetilde b$ by
\begin{equation*}
\bar b:=\int_{\mT^d} b(x)\,\mu(\dif x)\in\mR^d,
\qquad \widetilde b(x):=b(x)-\bar b.
\end{equation*}
Let $\Phi:\mT^d\to\mR^d$ be the unique zero-mean solution to the Poisson equation (also known as the cell problem in this setting):
\begin{equation*}
\sL_0\Phi(x)= -\widetilde b(x),
\qquad \int_{\mT^d}\Phi(x)\,\mu(\dif x)=0.
\end{equation*}
Standard Schauder estimates on the torus guarantee that $\Phi\in\bC^{2+\alpha}(\mT^d;\mR^d)$, with norms depending only on $d,\alpha,\lambda,\Lambda$ and the H\"older norms $\|a\|_{\bC^\alpha}, \|b\|_{\bC^\alpha}$.
Denote by $\nabla_x\Phi(x)$ denote the Jacobian matrix of the corrector and set
\begin{equation*}
K(x):=I_d+\nabla_x\Phi(x)\in\mR^{d\times d}.
\end{equation*}
The effective macroscopic coefficients are then given by
\begin{equation}\label{eq:eff-coeff}
\cF:=\int_{\mT^d} K(x)c(x)\,\mu(\dif x)\in\mR^d,
\qquad
\cG:=\int_{\mT^d} K(x)a(x)K(x)^*\,\mu(\dif x)\in\mR^{d\times d}.
\end{equation}
By uniform ellipticity, $\cG$ is symmetric and strictly positive definite.

\begin{remark}
The representation \eqref{eq:eff-coeff} provides an algebraically robust form of the effective diffusion matrix. Under additional smoothness assumptions, $\cG$ can be rewritten via integration by parts into several equivalent Eulerian forms (see, e.g., \cite{Bh85,Par}). For our purposes, \eqref{eq:eff-coeff} arises organically and requires no supplementary regularity.
\end{remark}


To extract the macroscopic behavior, we define the centered slow variable:
\begin{equation}\label{eq:Ye}
Y_t^\eps:=X_t^\eps-\frac{\bar b}{\eps}t.
\end{equation}
The following result shows that $Y_t^\eps$ converges uniformly in time to an effective Brownian motion with drift, with an explicit convergence rate of order $\eps$.

\begin{theorem}[Uniform-in-time periodic homogenization]\label{thm:holder-homo}
Assume {\bf (A$^\alpha$)} and let $\gamma>1$. There exists a constant $C>0$ (depending only on $d,\alpha,\gamma,\lambda,\Lambda$ and the $\bC^\alpha$-norms of $b,c,\sigma$) such that for any observable $\varphi\in \bC^\gamma(\mT^d)$, extended periodically,
\begin{equation}\label{eq:unif-homo}
\sup_{t\geq0}\sup_{x\in\mR^d}
\Big|
\mE\big[\varphi(Y_t^\eps(x))\big]
-
\mE\big[\varphi(x+\cF t+\sqrt{2\cG}W_t)\big]
\Big|
\leq
C\,\eps\|\varphi\|_{\bC^\gamma(\mT^d)}.
\end{equation}
\end{theorem}

\begin{remark}
Recall that $u_\eps(t,x)=\mE[\varphi(X_t^\eps(x))]$ solves the Cauchy problem (\ref{eq:ue}). Theorem \ref{thm:holder-homo} states that $u_\eps(t,x)$ satisfies, after shifting by the average drift $\bar b t/\eps$,
\[
\sup_{t\geq0}\sup_{x\in\mR^d} \big| u_\eps(t,x+\bar b t/\eps) - \bar u(t,x) \big| \leq C\eps \|\varphi\|_{\bC^\gamma(\mT^d)},
\]
where $\bar u$ solves the constant-coefficient  homogenized PDE
\[
\partial_t \bar u(t,x) = \cF\cdot\nabla_x \bar u(t,x) + \cG:\nabla_x^2 \bar u(t,x), \qquad \bar u(0,x)=\varphi(x).
\]
Thus, the highly oscillatory operator $\sL_\eps$ homogenizes to the constant-coefficient operator $\bar\sL = \cF\cdot\nabla_x + \cG:\nabla_x^2$, with an explicit $\cO(\eps)$ convergence rate uniform in time.
\end{remark}

\begin{proof}
Let us introduce the highly oscillatory rescaled fast variable $\widetilde X_t^\eps := \eps^{-1} X_t^\eps$. Then the pair $(\widetilde X_t^\eps, Y_t^\eps)$ exactly satisfies the coupled  system
\begin{equation*}
\left\{
\begin{aligned}
\dif \widetilde X_t^\eps
&=\eps^{-2}b(\widetilde X_t^\eps)\,\dif t+\eps^{-1}c(\widetilde X_t^\eps)\,\dif t+\eps^{-1}\sqrt2\,\sigma(\widetilde X_t^\eps)\,\dif W_t,\qquad \widetilde X_0^\eps=x/\eps,\\
\dif Y_t^\eps
&=c(\widetilde X_t^\eps)\,\dif t+\eps^{-1}\widetilde b(\widetilde X_t^\eps)\,\dif t+\sqrt2\,\sigma(\widetilde X_t^\eps)\,\dif W_t,\qquad\quad\qquad Y_0^\eps=x.
\end{aligned}
\right.
\end{equation*}
This corresponds to the general model (\ref{sde0}) with
\[
H(x,y)=\widetilde b(x),\qquad F(x,y)=c(x),\qquad G(x,y)=\sigma(x).
\]
Since all coefficients are uniformly bounded and H\"older continuous on the compact state space $\mT^d$, the hypotheses of Theorem \ref{main} are satisfied. This yields the uniform convergence estimate \eqref{eq:unif-homo} with rate $\eps$ (since $\beta=1$ in this  setting), the  effective drift
\[
\cF=\int_{\mT^d} \big(I_d+\nabla_x\Phi(x)\big)c(x) \,\mu(\dif x) = \int_{\mT^d} K(x)c(x)\,\mu(\dif x),
\]
and the  effective covariance matrix
\[
\cG=\int_{\mT^d} \left( a+a(\nabla_x\Phi)^*+(\nabla_x\Phi)a+\frac{1}{2}\big(\widetilde b\,\Phi^* + \Phi\,\widetilde b^*\big) \right) \mu(\dif x).
\]
We now show that this expression reduces to  \eqref{eq:eff-coeff}. Applying the generator $\sL_0$ to  $\Phi \Phi^*$ and using $\sL_0\Phi=-\widetilde b$ yields
\begin{align*}
\sL_0(\Phi\Phi^*)&=(\sL_0\Phi)\Phi^*+\Phi(\sL_0\Phi)^*+2(\nabla_x\Phi) a (\nabla_x\Phi)^*\\
&= -\widetilde b \,\Phi^* - \Phi\,\widetilde b^* + 2 (\nabla_x\Phi) a (\nabla_x\Phi)^*.
\end{align*}
Integrating with respect to the invariant measure $\mu$ and noting that $\int_{\mT^d} \sL_0 f \dif \mu = 0$ for any regular function $f$, we extract the exact cross-term relation:
\[
\frac{1}{2}\int_{\mT^d} \big(\widetilde b\,\Phi^* + \Phi\,\widetilde b^*\big) \,\mu(\dif x) = \int_{\mT^d} (\nabla_x\Phi) a (\nabla_x\Phi)^* \,\mu(\dif x),
\]
which simplifies $\cG$ to the desired form \eqref{eq:eff-coeff}.
\end{proof}

\subsubsection{A functional CLT}

Theorem \ref{thm:holder-homo} provides uniform-in-time estimates for one-dimensional distributions, but it does not capture the full path structure. We next strengthen this to weak convergence in the Skorokhod space $C(\mR_+;\mR^d)$, thereby establishing a functional central limit theorem for the centered process $Y_t^\eps$.

Let $\mE^0$ denote the expectation with respect to the law of the unperturbed diffusion $X_t$ defined in \eqref{SDE9}. We first establish a standard 4-point correlation bound for the corresponding  Markov semigroup $\cT_t$.

\begin{lemma}[Four-point correlation bound]\label{lem:mixing}
Let $f\in L^\infty(\mT^d)$ be centered such that $\mu(f)=0$. For any ordered times $0 \leq r_1 \leq r_2 \leq r_3 \leq r_4$, the unperturbed diffusion $X$ satisfies
\begin{equation*}
\big|\mE^0\big[f(X_{r_1}) f(X_{r_2}) f(X_{r_3}) f(X_{r_4})\big]\big|
\leq C \|f\|_\infty^4\,\e^{-\kappa(r_2-r_1)}\e^{-\kappa(r_4-r_3)},
\end{equation*}
where $\kappa>0$ is the spectral gap constant from \eqref{eq:exp-erg}.
\end{lemma}

\begin{proof}
By the Markov property, conditioning on $\cF_{r_2}$ yields
\[
\mE^0\big[f(X_{r_3}) f(X_{r_4}) \mid \cF_{r_2}\big] = \cT_{r_3-r_2}\big(f \, \cT_{r_4-r_3}f\big)(X_{r_2}) =: \Psi(X_{r_2}).
\]
Because $\mu(f)=0$, the ergodicity bound \eqref{eq:exp-erg} implies $\|\cT_{r_4-r_3}f\|_\infty \leq C\e^{-\kappa(r_4-r_3)}\|f\|_\infty$. Since $\cT_{r_3-r_2}$ acts as a contraction in $L^\infty$, we obtain the estimate
\[
\|\Psi\|_\infty \leq \big\|f \, \cT_{r_4-r_3}f\big\|_\infty \leq C\e^{-\kappa(r_4-r_3)}\|f\|_\infty^2 .
\]
We now evaluate the full expectation. Since $X$ is strictly stationary under the invariant measure $\mu$, we have
\begin{align*}
\mE^0\big[f(X_{r_1}) f(X_{r_2}) \Psi(X_{r_2})\big]
&= \int_{\mT^d} f(x) \cT_{r_2-r_1}(f\Psi)(x)\,\mu(\dif x) \\
&= \int_{\mT^d} f(x) \Big[ \cT_{r_2-r_1}(f\Psi)(x) - \mu(f\Psi) \Big] \,\mu(\dif x) + \mu(f)\mu(f\Psi).
\end{align*}
Because $\mu(f) = 0$, the second term identically vanishes. Applying \eqref{eq:exp-erg} to the centered function $f\Psi - \mu(f\Psi)$, the remaining term is structurally bounded by
\[
\|f\|_\infty \big\|\cT_{r_2-r_1}\big(f\Psi - \mu(f\Psi)\big)\big\|_\infty \leq 2C \|f\|_\infty \e^{-\kappa(r_2-r_1)} \|f\Psi\|_\infty.
\]
Using the triangle inequality bound $\|f\Psi\|_\infty \leq \|f\|_\infty \|\Psi\|_\infty \leq C \|f\|_\infty^3 \e^{-\kappa(r_4-r_3)}$, we immediately conclude:
\[
\big|\mE^0\big[f(X_{r_1}) f(X_{r_2}) f(X_{r_3}) f(X_{r_4})\big]\big| \leq C \|f\|_\infty^4 \e^{-\kappa(r_2-r_1)}\e^{-\kappa(r_4-r_3)},
\]
completing the proof.
\end{proof}

The following result strengthens Theorem \ref{thm:holder-homo} by establishing weak convergence of the entire path of $Y_t^\eps$, which is essential for applications that depend on the full trajectory, such as the analysis of hitting times, occupation measures, and other path-dependent functionals, see e.g. \cite{CS}.

\begin{theorem}[Functional CLT]
Assume {\bf (A$^\alpha$)}. Then the centered process $(Y_t^\eps)_{t\geq0}$ defined in \eqref{eq:Ye} converges weakly in the Skorokhod space $C(\mR_+;\mR^d)$ to the effective Brownian motion with drift $(x+\cF t+\sqrt{2\cG}W_t)_{t\geq0}$.
\end{theorem}

\begin{proof}
By the uniform convergence \eqref{eq:unif-homo} and standard applications of the Markov property, the finite-dimensional distributions of $(Y_t^\eps)_{t\geq0}$ converge to those of the limiting macroscopic process. It therefore remains to prove tightness. Writing
\[
Y_t^\eps-Y_s^\eps
=
\int_s^t c(\widetilde X_r^\eps)\,\dif r
+
\eps^{-1}\int_s^t \widetilde b(\widetilde X_r^\eps)\,\dif r
+
\sqrt2\int_s^t \sigma(\widetilde X_r^\eps)\,\dif W_r,
\]
we see that the regular drift term and the martingale term satisfy the required moment bounds immediately because $c$ and $\sigma$ are bounded. Thus the only nontrivial contribution is the singular centered-drift term. By Kolmogorov's continuity criterion, it is enough to show that for any horizon $T>0$ and all $0\leq s\leq t\leq T$,
\begin{equation}\label{eq:kolmogorov}
\sup_{\eps\in(0,1)}
\mE^\eps\left[\Big|\eps^{-1}\int_s^t \widetilde b(\widetilde X_r^\eps)\,\dif r\Big|^3\right]
\leq C_T\,|t-s|^{3/2}.
\end{equation}
Consider the macroscopically accelerated diffusion and scaled Brownian motion:
\begin{equation*}
\widehat X_t^\eps:=\widetilde X_{\eps^2 t}^\eps,
\qquad
\widehat W_t:=\eps^{-1}W_{\eps^2 t}.
\end{equation*}
The process $\widehat X^\eps$ solves the $\mathcal{O}(1)$-volatility SDE
\begin{equation}\label{eq:Xtilde-sde}
\dif \widehat X_t^\eps = \big(b+\eps c\big)(\widehat X_t^\eps)\,\dif t + \sqrt2\,\sigma(\widehat X_t^\eps)\,\dif \widehat W_t.
\end{equation}
Using the temporal change of variables $r = u \eps^2$, the target integral maps to:
\[
\eps^{-1}\int_s^t \widetilde b(\widetilde X_r^\eps)\,\dif r
=\eps\int_{s/\eps^2}^{t/\eps^2}\widetilde b(\widehat X_u^\eps)\,\dif u.
\]
Thus, verifying the tightness condition \eqref{eq:kolmogorov} is equivalent to showing
\begin{equation*}
\mE^\eps\left[\Big|\int_{S_0}^{S_1}\widetilde b(\widehat X_u^\eps)\,\dif u\Big|^3\right]
\leq C_T\,\frac{|t-s|^{3/2}}{\eps^3},
\end{equation*}
where $S_0 = s/\eps^2$ and $S_1 = t/\eps^2$. Let $X$ denote the  unperturbed diffusion \eqref{SDE9},  with expectation operator $\mE^0$. By Girsanov's theorem, the law of \eqref{eq:Xtilde-sde} is absolutely continuous with respect to the law of $X$ on $[0, S_1]$ via the Radon--Nikodym density martingale
\[
\cE_{S_1}^\eps=\exp\left\{\frac{\eps}{\sqrt2}\int^{S_1}_0 \big\langle (\sigma^{-1}c)(\widehat X^\eps_s), \dif \widehat W_s \big\rangle - \frac{\eps^2}{4}\int^{S_1}_0\big|(\sigma^{-1}c)(\widehat X^\eps_s)\big|^2\dif s\right\}.
\]
Since $\sigma^{-1}$ (by uniform ellipticity) and $c$ are uniformly bounded, Novikov's criterion holds uniformly  in $\eps$, ensuring that for every fixed $T>0$, $\sup_{\eps\in(0,1)}\mE^0[(\cE_{S_1}^\eps)^4] \leq C_T$. Applying H\"older's inequality over the measure change yields
\begin{align*}
\mE^\eps\left[\Big|\int_{S_0}^{S_1}\widetilde b(\widehat X_u^\eps)\,\dif u\Big|^3\right]
&\leq \left(\mE^0\big[(\cE_{S_1}^\eps)^4\big]\right)^{1/4} \left(\mE^0\left[\Big|\int_{S_0}^{S_1}\widetilde b(X_u)\,\dif u\Big|^4\right]\right)^{3/4} \\
&\leq C_T \left(\mE^0\left[\Big|\int_{S_0}^{S_1}\widetilde b(X_u)\,\dif u\Big|^4\right]\right)^{3/4}.
\end{align*}
Fix a coordinate $i\in\{1,\dots,d\}$ and set $f=\widetilde b_i$. Expanding the fourth moment algebraically yields:
\[
\mE^0\left[\Big|\int_{S_0}^{S_1} f(X_u)\,\dif u\Big|^4\right]
= 4! \int_{S_0\leq r_1\leq \dots \leq r_4\leq S_1}
\mE^0\big[ f(X_{r_1}) \dots f(X_{r_4}) \big] \dif r_1\cdots\dif r_4.
\]
Invoking Lemma~\ref{lem:mixing}, the  integrand decays exponentially in the time gaps. Integrating this bound over the ordered simplex yields an upper bound of order $\cO((S_1-S_0)^2)$. Thus,
\[
\mE^0\left[\Big|\int_{S_0}^{S_1}\widetilde b(X_u)\,\dif u\Big|^4\right]
\leq C\,(S_1-S_0)^2 = C\left(\frac{t-s}{\eps^2}\right)^2 = C\,\frac{|t-s|^2}{\eps^4}.
\]
Substituting this kinematic bound directly back into the H\"older reduction, we obtain
\[
\mE^\eps\left[\Big|\int_{S_0}^{S_1}\widetilde b(\widehat X_u^\eps)\,\dif u\Big|^3\right]
\leq C_T \left( \frac{|t-s|^2}{\eps^4} \right)^{3/4} = C_T\,\frac{|t-s|^{3/2}}{\eps^3}.
\]
This rigorously confirms the tightness condition \eqref{eq:kolmogorov}.
\end{proof}

\subsection{Periodic homogenization with distributional fast drift}

We now extend the analysis to  the highly singular regime where the fast drift is \emph{distributional}. For clarity of exposition, we take the diffusion coefficient to be the identity; the extension to uniformly elliptic $\bC^\beta$ diffusions follows the same geometric transformation strategy.

We focus primarily on the deeply singular regime $\alpha \in (-1, -1/2)$. Consider the  operator
\begin{equation*}
\sL_\eps f(x)=\Delta_x f(x)+\Big[\eps^{-1}b+c\Big]\!\left(\frac{x}{\eps}\right)\cdot\nabla_x f(x),
\end{equation*}
where $b$ is a distribution and $c$ is classically regular. We impose the following conditions:

\vspace{2mm}
\begin{enumerate}[{\bf(A$^\alpha_\beta$)$''$}]
\item
Let $\alpha\in(-1,-1/2)$ and $\beta\in(0,1)$. Assume that $b\in\bC^\alpha(\mT^d;\mR^d)$ is  divergence-free ($\nabla_x \cdot b = 0$), and $c\in\bC^\beta(\mT^d;\mR^d)$ is $1$-periodic.
\end{enumerate}

\begin{remark}
The restriction $\alpha \in (-1,-1/2)$ places the system in the deeply singular regime. Here, the product $b\cdot \nabla_x \mathbf{u}$  possesses a negative sum of regularities ($2\alpha+1-\theta \leq 0$) and is classically ill-defined via standard Bony paraproducts. One must   exploit the divergence-free condition $\nabla_x \cdot b = 0$ to regularize the transport via commutator estimates (see, e.g., \cite{HZ25}). This delicate  cancellation justifies the fractional heat kernel bounds used below.
\end{remark}

Let $(X_t^\eps)_{t\geq0}$ be the diffusion process solving
\begin{equation}\label{eq:Xeps-dist}
\dif X_t^\eps=\Big[\eps^{-1}b+c\Big]\!\left(\frac{X_t^\eps}{\eps}\right)\dif t+\sqrt2\,\dif W_t,
\qquad X_0^\eps=x\in\mR^d.
\end{equation}
The well-posedness of this singular SDE under \textbf{(A$^\alpha_\beta$)$''$} is established via a periodic Zvonkin transformation.

\subsubsection{Periodic Zvonkin transform}

To eliminate the distributional singularity of $\sL_\eps$, we construct a steady-state Zvonkin-type spatial corrector. Consider the decoupled elliptic PDE system on $\mR^d$:
\begin{equation}\label{eq:Zvonkin-PDE}
\Delta_x u_i - \lambda u_i + b \cdot \nabla_x u_i = f_i, \qquad i=1,\dots,d,
\end{equation}
where $f=(f_1,\dots,f_d)\in\bC^\alpha(\mT^d;\mR^d)$. In the application below we take $f=-b$, which is precisely the choice that cancels the singular drift after applying It\^o's formula.

\begin{theorem}
Let $\alpha\in(-1,-1/2)$ and $\theta\in(0,1+\alpha)$. For any $f\in\bC^\alpha(\mT^d;\mR^d)$, there exists a threshold $\lambda_0=\lambda_0(d,\alpha,\theta,\|b\|_{\bC^\alpha})\geq 1$ such that for all $\lambda\geq\lambda_0$, the elliptic equation \eqref{eq:Zvonkin-PDE} admits a unique weak solution $\mathbf{u}$ satisfying
\begin{equation}\label{Da4}
\|\mathbf{u}\|_{\bC^{2+\alpha-\theta}}
\leq C \lambda^{-\theta/2}\|f\|_{\bC^\alpha},
\end{equation}
where the constant $C$ is independent of $\lambda$.
\end{theorem}
\begin{proof}
It suffices to establish the a priori estimate \eqref{Da4}. By Duhamel's principle, the solution admits the resolvent representation:
\[
\mathbf{u}(x)=\int^\infty_0\e^{-\lambda t}P_t\big(b\cdot \nabla_x \mathbf{u}+f\big)(x)\,\dif t,
\]
where  $P_t:=\e^{t\Delta}$ is the standard Gaussian heat semigroup.
Because $\alpha < -1/2$, the  product $b \cdot \nabla \mathbf{u}$ has a negative sum of regularities and is classically undefined. However, exploiting the divergence-free condition $\nabla \cdot b = 0$, the singular product is  regularized via commutators. The  fractional heat kernel estimates for singular divergence-free convective terms established in \cite[(2.16)]{HZ25} give
\[
\big\|P_t(b\cdot \nabla_x \mathbf{u})\big\|_{\bC^{2+\alpha-\theta}} \leq C t^{-1+\theta/2}\|b\|_{\bC^\alpha}\|\mathbf{u}\|_{\bC^{2+\alpha-\theta}},\qquad t>0.
\]
Integrating against the resolvent yields
\begin{align*}
\|\mathbf{u}\|_{\bC^{2+\alpha-\theta}}
&\leq \int^\infty_0 \e^{-\lambda t} \Big( \big\|P_t(b\cdot \nabla_x \mathbf{u})\big\|_{\bC^{2+\alpha-\theta}} + \big\|P_t f\big\|_{\bC^{2+\alpha-\theta}} \Big) \,\dif t\\
&\leq C \int^\infty_0 \e^{-\lambda t} t^{-1+\theta/2} \Big( \|b\|_{\bC^\alpha}\|\mathbf{u}\|_{\bC^{2+\alpha-\theta}} + \|f\|_{\bC^\alpha} \Big) \,\dif t\\
&\leq C \lambda^{-\theta/2} \Big( \|b\|_{\bC^\alpha}\|\mathbf{u}\|_{\bC^{2+\alpha-\theta}} + \|f\|_{\bC^\alpha} \Big),
\end{align*}
where we used $\int_0^\infty \e^{-\lambda t} t^{-1+\theta/2} \dif t \lesssim \lambda^{-\theta/2}$. Choosing $\lambda_0$ sufficiently large so that $C\lambda^{-\theta/2}\|b\|_{\bC^\alpha}<1/2$, the term involving $\|\mathbf{u}\|_{\bC^{2+\alpha-\theta}}$ can   be   absorbed into the left-hand side,  yielding the global contraction bound \eqref{Da4}.
\end{proof}

Let $\mathbf{u}=(u_1,\dots, u_d)$ be the unique solution to \eqref{eq:Zvonkin-PDE} with $f=-b$ and with $\lambda$ chosen large enough (using \eqref{Da4}) so that $\|\mathbf{u}\|_{\bC^{2+\alpha-\theta}}\leq1/2$.
Since $b$ is periodic, $\mathbf{u}$  inherits this periodicity. We construct the periodic Zvonkin map by
\begin{equation}\label{eq:Phi}
\Phi(x):=x+\mathbf{u}(x),\qquad x\in\mR^d.
\end{equation}
Since $2+\alpha-\theta > 1$, $\Phi(x)$ generates a global $C^1$-diffeomorphism and satisfies
\begin{equation*}
\frac{1}{2}|x-x'|\leq |\Phi(x)-\Phi(x')|\leq \frac{3}{2}|x-x'|.
\end{equation*}

\subsubsection{Reduction to the H\"older regime}

To  bypass the distributional singularity, we introduce the macroscopically scaled process
\begin{equation*}
\widehat X_t^\eps:=\eps^{-1}X_{\eps^2 t}^\eps,\qquad \widehat W_t:=\eps^{-1}W_{\eps^2 t}.
\end{equation*}
Under this scaling, \eqref{eq:Xeps-dist} becomes
\begin{equation*}
\dif \widehat X_t^\eps =\big(b+\eps c\big)(\widehat X_t^\eps)\,\dif t+\sqrt2\,\dif \widehat W_t.
\end{equation*}
Let $\widehat Y_t^\eps:=\Phi(\widehat X_t^\eps)$ be the spatially transformed  trajectory. Applying It\^o's formula and utilizing the  cancellation identity $\Delta_x \mathbf{u} + b \cdot \nabla_x \mathbf{u} = \lambda \mathbf{u} - b$, we find that $\widehat Y^\eps$ solves the regularized SDE (see \cite[Lemma 4.5]{HZ25}):
\begin{equation}\label{eq:regularized-sde}
\dif \widehat Y_t^\eps
=
\big(\widehat b+\eps\widehat c\big)(\widehat Y_t^\eps)\,\dif t
+\sqrt2\,\widehat \sigma(\widehat Y_t^\eps)\,\dif \widehat W_t,
\end{equation}
where
\[
\widehat b(y):=\lambda \mathbf{u}\big(\Phi^{-1}(y)\big),\quad \widehat c(y):=\big((I_d+\nabla \mathbf{u})c\big)\big(\Phi^{-1}(y)\big),
\quad \widehat \sigma(y):=\big(I_d+\nabla\mathbf{u}\big)\big(\Phi^{-1}(y)\big).
\]
Crucially, these  coefficients possess strictly positive H\"older regularity:
\[
\widehat b\in \bC^{2+\alpha-\theta}(\mT^d),
\quad
\widehat c\in \bC^{\beta\wedge(1+\alpha-\theta)}(\mT^d),
\quad
\widehat \sigma\in\bC^{1+\alpha-\theta}(\mT^d).
\]
Set $\widehat\alpha := \beta\wedge(1+\alpha-\theta)$. Since $\theta < 1+\alpha$, we have $\widehat\alpha > 0$. Thus, the highly singular problem is transformed into the regular H\"older framework \textbf{(A$^{\widehat\alpha}$)} of Subsection~\ref{subsec:holder-homo}.

\subsubsection{Homogenization}
Let $\widehat \mu$ be the unique invariant probability measure of the transformed fast diffusion governed by $\widehat b$ and $\widehat \sigma$, and define its effective macroscopic drift
\[
\bar b:=\int_{\mT^d}\widehat b(x)\,\widehat\mu(\dif x).
\]
Let $\cF$ and $\cG$ be the effective coefficients defined in \eqref{eq:eff-coeff}, computed from the transformed regularized dynamics. By mapping our uniform bounds back through the bounded corrector geometry, we prove the main distributional homogenization theorem.

\begin{theorem}[Periodic homogenization with distributional drift]
Assume {\bf (A$^\alpha_\beta$)$''$} and let $\gamma>1$.
There exists a constant $C>0$  (depending only on parameters and defined norms) such that for all $\eps\in(0,1)$ and any periodic observable $\varphi\in\bC^\gamma(\mT^d)$,
\begin{equation}\label{eq:dist-unif}
\sup_{t\geq0}\sup_{x\in\mR^d}
\Big|
\mE\big[\varphi(X_t^\eps(x)-\bar b\,t/\eps)\big]
-
\mE\big[\varphi(x+\cF t+\sqrt{2\cG}W_t)\big]
\Big|
\leq C\,\eps\|\varphi\|_{\bC^\gamma(\mT^d)}.
\end{equation}
\end{theorem}

\begin{proof}
To satisfy the $1$-periodic observable requirement of the uniform homogenization framework, we consider the macroscopic process directly, rather than dynamically scaling the observable (which creates an $\eps^{-1}$-periodic function and  violates Theorem~\ref{thm:holder-homo}). Define the macroscopic Zvonkin process:
\[
Z_t^\eps := \eps \widehat Y^\eps_{t/\eps^2}.
\]
By applying the inverse It\^o scaling to \eqref{eq:regularized-sde}, and noting that  $W^\eps_t := \eps \widehat W_{t/\eps^2}$ has the same law as $W_t$, we find that $Z_t^\eps$  solves the standard $\eps$-periodic multiscale SDE:
\[
\dif Z_t^\eps = \Big[\eps^{-1}\widehat b + \widehat c\Big]\!\left(\frac{Z_t^\eps}{\eps}\right)\dif t + \sqrt{2}\widehat\sigma\!\left(\frac{Z_t^\eps}{\eps}\right)\dif W_t.
\]
Since $(\widehat b, \widehat c, \widehat \sigma)$  possess positive H\"older regularity, Theorem~\ref{thm:holder-homo} applies to the centered macroscopic process $\bar Z_t^\eps := Z_t^\eps - \bar b t/\eps$. Evaluating the original fixed $1$-periodic observable $\varphi$ at the initial condition $Z_0^\eps = \eps\Phi(x/\eps) = x + \eps\mathbf{u}(x/\eps)$, we  obtain:
\begin{equation}\label{eq:Z-limit}
\sup_{t\geq0}\sup_{x\in\mR^d} \Big| \mE\big[\varphi(\bar Z_t^\eps)\big] - \mE\big[\varphi(Z_0^\eps+\cF t+\sqrt{2\cG}W_t)\big] \Big| \leq C\eps \|\varphi\|_{\bC^\gamma(\mT^d)}.
\end{equation}

Finally, we  reconcile $\bar Z_t^\eps$ with the original state $X_t^\eps$. Unrolling the macroscopic scaling and using the Zvonkin diffeomorphism \eqref{eq:Phi}, we have
\[
Z_t^\eps = \eps \Phi\big(\widehat X^\eps_{t/\eps^2}\big) = \eps \Phi\big(\eps^{-1}X_t^\eps\big) = X_t^\eps + \eps \mathbf{u}\big(\eps^{-1}X_t^\eps\big).
\]
Subtracting the secular drift yields
\[
\bar Z_t^\eps = \Big(X_t^\eps - \frac{\bar b}{\eps}t\Big) + \eps \mathbf{u}\big(\eps^{-1}X_t^\eps\big).
\]
Since   $\varphi$ is globally Lipschitz (as $\gamma>1$), the geometric error is bounded by
\[
\Big|\varphi(\bar Z_t^\eps) - \varphi\Big(X_t^\eps - \frac{\bar b}{\eps}t\Big)\Big|
\leq \|\nabla\varphi\|_\infty \eps \|\mathbf{u}\|_\infty
\leq C\eps \|\varphi\|_{\bC^\gamma(\mT^d)}.
\]
A similar Lipschitz estimate replaces the shifted initial condition $Z_0^\eps = x + \eps \mathbf{u}(x/\eps)$ by $x$ in  the Brownian expectation on the right-hand side of \eqref{eq:Z-limit}. Taking expectations and applying the triangle inequality against the uniform approximation bound \eqref{eq:Z-limit}   yields \eqref{eq:dist-unif}.
\end{proof}

\end{document}